\documentclass[aos]{imsart}

\RequirePackage[OT1]{fontenc}
\RequirePackage{amsthm,amsmath}
\RequirePackage[numbers]{natbib}
\RequirePackage[colorlinks,citecolor=blue,urlcolor=blue]{hyperref}

\usepackage{amssymb,amsmath,amsthm}
\usepackage{graphicx}
\usepackage{subfigure}
\usepackage{epstopdf}

\usepackage{algpseudocode}
\usepackage[section]{algorithm}
\usepackage{algorithmicx}

\startlocaldefs

\newtheorem{lemma}{Lemma}
\newtheorem{thm}{Theorem}
\newtheorem{cor}{Corollary}
\newtheorem{defn}{Definition}
\newtheorem{remark}{Remark}
\newtheorem{ex}{Example}
\newtheorem{pro}{Proposition}

\newtheorem{manualcorinner}{Corollary}
\newenvironment{manualcor}[1]{%
  \renewcommand\themanualcorinner{#1}%
  \manualcorinner
}{\endmanualcorinner}

\newcommand{\EP}{\,\mathbb{P}}
\newcommand{\EE}{\,\mathbb{E}}
\newcommand{\breg}{{\mathbf{\Delta}}}
\newcommand{\Breg}{{\mathbf{D}}}

\newcommand{\bm}{\boldsymbol}
\newcommand{\bsb}{\boldsymbol}
\newcommand{\rd}{\,\mathrm{d}}
\newcommand{\rD}{\,\mathrm{D}}

\newcommand{\sgn}{\mathrm{sgn}}
\newcommand{\avg}{\mathrm{avg}}
\newcommand{\back}{\oset\smallsetminus}
\newcommand{\sym}{\bar}

\makeatletter
\newcommand{\oset}[3][0ex]{%
  \mathrel{\mathop{#3}\limits^{
    \vbox to#1{\kern-2\ex@
    \hbox{$\scriptstyle#2$}\vss}}}}
\makeatother

\endlocaldefs

\begin{document}

\begin{frontmatter}
\title{Analysis of Generalized Bregman Surrogate Algorithms for Nonsmooth Nonconvex Statistical Learning}
\runtitle{Analysis of Bregman-surrogate Algorithms}

\begin{aug}
\author{\fnms{Yiyuan} \snm{She}},
\author{\fnms{Zhifeng} \snm{Wang}} \and
\author{\fnms{Jiuwu} \snm{Jin}}

\runauthor{Y. She et al.}


\address{Department of Statistics,
Florida State University
}

\end{aug}

\begin{abstract}
  Modern statistical applications often involve minimizing an  objective function  that may be nonsmooth and/or nonconvex.   This paper focuses on  a broad Bregman-surrogate   algorithm framework  including the   local linear approximation, mirror descent, iterative thresholding,  DC programming and many others
 as particular instances. The   recharacterization via generalized Bregman functions enables us to construct suitable error measures and  establish global convergence rates   for    nonconvex and nonsmooth  objectives in possibly high dimensions. For sparse learning problems with a composite objective, under some regularity conditions, the obtained estimators  as the surrogate's fixed points, though not  necessarily local minimizers,  enjoy provable statistical guarantees, and  the sequence of iterates can be shown to approach the statistical truth within the desired accuracy geometrically fast. The paper also studies how to  design adaptive momentum based accelerations without assuming convexity or smoothness by  carefully controlling stepsize and relaxation  parameters. 
\end{abstract}



\end{frontmatter}

\section{Introduction}

Many statistical learning problems can be formulated as minimizing a certain objective function. In shrinkage estimation, the objective  can often be represented as the sum of a loss function and a penalty function, neither of which is necessarily smooth or convex. For example, when the number of variables is much larger than the number of observations ($p \gg n$), sparsity-inducing penalties come into play and result in nondifferentiability. Furthermore, many popular penalties are nonconvex \citep{Frank1993Bridge, FanLi2001SCAD, Zhang2010MCP}, making the computation and analysis more challenging.
Although in low dimensions there are ways to tackle nonsmooth nonconvex optimization, statisticians often prefer easy-to-implement algorithms that scale well in big data applications. Therefore, first-order methods, gradient-descent type algorithms in particular, have recently attracted a great deal of attention due to their lower complexity per iteration and better numerical stability than Newton-type algorithms.

In this work, we study a class of algorithms   in a \textit{Bregman surrogate} framework.
The idea is that instead of solving the original problem $\min_{\bm\beta} f(\bm\beta)$, one constructs a surrogate function
\begin{equation}\label{g-function}
g(\bm{\beta};\bm{\beta}^-) = f(\bm{\beta}) + \bm\Delta_\psi(\bm\beta,\bm\beta^-),
\end{equation}
and generates a sequence of iterates according to
\begin{equation}
\bm\beta^{(t+1)} \in \mathop{\arg\min}_{\bm\beta} g(\bm{\beta};\bm{\beta}^{(t)}).
\end{equation}
The generalized Bregman function $\bm\Delta_\psi$ will be rigourously defined in Section \ref{sec:Bregnotation}, and we will call $g$ a (generalized) Bregman surrogate.
Note that $\bm\Delta_\psi$ is not necessarily the standard Bregman divergence \citep{Bregman1967} because we do not restrict $\psi$ to be smooth or strictly convex or even convex. Bregman divergence does not seem to have been widely used in the statistics community, but see \cite{Zhang2010Bregman}. The generalized Bregman surrogate framework has a close connection to the majorization-minimization (MM) principle \citep{Hunter2004tutorial,Hunter2005}. But   the  surrogate here as a function of $\bm\beta$ matches $f(\bm\beta)$ to a higher order when $\bm\beta^-$ is set to $\bm\beta$ (cf. Lemma \ref{lemma:degeneracy}) and  we do not always invoke the majorization condition $g(\bm{\beta};\bm{\beta}^-) \geq f(\bm{\beta})$; the benefits will be seen in step size control and acceleration.

A variety of algorithms can be recharacterized by Bregman surrogates, including DC programming \citep{Tao1986}, local linear approximation (LLA) \citep{Zou2008LLA} and iterative thresholding \citep{Blumensath2009,SheTISP}. In contrast to the large body of literature in convex optimization, little research has been done on the rate of convergence of nonconvex optimization algorithms when $p > n$, and there is a lack of universal methodologies.
Instead of proving local convergence results for some carefully chosen initial points,  this work aims to   establish \textit{global} convergence rates  regardless of the specific choice of the starting point, where a crucial element is the    error measure. We will see that    the  most natural  measures are unsurprisingly problem-dependent, but  can be conveniently constructed via generalized Bregman functions.

Another perhaps more intriguing question to statisticians is  how the statistical accuracy improves or deteriorates as the cycles progress, and whether the finally obtained estimators can  enjoy provable  guarantees in a statistical sense. See, for example,  \cite{Agarwal2012,Fan2014,Wang2014}; 
in particular, \cite{Loh2015}, one of the main motivations of our work, showed that   for a composite objective composed of a loss and a regularizer that enforces sparsity, the sequence of iterates $\bm\beta^{(t)}$ generated by  gradient-descent type algorithms can approach a  minimizer $\bm\beta^o$   at a linear rate even when $p > n$, if the problem under consideration satisfies some regularity conditions.
This article reveals  broader conclusions when using  generalized Bregman surrogate algorithms in the composite setting: the  more straightforward \textit{statistical error} between the $t$-th iterate $\bm\beta^{(t)}$ and the statistical truth $\bm\beta^*$     enjoys   fast convergence, and  the convergent fixed points, though not necessarily local minimizers, let alone global minimizers,  possess   the desired statistical accuracy in a minimax sense.  The studies support the practice of avoiding unnecessary over-optimization in  high-dimensional sparse learning tasks. Our  theory will make heavy use of the calculus of generalized Bregman  functions---in fact, the proofs  become readily on hand with some nice properties of $\bm\Delta$ established. Again,  a wise choice of the discrepancy measure  can facilitate theoretical analysis  and lead to less restrictive regularity conditions.

Finally, we would like to study and extend   Nesterov's first and second accelerations \citep{Nesterov1983, Nesterov1988}. Accelerated gradient algorithms \citep{Beck2009FISTA, Tseng2008,Krich2015} have lately gained popularity in high-dimensional convex programming because they can attain the optimal rates of convergence among first-order methods. However, since convexity is indispensable to these theories, how to adapt  the momentum techniques  to nonsmooth nonconvex programming is largely unknown.
 Ghadimi and Lan \cite{Ghadimi2016} studied how to accelerate gradient descent type algorithms  when the objective function is nonconvex but strongly smooth;  the obtained convergence rate is of the same order     as gradient descent for  nonconvex problems.
We are interested in more general Bregman surrogates  with a possible lack of smoothness and convexity, most notably in high-dimensional nonconvex sparse learning.  This work will come up with  two momentum-based schemes to accelerate   Bregman-surrogate algorithms by carefully controlling the sequences of relaxation parameters and  step sizes.

Overall, this paper aims to provide a universal tool of generalized Bregman functions in   the interplay between optimization and statistics,  and to  demonstrate its active roles  in  constructing error measures,  formulating less restrictive regularity conditions, characterizing  strong convexity,  deriving the so-called  basic inequalities in nonasymptotic statistical analysis,    devising line search and momentum-based updates, and so on.
The rest of this paper is organized as follows. In Section \ref{sec:Bregman}, we introduce the generalized Bregman surrogate framework and present some examples. Section \ref{sec:results} gives the main theoretical results on  computational accuracy and statistical accuracy. Section \ref{sec:acc} proposes and analyzes two acceleration schemes. We conclude in Section \ref{sec:summ}. Simulation studies and  all technical details are provided in the Appendices.

\paragraph*{Notation} Throughout the paper, we use $C,c$ to denote positive constants. They are not necessarily the same at each occurrence.
The class of   continuously differentiable functions is denoted by  $\mathcal C^{1}$.
Given any matrix $\bm A$, we denote its $(i,j)$-th element by $A_{ij}$. The spectral norm and the Frobenius norm of $\bm A$ are denoted by $\|\bm A\|_2$ and $\|\bm A\|_F$, respectively. The Hadamard product of two matrices $\bm A$ and $\bm B$ of the same dimension is denoted by $\bm A\circ\bm B$ and their inner product is  $\langle\bm A,\bm B\rangle = tr\{{\bm{A}}^\top \bm B\}$.
 If $\bm A-\bm B$ is positive semi-definite, we also write $\bm A \succeq\bm B$.  Let $[p] := \{1,\cdots,p\}$. Given $\mathcal J\subset [p]$, we use $\bm A_\mathcal J$ to denote the submatrix of $\bm A$ formed by the columns indexed by $\mathcal J$. Given a set $A\subset \mathbb R^n$, we use  $A^\circ$,  $\mbox{ri}(A)$,   $\overline A$   to denote its interior,  relative interior, and closure, respectively \cite{Rockafellar1970}. When   $f$ is  an extended real-valued  function from $ D\subset \mathbb R^p$ to $ \mathbb R\cup \{+\infty\}$,   its effective domain is defined as $\mbox{dom}(f)=\{ \bsb{\beta}\in \mathbb R^p: f(\bsb{\beta})<+\infty\}$.
Let $\mathbb R_+=[0, +\infty)$.

\section{Basics of generalized Bregman surrogates}\label{sec:Bregman}

\subsection{Generalized Bregman functions} \label{sec:Bregnotation}

Bregman divergence \citep{Bregman1967}, typically defined for continuously differentiable and strictly convex functions, plays an important role in convex analysis. An extension of it based on ``right-hand'' Gateaux differentials helps to handle   nonsmooth  nonconvex optimization problems.
We begin with  one-sided directional derivative. 
\begin{defn}
\label{def:Gateaux}
Let $\psi: D\subset \mathbb R^p\rightarrow \mathbb R$ be a function. The one-sided directional derivative  of $\psi$ at $\bm{\beta}\in D$ with increment $\bm{h}$ is defined as
\begin{equation} \label{Gateaux}
\delta \psi(\bm{\beta}; \bm{h}) = \lim_{\epsilon \rightarrow 0+} \frac{\psi(\bm{\beta} + \epsilon \bm{h}) - \psi(\bm{\beta})}{\epsilon},
\end{equation}
provided $\bm h$ is admissible in the sense that $\bm\beta+\epsilon\bm h \in D$ for sufficiently small  $\epsilon: 0<\epsilon <\epsilon_0$. When $\psi: D\rightarrow \mathbb R^n$ is a vector function,  $\delta \psi$ is defined componentwise.
\end{defn}

In the following,      $\psi$   is called (one-sided) directionally differentiable at $\bm\beta$ if $\delta \psi(\bm{\beta}; \bm{h})$ as defined in \eqref{Gateaux} exists \emph{and} is finite for all admissible $\bm h$, and if this holds for all  $\bsb{\beta}\in D $, we say that $\psi$ is directionally differentiable.


When $a>0$, $\delta \psi(\bm{\beta}; a \bm{h})=a \delta  \psi(\bm{\beta};    \bm{h}) $, but   $\delta \psi$ is not necessarily  a linear operator with respect to $\bsb{h}$. Definition \ref{def:Gateaux} is a relaxed version of the standard Gateaux differential which studies the limit when $\epsilon \rightarrow 0$.  In high-dimensional sparse problems where nonsmooth regularizers and/or losses are widely used, \eqref{Gateaux} is more convenient and useful.

\begin{defn}[\textbf{Generalized Bregman Function (GBF)}] \label{def:Breg}
The generalized Bregman function associated with a function $\psi$ is defined by
\begin{equation}
\bm{\Delta}_{\psi}(\bm{\beta},\bm{\gamma}) = \psi(\bm{\beta}) - \psi(\bm{\gamma}) - \delta \psi(\bm\gamma; \bm\beta-\bm\gamma),   \label{GBFdef}
\end{equation}
assuming  $\bm{\beta},\bm{\gamma}\in \mbox{dom}(\psi)$ and $\delta \psi(\bm\gamma; \bm\beta-\bm\gamma)$ is meaningful and finite.
In particular, when $\psi$ is  differentiable and strictly convex, the generalized Bregman function $\bm{\Delta}_{\psi}$ becomes the standard Bregman divergence:
\begin{equation}
\mathbf{D}_{\psi}(\bm{\beta},\bm{\gamma}) := \psi(\bm{\beta}) - \psi(\bm{\gamma}) - \langle \nabla \psi(\bm\gamma), \bm\beta-\bm\gamma \rangle. \label{stdBDdef}
\end{equation}
When $\psi$ is a vector function, a vector version of $\breg$  is defined componentwise.
\end{defn}

When  $\nabla\psi$ exists at $\bsb{\beta}$, $\delta \psi(\bm{\beta}, \bm{h})$ reduces to $\langle \nabla \psi(\bm\beta), \bm h \rangle$, which is linear in $\bm h$. So if $\psi$ is the restriction of a    function $\varphi\in \mathcal C^1$ to a convex set,   $\breg_{\psi} (\bsb{\beta}, \bsb{\gamma})= \breg_{\varphi}(\bsb{\beta}, \bsb{\gamma})$ for all $\bsb{\beta}, \bsb{\gamma}\in \mbox{dom}(\psi)$.
For simplicity, all functions in our paper are assumed to be defined on a \textit{whole} vector space ($\mathbb R^p$, typically)   unless otherwise mentioned, although most results can be formulated in the case of extended real-valued functions  under the convexity of their effective domains. 

The generalized Bregman   $\bm\Delta_\psi(\cdot,\bm\gamma)$ can be seen as the difference between the function $\psi$ and its radial  approximations   made at  $\bm\gamma$.
A simple  but important example is $\mathbf{D}_2(\bm\beta,\bm\gamma) := \mathbf{D}_{\|\cdot\|_2^2/2} (\bm\beta,\bm\gamma) =  \|\bm\beta-\bm\gamma\|_2^2/2$. In general, $\bm{\Delta}_\psi$ or $\mathbf{D}_\psi$ may not be symmetric. The following symmetrized version turns out to be useful:
\begin{equation}
\sym{\bm{\Delta}}_\psi(\bm{\beta},\bm{\gamma}) :=  \frac{1}{2}(\bm\Delta_\psi + \back{\bm\Delta}_\psi)(\bm\beta,\bm\gamma) = \frac{1}{2}\{\bm{\Delta}_\psi(\bm{\beta},\bm{\gamma}) + \bm{\Delta}_\psi(\bm{\gamma},\bm{\beta})\},
\end{equation}
where ${\back{\breg}}(\bm{\beta},\bm\gamma)$ denotes $\bm\Delta(\bm\gamma,\bm\beta)$. If $\psi$ is smooth, $\sym{\bm{\Delta}}_\psi(\bm{\beta},\bm{\gamma})  = \langle \nabla \psi(\bsb{\beta}) - \nabla \psi(\bsb{\gamma}),  \bsb{\beta} - \bsb{\gamma}\rangle$.

To simplify the notation, we   use $\bm\Delta_\psi\ge \bm\Delta_\phi$ to denote $\bm\Delta_\psi(\bm\beta,\bm\gamma) \ge \bm\Delta_\phi(\bm\beta,\bm\gamma)$ for all $\bm\beta,\bm\gamma $, and so $\bm\Delta_\psi \ge 0$ stands for $\bm\Delta_\psi(\bm\beta,\bm\gamma) \ge 0, \forall\bm\beta,\bm\gamma$. Some basic properties of $\bm\Delta$   are given as follows.

\begin{lemma}\label{lemma:Delta_linear}
Let $\psi$ and $\varphi$ be  directionally differentiable functions. Then for any $\bm{\alpha}, \bm{\beta}, \bm{\gamma}$, we have
the following properties.
\begin{description}
  \item (i) \ $\bm{\Delta}_{a\psi + b\varphi}(\bm{\beta},\bm{\gamma}) = a\bm{\Delta}_{\psi}(\bm{\beta},\bm{\gamma}) + b\bm{\Delta}_{\varphi}(\bm{\beta},\bm{\gamma})$, $\forall a, b\in \mathbb{R}$.
  \item (ii) If $\it\psi$ is convex,  it is directionally differentiable and $\bm\Delta_\psi \ge 0$; conversely, if $\psi$ is directionally differentiable and  $\bm\Delta_\psi\ge 0$ then   $\psi$ is convex.

\item (iii) If $\psi  : \mathbb R^n\rightarrow \mathbb R  $ is differentiable and $\varphi: \mathbb R^p \rightarrow \mathbb R^n$ is continuous and directionally differentiable, then $\breg_{\psi \circ \varphi} (\bsb{\beta}, \bsb{\gamma}) = \breg_{\psi  } (\varphi(\bsb{\beta}), \varphi(\bsb{\gamma})) + \langle \breg_{\varphi}(\bsb{\beta}, \bsb{\gamma}), \nabla \psi(\varphi(\bsb{\gamma})) \rangle $.   Also, if $\psi: \mathbb R^n\rightarrow \mathbb R  $ is directionally differentiable and    $\varphi: \mathbb R^p \rightarrow \mathbb R^n$  is linear,   then $\breg_{\psi \circ \varphi  } (\bsb{\beta}, \bsb{\gamma}) = \breg_{\psi  } (\varphi( \bsb{\beta}), \varphi(  \bsb{\gamma}))$.

  \item (iv) $\bm\Delta_\psi(\bm\beta,\bm\gamma) = \int_0^1 \big[\delta\psi\big(\bm\gamma + t(\bm\beta-\bm\gamma);\,\bm\beta-\bm\gamma \big) - \delta \psi(\bm\gamma; \bm\beta-\bm\gamma)\big]\mathrm{d}t$, provided $\delta\psi(\bm\gamma + t(\bm\beta-\bm\gamma);\,\bm\beta-\bm\gamma )$ is integrable over $t\in[0,1]$.
\end{description}
\end{lemma}

The properties will be frequently  used in the rest of the paper. For instance, for $\psi =\rho\| \cdot\|_2^2/2 - f$, by (i) we can write $\breg_\psi = \rho \Breg_2 - \breg_f$. Sometimes, though $f$ is not necessarily convex,   $f + \nu \| \cdot\|_2^2/2$ is so for some $\nu\in \mathbb R$, which means   $\breg_f\ge -\nu\Breg_2$, owing to (ii). For $l(\bsb{\beta}) = l_0(\bsb{X} \bsb{\beta}+\bsb{\alpha})$, commonly encountered in statistical applications, (iii) states that $\breg_l(\bsb{\beta}, \bsb{\gamma}) = \breg_{l_0}(\bsb{X} \bsb{\beta}+\bsb{\alpha}, \bsb{X} \bsb{\gamma}+\bsb{\alpha})$.   For (iv), the integrability condition  is met when the directional derivative restricted to the interval $[\bsb{\beta}, \bsb{\gamma}]$ is bounded by  a constant (or more generally a   Lebesgue integrable function); in particular,   if    $\psi$  is $L$-strongly smooth, that is, $\nabla \psi$ exists  and  is Lipschitz continuous:
$ 
\|\nabla \psi(\bm\beta) - \nabla \psi(\bm\gamma)\|_* \le L\|\bm\beta - \bm\gamma\|  \text{ for any }\bm\beta, \bm\gamma,
$ where $\|\cdot \|_*$ is the dual norm of $\| \cdot \|$, 
 $\breg_\psi(\bsb{\beta}, \bsb{\gamma})\le L \| \bsb{\beta}-\bsb{\gamma}
\|^2/2$ and for the Euclidean norm, $\breg_\psi \le L \Breg_2$ results.
\\

Moreover, the GBF  operator satisfies some interesting ``idempotence'' properties under some mild assumptions, which is extremely helpful in studying iterative optimization algorithms.

\begin{lemma}\label{lemma:Delta_Delta}
(i) When $\psi$ is convex, $ \bm\Delta_{\bm\Delta_\psi(\cdot,\bm\alpha)}(\bm\beta,\bm\gamma)\le \bm\Delta_\psi(\bm\beta,\bm\gamma)  $, and when $\psi$ is concave, $ \bm\Delta_{\bm\Delta_\psi(\cdot,\bm\alpha)}(\bm\beta,\bm\gamma)\ge \bm\Delta_\psi(\bm\beta,\bm\gamma) $ for all $\bsb{\alpha}, \bsb{\beta}, \bsb{\gamma}$.

(ii) When $\psi$ is directionally differentiable, for all $\bsb{\alpha}=(1-\theta) \bsb{\gamma} +  \theta  \bsb{\beta}$ with $\theta\not\in(0,1)$,  $ \bm\Delta_{\bm\Delta_\psi(\cdot,\bm\alpha)}(\bm\beta,\bm\gamma)= \bm\Delta_\psi(\bm\beta,\bm\gamma)  $ and in particular,
\begin{align} \bm\Delta_{\bm\Delta_\psi(\cdot,\bm\beta)}(\bm\beta,\bm\gamma)= \bm\Delta_{\bm\Delta_\psi(\cdot,\bm\gamma)}(\bm\beta,\bm\gamma)=\bm\Delta_\psi(\bm\beta,\bm\gamma). \label{eq:idempotence-weak}
\end{align}

(iii) When  $\delta\psi (\cdot; \bsb{\beta}-\bsb{\gamma})$ is bounded in a neighborhood of $\bsb{\alpha}$ and has restricted radial continuity at $\bsb{\alpha}$: $\lim_{\epsilon\rightarrow 0+} \delta \psi(\bsb{\alpha}+\epsilon \bsb{h}; \bsb{\beta} - \bsb{\gamma}) =\delta \psi(\bsb{\alpha}; \bsb{\beta} - \bsb{\gamma}) $ for any $\bsb{h}\in[\bsb{\beta}-\bsb{\alpha}, \bsb{\gamma}-\bsb{\alpha}] $, or when  $\delta\psi (\bsb{\alpha}; \cdot)$ has restricted linearity $\delta\psi (\bsb{\alpha}; \bsb{h})=\langle g(\bsb{\alpha}), \bsb{h}\rangle$ for  some $g$ and  all $\bsb{h}\in [\bsb{\beta}-\bsb{\alpha}, \bsb{\gamma}-\bsb{\alpha}]$,    we have
   \begin{align} \bm\Delta_{\bm\Delta_\psi(\cdot,\bm\alpha)}(\bm\beta,\bm\gamma)= \bm\Delta_\psi(\bm\beta,\bm\gamma). \label{eq:idempotence}
\end{align} In particular, \eqref{eq:idempotence} holds when $\psi$ is differentiable at $\bsb{\alpha}$ or    $\delta\psi (\cdot; \bsb{\beta}-\bsb{\gamma})$ is continuous at $\bsb{\alpha}$.
\end{lemma}

We  refer to (ii) as the \emph{weak idempotence} property and (iii) as the \emph{strong idempotence} property.
When $\bm\Delta_\psi$ becomes a legitimate Bregman divergence, \eqref{eq:idempotence} can be rephrased into the three-point property $\Breg_{\psi} (\bm{\beta},\bm{\gamma})= \Breg_{\psi} (\bm{\beta},\bm{\alpha})+ \Breg_{\psi}(\bm{\alpha},\bm{\gamma}) - \langle\bm{\beta} - \bm{\alpha}, \nabla \psi (\bm{\gamma}) - \nabla \psi(\bm{\alpha})) \rangle$ \citep{Chen1993}.
It is worth mentioning that although from (iii), differentiability can be used to gain strong idempotence,   the weak    idempotence \eqref{eq:idempotence-weak}  is often what we need, which always holds under just directional differentiability. 
\\

At the end of the subsection, we give some important facts of GBFs    for    canonical generalized linear models (GLMs) that are widely used in statistics modeling.
Here, the response  variable $\bsb{y}\in \mathcal Y^n  \subset \mathbb R^n$ has density $p_{ \bsb{\eta}}(\cdot )=\exp\{(\langle \cdot, \bsb{\eta}\rangle -  {b}(\bsb{\eta} ))/\sigma^{2}  - c(\cdot,\sigma^2)\}\ $ with respect to measure $\nu_0$ defined on $\mathcal Y^n$ (typically the counting measure or   Lebesgue measure), where   $\bsb{\eta}\in \mathbb R^n$ represents the systematic component of interest, and $\sigma$  is the scale parameter; see  \cite{Jor87}.
Since $\sigma$ is not the parameter of interest, it is more convenient to define the density $\exp\{   (\langle \cdot, \bsb{\eta}\rangle -  {b}(\bsb{\eta} ))/\sigma^{2}\}$ (still written as $p_{ \bsb{\eta}}(\cdot )$ with a slight abuse of notation) with respect to the base measure  $\rd \nu = \exp(- c(\cdot,\sigma^2))\rd \nu_0$.    The  loss for   $\bsb{\eta}$  can be written as
\begin{align}\label{l0glmass}
l_0 ( \bsb{\eta} ; \bsb{y}  ) =  \{-\langle \bsb{y},   \bsb{\eta}\rangle + b(\bsb{\eta} )\}/\sigma^2 .\end{align} That is, $l_0$  corresponds to  a distribution in the exponential dispersion family with  cumulant function $b(\cdot)$,  dispersion $\sigma^2$ and natural parameter $\bsb{\eta} $.
In the Gaussian case, $l_0 (\bsb{\eta} ) =  - \langle  \bsb{\eta},   \bsb{y}\rangle/\sigma^2 + \| \bsb{y} \|_2^2/(2 \sigma^2) $. 

Following \cite{WainJordan08},   we define the natural parameter space    $\Omega =\mbox{dom}(b)= \{\bsb{\eta} \in \mathbb R^n: b(\bsb{\eta})< \infty\}$ (always assumed to be nonempty) and the mean parameter space  $\mathcal M = \{\bsb{\mu}\in \mathbb R^n: \bsb{\mu} = \EE  \bsb{y}, \mbox{ where  } \bsb{y}\sim  p \mbox{ for some density } p \mbox{ defined on }   \mathcal Y^n \allowbreak \mbox{ with respect to } \nu\} $, and call  $p_{ \bsb{\eta}}$     minimal if      $\langle \bsb{a}, \bsb{z}\rangle = c$  for almost every   $\bsb{z}\in \mathcal Y^{n}$ with respect to  $\nu$   implies $\bsb{a} = \bsb0$.   When $\Omega$ is open,    $p_{\bsb{\eta}}$ is called regular, and   $ b$ can be shown to be differentiable to any order and  convex, but not necessarily strictly convex; if, in addition, $p_{ \bsb{\eta}} $   is   minimal, $b$ is strictly convex and the canonical link $g = (\nabla b )^{-1}$ is well-defined on $\mathcal M^\circ$. These can all be derived  from, say, the propositions in   \cite{WainJordan08}.

\begin{lemma} \label{lemma:glmBreg}
Assume the  exponential dispersion family setup with  the associated loss defined in \eqref{l0glmass}.  (i) If $\Omega$ is an open set or $p_{\bsb{\eta}}$ is regular, then
\begin{align}
l_0(\bsb{\eta}; \bsb{z}) =\breg_{b}(\bsb{\eta}, \partial b^*(   \bsb{z}))/\sigma^2 -b^{*}(\bsb{z})/\sigma^2\label{bregglm_cum}
\end{align} for all $\bsb{\eta} \in \Omega, \bsb{z} \in \mbox{ri}(\mathcal M)$, where  $b^*$ is the Fenchel conjugate of $b$, and $\partial b^*  (   \bsb{z})$ can take any subgradient of $b^*$ at $\bsb{z}$. If $p_{\bsb{\eta}}$ is also minimal, $\breg_{b}$ becomes $\Breg_b$, $ \partial b^*(   \bsb{z})$ becomes  $g(   \bsb{z})$ (which is unique), and  $\mbox{ri}(\mathcal M)$ becomes $\mathcal M^\circ$.  (ii) As long as  $\Omega$ is  open,
\begin{align}
l_0(\bsb{\eta}; \bsb{z}) = \breg_{  b^* }(\bsb{z}, \nabla b (\bsb{\eta}))/\sigma^2  -b^{*}(\bsb{z})/\sigma^2 \label{bregglm_conj}
 \end{align} for all $\bsb{\eta} \in \Omega, \bsb{z} \in \mathcal \mbox{ri}(\mathcal M)$. If $p_{\bsb{\eta}}$ is also minimal,  $  \breg_{  b^* }=\Breg_{b^*}$ and  $\mbox{ri}(\mathcal M)=\mathcal M^\circ$. (iii) Given any  $\bsb{\eta}_1\in \Omega^\circ$ and  $\bsb{\eta}_2\in \Omega$, the Kullback Leibler (KL) divergence of   $p_{\eta_2} $ from  $p_{\eta_1} $   relates to the GBF of    $l_0$ or $b$ by
\begin{align}
 \mbox{KL}(  p_{\bsb{\eta}_1} , p_{\bsb{\eta}_2})   = \breg_{l_0} (   \bsb{\eta}_2,   \bsb{\eta}_1) = \breg_b( \bsb{\eta}_2, \bsb{\eta}_1)/\sigma^2. \label{bregglm_KL}
  \end{align}
\end{lemma}

Property (i) shows the importance of GBF in maximum likelihood estimation. A  Bregman version of Property (ii) was  first described in \cite{banerjee2005clustering}, while our conclusions   based on   $\breg_b, \breg_{b^*}$ are more general, as they   do \textit{not} require the strict convexity of $b$ or the differentiability of $b^*$.  Consider for instance the multinomial GLM under a symmetric parametrization: for $[y_1, \ldots, y_m] \in \mathcal Y = \{y_k \in \{ 0, 1\}, 1\le k \le m, \sum y_k =1\}$ ($n=1$),    $\EE y_k \propto \exp( {\eta}_k)$ or  $\EE y_k = \exp( {\eta}_k)/\sum \exp( {\eta}_k)$   gives  $b = \log \sum \exp( {\eta}_k)$, and thus $b^*(\bsb{\mu})$ takes $ \sum \mu_k \log \mu_k  $ for $ [\mu_1, \ldots, \mu_m] \in \mathcal M =\{  [\mu_k]: \sum \mu_k =1, \mu_k\ge 0\} $ and $+\infty$ otherwise. Clearly,    $b^*$ is  not differentiable (given any $\bsb{z}\in \mbox{ri} (\mathcal M)$, $\partial b^*(\bsb{z}) = \{\log \bsb{z} + t \bsb{1}:  t\in \mathbb R\}$),  but nicely our  two GBF representations  still hold.   In addition, if the right-hand side of  \eqref{bregglm_cum} or  \eqref{bregglm_conj}, as a function of $\bsb{z}$, is continuous on $\overline{ \mathcal M}$, which is the case for    Bernoulli, multinomial and Poisson,   (i) and (ii) hold for any $\bsb{z}\in \overline{ \mathcal M}$ from  \cite[Theorem 3.4]{WainJordan08}.

Property (iii) (notice the exchange of $\bsb{\eta}_1$ and $\bsb{\eta}_2$ in the generalized Bregman expressions) can be used to formulate and verify model       regularity conditions  in    minimax  studies of {sparse} GLMs, which  are of great interest in high-dimensional statistical learning \citep{Tsybakov2008}.
More concretely, consider a general signal class 
\begin{equation}
\mathcal B(s^*, M) = \{ \bm\beta^*\in \mathbb R^p : \|\bm\beta^*\|_0\le s^*, \| \bsb{\beta}^*\|_\infty\le  M \},
\end{equation}
where $s^*\le  p$, $0\le M\le +\infty$.
Some applications limit the magnitude of the coefficients $\beta_j$ via a constraint or a penalty, resulting in a finite $M$. Let $I(\cdot)$ be any nondecreasing   function with $I(0)=0, I\not\equiv 0$.
Some particular examples are $I(t) = t$ and $I(t) = 1_{t\ge c}$.
Recall the regular exponential dispersion family  with  systematic component $\bsb{\eta}=\bsb{X} \bsb{\beta}$ and loss   $l(\bsb\beta) = l_0(\bsb{\eta})$    defined by  \eqref{l0glmass}. 

\begin{thm}\label{thm:minimaxglm}
In the regular exponential dispersion family setup (with    $ \mbox{dom}(b)$   a nonempty open set), assume    $  p\ge 2, 1\le s^*\le p/2$. Let
  \begin{align}
P (s^* ) = s^*\log(ep/s^*).
\end{align}

(i) If
\begin{align}
\breg_{l_0} (\bsb{0}, \bsb{X}    \bsb{\beta}   )\sigma^2 \le    \kappa\Breg_2(    \bsb{0},  \bsb{\beta} ), \quad \forall   \bm\beta\in \mathcal B( s^*, M) \label{regcond-glmestminimax}
\end{align}
 where $ \kappa>0$,   there exist positive constants $ c, \tilde c$, depending on $I(\cdot)$ only, such that
\begin{align*}
\inf_{ \hat{\bm\beta} }\,\sup_{ \bm\beta^* \in \mathcal B(s^* , M)} \EE \big\{I\big(\Breg_2  (   \bsb{\beta}^*,    \hat{  \bsb{\beta}} )/[\tilde c \min\{ \sigma^2  P (s^*)/\kappa, M^2 s^*\}]\big)\big\} \ge c >0, 
\end{align*}
where $ \hat{\bm\beta} $ denotes any estimator of $ \bm\beta^*$.

(ii) If
\begingroup
  \singlespacing
\begin{equation}
\begin{cases}
  \underline\kappa\Breg_2(      \bsb{\beta}_1,    \bsb{\beta}_2)\le\Breg_2(     \bsb{X} \bsb{\beta}_1,    \bsb{X}\bsb{\beta}_2) \\  \breg_{l_0} ( \bsb{0}, \bsb{X}    \bsb{\beta}_1   )\sigma^2 \le     \overline\kappa\Breg_2(    \bsb{0},      \bsb{\beta}_1), \end{cases}  \qquad\forall   \bm\beta_i\in \mathcal B(s^*,M)
     \label{regcond-glmestminimax-2}
 \end{equation}
 \endgroup
where $\underline \kappa, \overline \kappa\ge 0$,   then there exist positive constants $ c, \tilde c$  depending on $I(\cdot)$ only  such that
\begin{align*}
\inf_{ \hat{\bm\beta} }\,\sup_{ \bm\beta^* \in \mathcal B(s^* ,M)} \EE\big\{I\big(\Breg_2  (  \bsb{X} \bsb{\beta}^*,   \bsb{X} \hat{  \bsb{\beta}} )/[c \min\{(\underline\kappa/ \overline \kappa)\sigma^2  P (s^* ), \underline \kappa M^2 s^*\} ]\big)\big\}  \ge c >0.  
\end{align*}
\end{thm}

The GBF-form conditions \eqref{regcond-glmestminimax}, \eqref{regcond-glmestminimax-2} can be viewed as an  extension  of restricted isometry    \citep{candes2007dantzig}, and are often easy  to check using the Hessian. For example, from     Lemma \ref{lemma:Delta_linear}, we immediately know that if $l_0$ is $L$-strongly smooth, 
 \eqref{regcond-glmestminimax}   is satisfied with $\kappa = L \| X\|_2^2$ even when  $M=+\infty$.    This is  the case for regression and logistic regression, and accordingly, no estimation algorithms can beat the  minimax  rate   $s^*\log(ep/s^*)$ (ignoring trivial factors). The {optimal} lower bounds provide useful guidance in establishing sharp statistical  error upper bounds of  Bregman-surrogate algorithms in Section \ref{subsec:stat}.

\subsection{Examples of Bregman surrogates}\label{sec:ex}

\begin{ex} \upshape \label{ex:mirrdes}
\textbf{(Gradient descent and mirror descent)}.
Gradient descent is a simple first-order method to minimize a  function $f\in \mathcal C^1$ which may be nonconvex. Starting with $\bm\beta^{(0)}$, the algorithm proceeds as follows:
\begin{equation}
\bm\beta^{(t+1)} = \bm\beta^{(t)} - \alpha\nabla f(\bm\beta^{(t)}),
\end{equation}
where $\alpha > 0$ is a step size parameter. Its rationale can be seen by formulating a Bregman-surrogate algorithm using $\bm\Delta_\psi = \rho\mathbf D_2 - \bm\Delta_f$:
\begin{subequations}
\label{GD}
\begin{align}
\bm\beta^{(t+1)}  &= \mathop{\arg\min}_{\bm\beta}\  g(\bm\beta;\bm\beta^{(t)})= f(\bm\beta) + (\rho\mathbf D_2 - \bm\Delta_f)(\bm\beta,\bm\beta^{(t)}) \label{GDa}\\
 &= \bm\beta^{(t)} - \frac{1}{\rho}\nabla f(\bm\beta^{(t)}),
\end{align}
\end{subequations}
where $f(\cdot) - \bm\Delta_f(\cdot,\bm\beta^{(t)})$ gives a linear approximation of $f$ and $1/\rho$ amounts to the step size. We call $\rho$ the inverse step size parameter.   (The generalized Bregman surrogate in   \eqref{GDa}   extends the class of algorithms to a directionally differentiable   $f$,     with the update given by  $\bm\beta^{(t+1)} = \bm\beta^{(t)} + (0\vee - \delta f(\bm\beta^{(t)};\bm h^\circ)) \bm h^\circ/\rho$ and $\bm h^\circ \in \mathop{\arg\max}_{\|\bm h\|_2=1}[\delta f(\bm\beta^{(t)};\bm h)]_-$, where $[\,]_-$ denotes the negative part ($t_- = (|t| - t)/2$).) 

More generally, we can use a   strictly convex $\varphi\in \mathcal C^1$ to construct
\begin{equation}\label{MD-g}
g(\bm\beta;\bm\beta^{(t)}) = f(\bm\beta) + (\rho\mathbf D_\varphi - \bm\Delta_f)(\bm\beta,\bm\beta^{(t)}),
\end{equation}
Minimizing \eqref{MD-g} with respect to $\bm\beta$ gives the renowned mirror descent  \citep{Nemirovski1983Book}:
$\bm{\beta}^{(t+1)} = (\nabla\varphi)^{-1}(\nabla\varphi(\bm{\beta}^{(t)}) - \nabla f(\bm{\beta}^{(t)})/\rho)$,
where $(\nabla \varphi)^{-1}$ is the inverse of $\nabla \varphi$.
Mirror descent is widely used in convex programming, but  this work does \textit{not} restrict $f$ to be convex.
\end{ex}


\begin{ex}\label{ex:thres} \upshape
\textbf{(Iterative thresholding)}.
Sparsity-inducing penalties are widely used in high-dimensional problems; see, for example, $\ell_0$, $\ell_1$ \citep{Tibshirani1996}, bridge penalties \citep{Frank1993Bridge}, SCAD \citep{FanLi2001SCAD}, capped-$\ell_1$ \citep{Zhang2010} and MCP \citep{Zhang2010MCP}. There is a universal connection between thresholding rules and penalty functions \citep{SheGLMTISP}, and the mapping from penalties to thresholdings is many-to-one.
This makes it possible to apply an iterative thresholding algorithm   to solve a general penalized problem of the form $\min_{\bm\beta} l(\bm\beta)+\sum_jP(\varrho\beta_j;\lambda)$ \citep{Blumensath2009,SheTISP}:
\begin{equation}\label{Thres-update}
\bm\beta^{(t+1)} = \Theta(\varrho\bm\beta^{(t)} -  \nabla l(\bm\beta^{(t)})/{\varrho}; \lambda)/{\varrho},
\end{equation}
where $\Theta$ is a thresholding function inducing $P$, and $\varrho > 0$ is an algorithm parameter for the sake of scaling and convergence control. This  class of iterative algorithms is called the \textit{Thresholding-based Iterative Selection
Procedures} (TISP) in \cite{SheTISP} and is  scalable in computation.
For the rigorous definition of $\Theta$ and the $\Theta$-$P$ coupling formula, see Section \ref{subsec:comp} for detail. Some examples of $\Theta$ include: (i) soft-thresholding $\Theta_S(t;\lambda) = \sgn(t)(|t|-\lambda)1_{|t|>\lambda}$, which induces the $\ell_1$ penalty, (ii) hard-thresholding $\Theta_H(t;\lambda) = t1_{|t|>\lambda}$, which is associated with (infinitely) many penalties, with the capped-$\ell_1$ penalty, \eqref{def:PH}, and the discrete $\ell_0$ penalty as particular instances. The nonconvex SCAD and MCP penalties also have their corresponding thresholding rules.
In this sense, thresholdings extend proximity operators. One can regard \eqref{Thres-update} as an outcome of minimizing the following Bregman surrogate
\begin{equation} \label{IterThres-g}
g(\bm\beta;\bm\beta^{(t)}) = l(\bm\beta) + \sum P(\varrho\beta_j;\lambda) + (\varrho^2\mathbf D_2 - \bm\Delta_l)(\bm\beta,\bm\beta^{(t)}).
\end{equation}
Here,  we  linearize $l$ only, as $\min_{\bm\beta} g(\bm\beta;\bm\beta^{(t)})$   has \eqref{Thres-update} as its globally optimal solution.
 Interestingly,  the set of fixed points under the $g$-mapping enjoys provable guarantees that may \textit{not} hold for the set of  local minimizers  to the original objective  (Section \ref{subsec:fixed}). This is particularly the case when  $\Theta$ has {discontinuities} and $P(t;\lambda)$ is given by $ P_\Theta(t;\lambda) + q(t;\lambda)$, where $P_\Theta$ is  defined by \eqref{pendef} and $q$ is a function satisfying $q(t;\lambda)\ge 0$ for all $t\in\mathbb R$ and $q(t;\lambda) = 0$ if $t = \Theta(s;\lambda)$ for some $s\in\mathbb R$  \citep{She2016}.

A closely related \textit{iterative quantile-thresholding} procedure  \cite{SheGLMTISP,SheSpec} proceeds by $\bsb{\beta}^{(t+1)} = \Theta^\#(\bm\beta^{(t)} -  \nabla l(\bm\beta^{(t)})/{\varrho}^2; q) $ for the sake of feature screening:  $\min l(\bsb{\beta})$ s.t. $\| \bsb{\beta}\|_0\le q$, and uses  a similar surrogate $g(\bm\beta;\bm\beta^{(t)}) = l(\bm\beta)  + (\varrho^2\mathbf D_2 - \bm\Delta_l)(\bm\beta,\bm\beta^{(t)})$. Here,  the quantile thresholding $\Theta^\#(\bsb{\alpha}; q)$, as an outcome of $\min g(\bsb{\beta}; \bsb{\beta}^{-})$, keeps the top $q$ elements of $\alpha_j$ after ordering them in magnitude, $|\alpha_{(1)}| \ge \cdots \ge  |\alpha_{(p)}|$,  and zero out the rest. To avoid ambiguity, we assume no ties  occur   in performing $\Theta^{\#}(\bm{\alpha}; q)$ throughout the paper, that is, $|\alpha_{(q)}| > |\alpha_{(q+1)}|$.
\end{ex}


\begin{ex} \label{ex:NMF}\upshape
\textbf{(Nonnegative matrix factorization)}.
Nonnegative Matrix Factorization (NMF) \citep{Lee1999NMF} provides an effective tool for feature extraction and finds widespread applications in computer vision, text mining and many other areas.
NMF approximates a nonnegative data matrix $\bm{X} \in \mathbb{R}_+^{n\times p}$ by the product of two nonnegative low-rank matrices $\bm{W}\in \mathbb{R}_+^{n\times r}$ and $\bm{H} \in \mathbb{R}_+^{r\times p}$.
The  KL  divergence is often used to make a cost function, that is,
$\min_{\bm W\in\mathbb R_+^{n\times r}, \bm H\in\mathbb R_+^{r\times p}} \mathrm{KL}(\bm{X},\bm{WH}) := \sum_{i,j}[X_{ij}\log(X_{ij}/(\bm W\bm H)_{ij}) - X_{ij} + (\bm W\bm H)_{ij}]$, which gives a nonconvex optimization problem. The following \textit{multiplicative} update rule  (MUR) shows good scalability in big data applications \citep{Cichocki2006}:
\begin{align}
H_{kj}^{(t+1)} &= H_{kj}^{(t)}\exp\Big[-\frac{1}{\rho}\sum_i\Big(W_{ik} - \frac{W_{ik}X_{ij}}{(\bm W\bm H^{(t)})_{ij}}\Big)\Big], \label{NMF-Hupdate}\\
W_{ik}^{(t+1)} &= W_{ik}^{(t)}\exp\Big[-\frac{1}{\rho}\sum_j\Big(H_{kj} - \frac{H_{kj}X_{ij}}{(\bm W^{(t)}\bm H)_{ij}}\Big)\Big]. \label{NMF-Wupdate}
\end{align}
The update formulas can be explained from a Bregman surrogate perspective.
Since the problem is symmetric in $\bm{W}$ and $\bm{H}$, $\bm\Delta_{\mathrm{KL}}(\bm{X},\bm{WH}) = \bm\Delta_{\mathrm{KL}}(\bm{X}^{\top},\bm{H}^{\top}\bm{W}^{\top})$, we take \eqref{NMF-Hupdate} for instance to illustrate the point.
Noticing that the criterion is separable in the column vectors of $\bm H$, it suffices to look at $\mathop{\min}_{\bm h \in \mathbb R_+^r}  f(\bm h) = \mathrm{KL}(\bm x, \bm W\bm h)= \sum_i[x_i\log(x_i/(\bm W\bm h)_i) - x_i + (\bm W\bm h)_i]$,
where $\bm x$ can be any column of $\bm X$.
Then it is easy to verify that the following Bregman surrogate,
\begin{equation}
g(\bm h; \bm h^{(t)}) = f(\bm h) + (\rho \mathbf D_\varphi - \mathbf D_{f})(\bm h,\bm h^{(t)}),~ \varphi(\bm h) = \sum(h_i\log h_i - h_i),
\end{equation}
leads to the multiplicative update formulas.
\end{ex}


\begin{ex}\label{ex:DC} \upshape
\textbf{(DC programming)}.
DC programming \citep{Tao1986} is capable of tackling a large class of nonsmooth nonconvex optimization problems; see, for example, \cite{Gasso2009,PanShen2013}. A ``difference of convex'' (DC) function $f$ is defined by $f(\bm\beta) = d_1(\bm\beta) - d_2(\bm\beta)$, where $d_1$ and $d_2$ are both closed convex functions.
To minimize $f(\bm\beta)$, a standard DC algorithm generates two sequences $\{\bm\beta^{(t)}\}$ and $\{\bm\gamma^{(t)}\}$ that obey
\begin{equation}\label{DC-seq}
\bm\gamma^{(t)} \in \partial d_2(\bm\beta^{(t)}),~ \bm\beta^{(t+1)} \in \partial d_1^*(\bm\gamma^{(t)}),
\end{equation}
where $\partial d(\bm\beta)$ is the subdifferential of $d(\cdot)$ at $\bm\beta$, and $d_1^*(\cdot)$ is the Fenchel conjugate  of $d_1(\cdot)$.
(As before,  $d_1, d_2$ are assumed to be real-valued  functions defined on $\mathbb R^p$, so  the sequences are well-defined and  finite.) This elegant algorithm does not involve any line search and guarantees global convergence given any initial point.
Many popular nonconvex algorithms can be derived from \eqref{DC-seq} \citep{An2005_DC}.

Focusing on the $\bm\beta$-update, we know that $\bm\beta^{(t+1)}$ must be  a solution to $\min_{\bm\beta} d_1(\bm\beta) - \langle\bm\beta, \bm\gamma^{(t)}\rangle$ or $\min_{\bm\beta} d_1(\bm\beta) - \langle\bm\beta-\bm\beta^{(t)}, \bm\gamma^{(t)}\rangle$.
Due to the convexity of $d_2$,    $\langle \bm\beta-\bm\beta^{(t)}, \bm\gamma^{(t)}\rangle \le\sup_{\bsb{\gamma}\in\partial d_2(\bm\beta^{(t)}) }\langle \bm\beta-\bm\beta^{(t)}, \bm\gamma\rangle= \delta d_2(\bm\beta^{(t)};\bm\beta-\bm\beta^{(t)})$ for all   $ \bm\gamma^{(t)}\in\partial d_2(\bm\beta^{(t)}), \bm\beta\in \mathbb R^p $. Thus $\min_{\bm\beta} d_1(\bm\beta) - \langle\bm\beta-\bm\beta^{(t)}, \bm\gamma^{(t)}\rangle$ should be no lower than $\min_{\bm\beta} d_1(\bm\beta) -  \delta d_2(\bm\beta^{(t)}; \bm\beta - \bm\beta^{(t)})$. Choosing $\bm\beta^{(t+1)}\in \arg\min \allowbreak d_1(\bm\beta) -  \delta d_2(\bm\beta^{(t)}; \bm\beta - \bm\beta^{(t)})$ and $\bm\gamma^{(t)} = \delta d_2(\bm\beta^{(t)}; \bm\beta^{(t+1)}-\bm\beta^{(t)})(\bm\beta^{(t+1)}-\bm\beta^{(t)})/\|\bm\beta^{(t+1)}-\bm\beta^{(t)}\|_2^2$ ensures \eqref{DC-seq}, which simply amounts to using  a  Bregman surrogate
\begin{equation}
g(\bm\beta;\bm\beta^{(t)}) = f(\bm\beta) + \bm\Delta_{d_2}(\bm\beta,\bm\beta^{(t)}).
\end{equation}
For the $\bsb\gamma$-updates,  a   Bregman surrogate   $g(\bm\gamma;\bm\gamma^{(t)}) = (d_2^* - d_1^*)(\bm\gamma) + \bm\Delta_{d_1^*}(\bm\gamma,\bm\gamma^{(t)})$ can be similarly constructed.
\end{ex}


\begin{ex}\label{ex:LLA} \upshape
\textbf{(Local linear approximation)}.
Zou and Li \cite{Zou2008LLA}  proposed an effective local linear approximation (LLA) technique to minimize penalized negative log-likelihoods. In their paper, the loss function is assumed to be convex and smooth, and the penalty is concave on $\mathbb R_+$. We give a new characterization of LLA by use of a Bregman surrogate.

Let $l$ be a directionally differentiable loss function {but not} necessarily continuously differentiable, and $P$ be a function that is concave and differentiable over $(0,+\infty)$, and satisfies $P(t) = P(-t)$ for any $t\in \mathbb R$, $P(0) = 0$. Consider the problem
$\min_{\bm\beta} l(\bm\beta) + \sum_jP(\beta_j)$. Using the generalized Bregman notation $\bm\Delta_{\|\cdot\|_1}(\bm\beta,\bm\gamma)$, or $\bm\Delta_1(\bm\beta,\bm\gamma)$ for short,  define
\begin{equation}\label{LLA-g}
g(\bm\beta;\bm\beta^{(t)}) = l(\bm\beta) + \sum P(\beta_j) + \sum\big[\alpha_j\bm\Delta_1 (\beta_j,\beta_j^{(t)}) - \bm\Delta_P(\beta_j,\beta^{(t)}_j)\big].
\end{equation}
In contrast to \eqref{IterThres-g}, \eqref{LLA-g} linearizes $P$ instead of $l$.
Simple calculation shows
\begin{align}
\bm\Delta_1(\beta_j,\beta_j^{(t)}) &= \begin{cases}|\beta_j| - \sgn(\beta_j^{(t)})\beta_j,&\beta_j^{(t)}\ne 0\\ 0,&\beta_j^{(t)} = 0,\end{cases}\label{Delta_1}\\
\bm\Delta_P(\beta_j, \beta_j^{(t)}) &= \begin{cases}P(\beta_j) - P(\beta_j^{(t)}) - P'(\beta_j^{(t)})(\beta_j - \beta_j^{(t)}),&\beta_j^{(t)} \ne 0\\ P(\beta_j) - P'_+(0)|\beta_j|,& \beta_j^{(t)} = 0,\end{cases} \label{Delta_P}
\end{align}
where $\sgn(\cdot)$ is the sign function and $P'_+(\beta)$ denotes the right derivative of $P(\cdot)$ at $\beta$.
Interestingly, with $\alpha_j = |P'_+(\beta_j^{(t)})|$, the $\bm\Delta_1$-based surrogate \eqref{LLA-g} can be shown to be
\begin{align*}
l(\bm\beta) + \sum_j \big[P(|\beta_j^{(t)}|) + P'_+(|\beta_j^{(t)}|)(|\beta_j| - |\beta_j^{(t)}|)\big],
\end{align*}
    which is exactly the surrogate  constructed by Zou and Li. To the best of our knowledge, the generalized Bregman formulation is new.

LLA requires solving a weighted lasso problem at each step. We can further linearize $l$  as in Example \ref{ex:thres} to improve its scalability. LLA is popular among statisticians, but to our knowledge, there is a lack of  \textit{global} convergence-rate studies   in large-$p$ applications. We will see that reformulating LLA from the generalized Bregman surrogate perspective leads to a convenient choice of the convergence measure in analyzing the algorithm.
\end{ex}


\begin{ex} \upshape
\textbf{(Sigmoidal regression)}.
We use the univariate-response sigmoidal regression to illustrate this type of nonconvex problems that is commonly seen in artificial neural networks. The formulation carries over to multilayered networks and recurrent networks \citep{She2014NN}.

Let $\bm X = [\bm x_1,\bm x_2, \ldots, \bm x_n]^\top \in \mathbb R^{n\times p}$ be the data matrix, and $\bm y = [y_1,\cdots, y_n]^\top$ be the response vector. Define $\pi(\nu) = e^\nu/(1+e^\nu)$; if $\nu$ is replaced by a vector, $\pi$ is defined componentwise. The sigmoidal regression solves
\begin{equation}
\min_{\bm\beta}\  f(\bm\beta) = \frac{1}{2}\sum_{i=1}^n (y_i - \pi(\bm x_i^\top\bm\beta))^2.
\end{equation}
Then $\nabla^2 f(\bm\beta) = \sum_{i=1}^n [(-2\mu_i^3+3\mu_i^2-\mu_i)y_i + (3\mu_i^4 - 5\mu_i^3+2\mu_i^2)] \bm x_i\bm x_i^\top$, where $\mu_i = \pi(\bm x_i^\top\bm\beta)$.
Because $\mu_i\in[0,1]$, we get $\nabla^2 f(\bm\beta) \preceq \bm X^\top\text{diag}\{|0.1y_i| + 0.08\}_{i=1}^n\bm X$,
which motivates a Bregman surrogate
{\begin{equation*}
g(\bm{\beta};\bm{\beta}^{(t)}) = f(\bm{\beta}) + \mathbf D_{\psi-f}(\bm{\beta},\bm{\beta}^{(t)}),~ \psi(\bm{\beta}) = \frac{1}{2}\bm{\beta}^{\top} \bm X^\top\text{diag}\{|0.1y_i| + 0.08\}\bm X \bm\beta.
\end{equation*}}Solving $\min_{\bm\beta} g(\bm\beta;\bm\beta^{(t)})$ yields $\bm\beta^{(t+1)} = \bm\beta^{(t)} + \bm B^{-1}\bm X^\top(\bm u^{(t)} - \bm u^{(t)} \circ \bm u^{(t)}) \circ (\bm y - \bm u^{(t)})$, where $\bm B =  \bm X^\top\text{diag}\{|0.1y_i| + 0.08\}_{i=1}^n\bm X$, $\bm u^{(t)} = \pi(\bm X^\top\bm\beta^{(t)})$ and $\circ$ denotes the Hadamard product.
This type of surrogate functions is closely related to proximal Newton-type methods \citep{Schmidt2010Thesis} and signomial programming \citep{Lange2014}.
\end{ex}

\section{Bregman-surrogate algorithm analysis} \label{sec:results}

Motivated by the examples in Section \ref{sec:Bregman}, we study a generalized Bregman-surrogate algorithm family for solving $\min_{\bm\beta} f(\bm\beta)$, with the sequence of iterates defined by
\begin{equation} \label{BregIter}
\bm\beta^{(t+1)} \in \mathop{\arg\min}_{\bm\beta} \  g(\bm\beta;\bm\beta^{(t)}) := f(\bm{\beta}) + \bm\Delta_\psi(\bm{\beta},\bm\beta^{(t)}), \ t \ge 0
\end{equation}
The
objective function $f$ and the auxiliary function $\psi$ are  assumed to be directionally differentiable but need not be smooth or convex.    $\psi$ has flexible options  as seen from the previous examples.

 Equation \eqref{BregIter} does not necessarily give an MM procedure, as the majorization condition $g(\bm\beta;\bm\beta^-) \ge f(\bm\beta)$ may not hold.   But we have the following zeroth-order \textit{and} first-order degeneracies when $\bm\beta^-=\bm\beta$, which provides  rationality of investigating the accuracy of \textit{fixed points}  under the $g$-mapping  \eqref{BregIter}.

\begin{lemma} \label{lemma:degeneracy}
Let $g(\bm\beta;\bm\beta^-) = f(\bm\beta) + \bm\Delta_\psi(\bm\beta,\bm\beta^-)$ with $f$ and $\psi$ directionally differentiable. Then (i) $g(\bm\beta;\bm\beta) = f(\bm\beta)$, and (ii) $\delta g(\bm\beta;\bm\beta^-,\bm h)|_{\bm\beta^-=\bm\beta} = \delta f(\bm\beta;\bm h), \forall \bm\beta,\bm h$, where $\delta g(\bm\beta;\bm\beta^-,\bm h)$ is the directional derivative  of $g(\,\cdot\,; \bsb{\beta}^-)$ at $\bm\beta$ with increment $\bm h$.
\end{lemma}

The lemma relates the set of fixed points of the algorithm mapping,
\begin{align} \label{fpsetdef}
\{\bsb{\beta}: \bsb{\beta}\in \arg\min_{\bsb{\beta}}   g(\bm\beta;\bm\beta^-)|_{\bm\beta^-=\bm\beta}\},
\end{align}
 which we will call the fixed points of $g$ for short,  to the set of directional stationary points of $f$ (under directional differentiability),
\begin{align}
\{\bm\beta: \delta f(\bm\beta;\bm h) \ge 0 \text{ for any admissible } \bm h\},
\end{align} which becomes  the set of stationary points when $f\in \mathcal C^{1}$. The link is general for any generalized Bregman surrogate in \eqref{BregIter}   \textit{regardless} of the specific form of $\psi$.
An important implication is that     in studying  convergence it is  legitimate to  measure   how    $\bsb{\beta}^{(t+1)}$ and $\bsb{\beta}^{(t)}$ differ, as widely used in practice.   Later we will see that it is indeed possible to provide provable guarantees for the fixed points of this type of surrogates.
In contrast, a general MM algorithm does not always  have the first-order degeneracy and so attaining $\bsb{\beta}^{(t+1)}=\bsb{\beta}^{(t)}$ does   not necessarily ensure a good-quality solution,   especially in nonconvex scenarios.


\subsection{Computational accuracy}\label{subsec:comp}

We first study the optimization error of  \eqref{BregIter}, then turn to its statistical error  in Section \ref{subsec:stat}. This subsection aims to derive universal rates of convergence under no regularity conditions.

\subsubsection*{$\bullet$ General setting}

In this part, the objective $ f(\bm\beta)$  does not have any known structure. To better connect with some conventional results  in convex optimization, we first present two propositions for \eqref{BregIter} on the function-value convergence and  iterate convergence. While the resultant rates are encouraging, the error bounds are most informative under certain smoothness and convexity assumptions. This  suggests the necessity of choosing a proper convergence measure in order to avoid stringent or awkward technical conditions  in nonconvex optimization.


\begin{pro}\label{th:comp_funcval}   Given an arbitrary initial point $\bm{\beta}^{(0)}$, let $\bm\beta^{(t)}$ be the sequence generated according to \eqref{BregIter}  where  $\psi$ is differentiable.
Then
\begin{equation}\label{th:comp_diff_bound-0}
\mathop{\avg}_{0\le t\le T}f(\bm{\beta}^{(t+1)}) - f(\bar{\bm{\beta}}) \leq \frac{1}{T+1}[\bm\Delta_\psi(\bar{\bm{\beta}},\bm{\beta}^{(0)}) - \bm\Delta_\psi(\bar{\bm{\beta}}, \bm{\bm{\beta}}^{(T+1)})]
\end{equation}for any $\bar{\bm{\beta}}$ satisfying
\begin{equation}\label{th1-condition}
\bm\Delta_\psi(\bm{\beta}^{(t+1)},\bm{\beta}^{(t)}) + \bm{\Delta}_f(\bar{\bm{\beta}},\bm{\beta}^{(t+1)}) \geq 0,~0\le t\le T.
\end{equation}
Here, $\mathop{\avg}_{0\le t\le T}f(\bm{\beta}^{(t+1)})$ denotes the average of $f(\bm\beta^{(1)}),\ldots, f(\bm\beta^{(T+1)})$.

In particular, if both $f$ and $\psi$ are convex, then $f(\bm\beta^{(t)})$ is nonincreasing and  \begin{equation}\label{th:comp_diff_bound-0-cvx}
f(\bm{\beta}^{(T+1)}) - f(\bm{\beta}) \leq \frac{\bm\Delta_\psi(\bm{\beta},\bm{\beta}^{(0)})}{T+1}, \ \forall \bsb{\beta}.
\end{equation}
\end{pro}


  Equation \eqref{th:comp_diff_bound-0}  shows a convergence rate of  $\mathcal O(1/T)$ under   \eqref{th1-condition} that  amounts to step size control. For example, for $\bm\Delta_\psi= \rho\mathbf D_\varphi - \bm\Delta_f$  in mirror descent,  \eqref{th1-condition} shows that $\rho$ should be sufficiently large, which in turns gives a small stepsize   $1/\rho$:
  $$\rho \ge (\bm\Delta_f(\bm\beta^{(t+1)},\bm\beta^{(t)}) - \bm\Delta_f(\bar{\bm{\beta}}, \bm\beta^{(t+1)}))/\mathbf D_\varphi(\bm\beta^{(t+1)}, \bm\beta^{(t)}),$$
or $\rho \ge \bm\Delta_f(\bm\beta^{(t+1)},\bm\beta^{(t)})/\mathbf D_\varphi(\bm\beta^{(t+1)}, \bm\beta^{(t)})$ when $f$ is convex.  In nonconvex scenarios,   the condition may be hard to verify, but  one has reason to believe that with a properly small step size, a generalized Bregman-surrogate  algorithm   should not be much slower than gradient descent. 

Actually, a faster rate of convergence may be obtained under some GBF comparison conditions,  \eqref{strcvx-cond-2} and \eqref{strcvx-cond-1} below, which can be viewed as substitutes for conventional    strong  convexity  in a more general sense. (The corresponding geometric decay  of the errors is  motivating in high dimensional statistical learning, in light of  the ``restricted'' strongly convexity   often possessed by  such a type of problems  \cite{Loh2015}.)

\begin{pro} \label{th:strongcvx}
Consider the iterative algorithm defined by \eqref{BregIter}   starting at an arbitrary point $\bm\beta^{(0)}$ with  $\psi$  differentiable, and let $\bm\beta^o$ be a  minimizer of $f(\bm\beta)$.
(i) If for some $\kappa > 1$,
 $\bm\Delta_\phi = \bm\Delta_\psi+\bm\Delta_f$ satisfies\begin{equation} \label{strcvx-cond-2}
\sym{\bm\Delta}_\phi \ge \frac{\kappa}{\kappa-1}\bm\Delta_\psi,
\end{equation}
 then for any $T\ge 0$, we have
\begin{equation} \label{strcvx-res-2}
\sym{\bm\Delta}_\phi(\bm\beta^o,\bm\beta^{(T+1)}) \le \Big(\frac{\kappa-1}{\kappa+1}\Big)^{T+1} \sym{\bm\Delta}_\phi(\bm\beta^o,\bm\beta^{(0)}) - \frac{\kappa}{2} \min _{0\le t \le  T} \breg_{\psi} (\bsb{\beta}^{(t+1)},\bsb{\beta}^{(t)} ).
\end{equation}
(ii) Alternatively, if \begin{equation}\label{strcvx-cond-1}
2\sym{\bm\Delta}_f \ge \varepsilon \bm\Delta_\psi
\end{equation}
for some $\varepsilon > 0$, then
\begin{equation} \label{strcvx-res-1}
\bm\Delta_\psi(\bm\beta^o, \bm\beta^{(T+1)}) \le \Big(\frac{1}{1+\varepsilon}\Big)^{T+1}\bm\Delta_\psi(\bm\beta^o,\bm\beta^{(0)}) -\frac{1}{\varepsilon} \min _{0\le t \le  T} \breg_{\psi} (\bsb{\beta}^{(t+1)},\bsb{\beta}^{(t)} )
\end{equation}
for any $T \ge 0$.
\end{pro}

\begin{remark} 
We give an illustration of  (i) and (ii)  to compare their assumptions and conclusions. In  gradient descent with $\bm\Delta_\phi = \rho\mathbf D_2$,  \eqref{strcvx-cond-2} becomes $\rho\mathbf D_2 \ge (\rho\mathbf D_2 - \bm\Delta_f)\kappa/(\kappa-1)$ or $\bm\Delta_f \ge (\rho/\kappa)\mathbf D_2$ and when $f$ is $\mu$-strongly convex    and   $\rho$-strongly smooth,  $\kappa = \rho/\mu  $. Then  \eqref{strcvx-res-2} reads
\begin{equation}  \label{strcvx-gd}
\mathbf D_2(\bm\beta^o,\bm\beta^{(T+1)}) \le \Big(\frac{\rho-\mu}{\rho+\mu}\Big)^{T+1}\mathbf D_2(\bm\beta^o,\bm\beta^{(0)}).
\end{equation}
The $\Breg_2$-form bound is classical for problems with strong convexity; see, for example,  Theorem 2.1.15 in \cite{Nesterov2004book}. Yet  it is worth mentioning that our Bregman comparison conditions do not require $\psi$ to be \emph{strongly} convex to attain the linear rate.
\eqref{strcvx-res-1} gives a linear convergence result, too, in terms of yet another measure. In the same setup,    \eqref{strcvx-cond-1}   holds for $\varepsilon: \varepsilon\rho/(2+\varepsilon) = \mu$  and  similarly \begin{equation}\bm\Delta_\psi(\bm\beta^o,\bm\beta^{(T+1)}) \le \Big(\frac{ \rho-\mu }{ \rho+\mu}\Big)^{T+1}\bm\Delta_\psi(\bm\beta^o,\bm\beta^{(0)}).\end{equation} A careful examination of the proof in Section \ref{subsec:proofofscvx}   shows that \eqref{strcvx-cond-1} is applied once, while \eqref{strcvx-cond-2} is applied twice on both sides of \eqref{strongcvx-3}, and so     (ii) appears   less technically demanding.  Picking  a suitable error function can    assist   analysis and relax  regularity assumptions. The same  $ \bm\Delta_\psi $ will be used in studying  the statistical error convergence in Theorem \ref{th:stat_bound}. \end{remark}

Instead of naively comparing $f(\bsb{\beta}^{(t)})$ with $f^o$, or $ \bsb{\beta}^{(t)} $ with $ \bsb{\beta}^{o} $, which may be unattainable or nonunique  in    nonconvex optimization,   one can   measure the algorithm convergence in a  wiser manner. Ben-Tal and Nemirovski  \cite{Nemirovski_Notes} pointed out that with an inappropriate  measure of discrepancy, the convergence rate of gradient descent for minimizing a nonconvex objective can be arbitrarily slow, and a common choice is to bound      \begin{align}\label{gradmeasureBN}\mathop{\min}_{ t\leq T}\|\nabla f(\bm{\beta}^{(t)})\|^2.\end{align} This is reasonable  since when $\nabla f(\bm{\beta}^{(t)}) = 0$, gradient descent stops iterating and delivers a stationary point. \eqref{gradmeasureBN} can be rewritten  as $\rho^2$ times
\begin{align}
\mathop{\min}_{ t\leq T} \Breg_2 (\bm{\beta}^{(t+1)}, \bm{\beta}^{(t)}) \end{align}
as   $\bm\beta^{(t+1)} - \bm\beta^{(t)} = -\nabla f(\bm\beta^{(t)})/\rho$.  The idea of checking stationarity by the difference between two successive iterates generalizes, thanks to  Lemma \ref{lemma:degeneracy}, and eventually leads to an error bound  that  can get rid of  condition  \eqref{th1-condition}.

\begin{thm}\label{th:comp_nonconvex}
Any generalized Bregman surrogate algorithm defined by \eqref{BregIter} satisfies the following bound for all $ T\ge 1$,
\begin{equation}\label{cor:comp_diff_bound-1}
\mathop{\avg}_{0\leq t \leq T}(2\sym{\bm\Delta}_\psi + \bm{\Delta}_f)(\bm{\beta}^{(t)}, \bm{\beta}^{(t+1)}) \leq \frac{1}{T+1}\big[f(\bm{\beta}^{(0)}) - f(\bm{\beta}^{(T+1)})\big].
\end{equation}
\end{thm}

\eqref{cor:comp_diff_bound-1} obtains the same   rate of convergence as  Proposition \ref{th:comp_funcval}, but  is \textit{free} of any conditions other than directional differentiability, because   only the weak idempotence is needed to derive the bound. A proper stepsize control can often make the GBF error   nonnegative (e.g., \eqref{ITsschoice}). But even when $\bsb{\beta}^{(t)}$ diverges, (45) still applies.

Notice the factor `2' proceeding the symmetrized Bregman $\sym{\bm\Delta}_\psi$ on the left-hand side of \eqref{cor:comp_diff_bound-1}.  This gives a relaxed stepsize control than  MM. We use  mirror descent        $\bm\Delta_\psi = \rho\mathbf D_\varphi - \bm\Delta_f$  to exemplify the point without requiring   $f$ to be convex,  cf. Example \ref{ex:mirrdes}.  
 \begin{cor} \label{cor-mirror}
In the  mirror descent setup with a possibly nonconvex objective, suppose that  $\breg_f \le L \sym\Breg_{\varphi}$ for some $L>0$, $\inf_{\bsb{\beta}} f(\bsb{\beta}) \ge 0$, and  the inverse stepsize parameter $\rho$ is taken such that  $\rho > L/2$. Then  any  accumulation point of $\bsb{\beta}^{(t)}$ is a fixed point of $g$ and
\begin{align}
\mathop{\avg}_{0\leq t \leq T}\sym{\mathbf D}_\varphi (\bm{\beta}^{(t)}, \bm{\beta}^{(t+1)}) \leq \frac{f(\bm{\beta}^{(0)})}{(T+1)(2\rho-L)}.\label{mirrdes-genconvrate}
\end{align}
 \end{cor}
Hence in the special case of gradient descent, \eqref{mirrdes-genconvrate} recovers    $\min_{0\le t\le T}\allowbreak\|\nabla f(\bm\beta^{(t)})\|_2^2 = \mathcal O(1/T)$   \citep{Nemirovski_Notes} when $\rho>L/2$.
 In comparison, MM algorithms always require  $\bm\Delta_\psi \ge 0$, or $\rho \ge L$. A smaller value of $\rho$ means a larger step size with which the algorithm converges faster. 

\subsubsection*{$\bullet$ Composite setting}

High-dimensional statistical learning  often has an additive objective $f(\bm\beta)= l_0(\bm X\bm\beta) + P(\varrho\bm\beta;\lambda)$, where $\bm X\in \mathbb R^{n\times p}$ is the predictor or feature matrix, $l_0(\cdot)$ is the loss  defined on $\bm X\bm\beta$ (and so $l(\bm\beta) = l_0(\bm X\bm\beta)$),  $P(\cdot; \lambda)$   is a sparsity-inducing regularizer and      $\varrho$ is   a controllable parameter,  typically taking     $ \|\bm X\|_2$ to match the scale. Unless otherwise mentioned,     $P( \bm\beta;\lambda)$ denotes $\sum_j P( \beta_j; \lambda) $     with a little abuse of notation.

Such a composite setup is widely assumed in convex optimization \cite{Tseng2008,Duchi2011}. But among the abundant choices of $l_0$ and $P$ in the literature, many of them are nonconvex. The good news is that   the main theorem proved in the previous subsection adapts to the composite setting and  we give some results for iterative thresholding  and LLA as an illustration (cf. Examples \ref{ex:thres}, \ref{ex:LLA}).

\vspace{2ex}
\noindent \textit{Iterative thresholding}.
Many popularly used penalty functions are associated with thresholdings rigorously defined  as follows. 

\begin{defn}[\textbf{Thresholding function}] \label{def:Theta}
A threshold function is a real-valued function $\Theta(t;\lambda)$ defined for $-\infty < t <\infty$ and $0\leq \lambda < \infty$ such that
(i) $\Theta(-t;\lambda) = -\Theta(t;\lambda)$;
(ii) $\Theta(t;\lambda) \leq \Theta(t';\lambda)$ for $t\leq t'$;
(iii) $\mathop{\lim}_{t\rightarrow\infty}\Theta(t;\lambda) = \infty$;
(iv) $0\leq \Theta(t;\lambda) \leq t$ for $0\leq t < \infty$.
\end{defn}

Given $\Theta$, a critical concavity number ${\mathcal L}_{\Theta} \le 1$ can be introduced such that $\rd \Theta^{-1}(u;\lambda)\rd u \ge 1 - \mathcal L_\Theta$ for almost every $u \ge 0$, or
\begin{equation} \label{def:LTheta}
{\mathcal L}_{\Theta} = 1- \mathrm{ess\,inf}\{ \mathrm{d} \Theta^{-1}(u;\lambda)/\mathrm{d} u: u \ge 0\},
\end{equation}
with ess\,inf  the essential infimum and $\Theta^{-1}(u;\lambda) := \sup\{t:\Theta(t;\lambda)\le u\}, \forall u>0$. For the widely used soft-thresholding $\Theta_S(t;\lambda) = \sgn(t)(|t|-\lambda)1_{|t|>\lambda}$ and hard-thresholding
$\Theta_H(t;\lambda) = t1_{|t|>\lambda}$, $\mathcal L_\Theta$ equals $0$ and $1$, respectively. In fact, when $\mathcal L_\Theta > 0$,    the penalty induced by $\Theta$ via \eqref{pendef} is nonconvex, and $\mathcal L_\Theta$ gives a concavity measure of it according to Lemma \ref{lemma:Lp}.
The Bregman surrogate characterization of iterative thresholding in \eqref{IterThres-g} yields a general conclusion for any $\Theta$ in possibly high dimensions.

\begin{pro} \label{pro:PTheta}
Given any  thresholding  $\Theta$ and  directionally differentiable  $l(\cdot)$, consider  the iterative thresholding procedure \eqref{Thres-update}: $\bm\beta^{(t+1)} = \Theta(\varrho\bm\beta^{(t)} - \nabla l(\bm\beta^{(t)})/\varrho; \lambda)/\varrho$ with $\varrho>0$. Construct \begin{equation} \label{pendef}
P_{\Theta}(t; \lambda)=\int_0^{|t|} (\Theta^{-1}(u;\lambda) - u) \rd u, ~\forall t \in \mathbb R,
\end{equation}
and define $f(\bm\beta) = l(\bm\beta) + P_\Theta(\varrho\bm\beta;\lambda)$,  $g(\bsb{\beta}, \bsb{\beta}^-) = l(\bm\beta) +  P_\Theta(\varrho\bsb{\beta};\lambda) + (\varrho^2\mathbf D_2 - \bm\Delta_l)(\bm\beta,\bm\beta^-)$. 
Then   $ \bm\beta^{(t )} \in  \mathop{\arg\min}_{\bm\beta} g(\bm\beta, \bm\beta^{(t-1)}) $ and for all $T\ge 1$
\begin{equation}\label{th:Theta_comp_grad_bound-1}
\mathop{\avg}_{0\leq t \leq T}(\varrho^2(2 - \mathcal L_
\Theta)\mathbf D_2 -\back {\bm\Delta}_l)(  \bm{\beta}^{(t)}, \bm{\beta}^{(t+1)}) \leq \frac{1}{T+1}\big[f(\bm{\beta}^{(0)}) - f(\bm{\beta}^{(T+1)})\big].
\end{equation}
\end{pro}

When the loss satisfies $\breg_l \le L \Breg_2$, 
a reasonable choice of  $\varrho$ is \begin{equation}\label{ITsschoice}\varrho^2> L/(2-\mathcal L_\Theta).\end{equation} So when $\mathcal L_\Theta>0$, the step size upper bound will be  smaller than that  as $\mathcal L_\Theta = 0$. This   is often the price to pay for nonconvex optimization. On the other hand, \eqref{th:Theta_comp_grad_bound-1} still  ensures the universal rate of convergence of $\mathcal O(1/T)$, in spite of the high dimensionality and nonconvexity.

\vspace{2ex}
\noindent \textit{Local linear approximation}.
Next, we study the computational convergence of LLA  for solving the penalized estimation problem $\min f(\bsb{\beta})=  l(\bm\beta)+P(\varrho\bm\beta)$, assuming  $l$ is directionally differentiable, $P(0) = 0$,  $P'_+(0)< +\infty$,  $P(t)=P(-t)\ge 0$ and  $P(t)$ is differentiable for any $t>0$.
Recall its Bregman form  surrogate
\begin{equation} \label{LLA-surro}
g^{(t)}_\mathrm{LLA}(\bm\beta;\bm\beta^{(t)}) = l(\bm\beta) + P(\varrho\bm\beta) + \bm\Delta_{\|\bm\alpha^{(t)}\text{\tiny$\circ$}(\cdot)\|_1 - P(\cdot)}(\varrho\bm\beta,\varrho\bm\beta^{(t)}),
\end{equation}
where $\bm\alpha^{(t)} = [\alpha_j^{(t)}]$ with $\alpha_j^{(t)} = |P'_+(\beta_j^{(t)})|, 1\le j\le p$.
We abbreviate {\small$\bm\Delta_{\|\bm\alpha^{(t)}\text{\tiny$\circ$}(\cdot)\|_1 - P(\cdot)}$} to {\small$\bm\Delta_{\mathrm{LLA}}^{(t)}$}, which  does not satisfy strong idempotence. By combining  {\small$\sym{\bm\Delta}_{\mathrm{LLA}}^{(t)}$} and $\bm\Delta_f$  to evaluate LLA's optimization error, we  obtain a  convergence result without any additional assumptions.

\begin{pro}\label{pro:LLA_comp}
Given any starting point $\bm{\beta}^{(0)}$,  the LLA iterates satisfy the following bound for all $T \ge 1$:
{\begin{equation*} 
\mathop{\avg}_{0\leq t \leq T} [2\sym{\bm\Delta}_{\mathrm{LLA}}^{(t)}(\varrho\bm{\beta}^{(t)}, \varrho\bm{\beta}^{(t+1)}) + \bm{\Delta}_f(\bm{\beta}^{(t)}, \bm{\beta}^{(t+1)})]  \leq \frac{1}{T+1}[f(\bm{\beta}^{(0)}) - f(\bm{\beta}^{(T+1)})].
\end{equation*}}
\end{pro}

Ignoring the cost difference per iteration, the convergence rate of LLA  is no slower than that of gradient descent.   If  $l$ is a negative log-likelihood function associated with a log-concave density and $P$ is concave on $\mathbb R_+$, as assumed in \cite{Zou2008LLA},  $2\sym{\bm\Delta}_{\mathrm{LLA}}^{(t)}(\varrho\bm{\beta}, \varrho\bm{\beta}') + \bm{\Delta}_f(\bm{\beta}, \bm{\beta}') = \bm\Delta_l(\bm\beta,\bm\beta') + \bm\Delta_{-P}(\varrho\bm\beta',\varrho\bm\beta) + 2\sum_j\alpha_j^{(t)}\sym{\bm\Delta}_1(\varrho\beta_j,\varrho\beta_j') \ge 0, \forall \bm\beta, \bm\beta'  $.   But Proposition  \ref{pro:LLA_comp}  holds even when  $P$ is nonconcave on $\mathbb R_+$ and $l$ is nonconvex.

The global convergence-rate results presented in this subsection are   free of any  regularity conditions on sparsity, sample size, initial point and   design incoherence. High-dimensional learning algorithms may however   show  a better convergence rate  when the problems under consideration are ``regular'' in a certain sense. 

\subsection{Statistical accuracy}\label{subsec:stat}

To statisticians,  the  statistical accuracy of Bregman-surrogate algorithms with respect to  a  statistical truth  (denoted by $\bm\beta^*$)  is  perhaps more meaningful   than the optimization error to a certain local or global minimizer, since real world data are always noisy.
 Section \ref{subsec:fixed}  and   Section \ref{sec:stat-seq} will study  the statistical error of the final estimate $\hat{\bm\beta}$  and   the $t$-th iterate $\bm\beta^{(t)}$, respectively, where  combining the generalized Bregman calculus    and the empirical process theory eases the treatment of  a nonquadratic loss.

The  techniques based on GBFs  apply to a  general problem (see, e.g., {Theorem \ref{theorem:genalgananal}} in  Section \ref{subsec:genopt}), but here we  focus on the aforementioned sparse learning   in the composite setting:
$\min_{\bm\beta} l(\bm\beta) + P_\Theta(\varrho\bm\beta;\lambda)$, where    $l(\bm{\beta}) =l_0(\bm{\eta})= l_0(\bm{X\beta})$ is  directionally differentiable and  $P_{\Theta}(\cdot;\lambda)$ is induced by a thresholding $\Theta$ via  \eqref{pendef}.  Since $l_0$ is placed on $\bsb{X}\bsb{\beta}$, we include here a scaling parameter $\varrho$ (often $\| \bsb{X}\|_2$) in the penalty; this  will yield a universal choice  of the regularization parameter  $\lambda$ that does not vary with the sample size.  Throughout Section \ref{subsec:stat}, we assume that $\varrho$ satisfies $ \varrho\ge\|\bsb{X}\|_2$. Note that neither the loss nor the penalty needs to be convex or smooth.

Give any directionally differentiable  $\psi$,  the sequence of iterates is generated by
\begin{equation} \label{P_iterate}
\bm\beta^{(t+1)} \in \mathop{\arg\min}_{\bm\beta} g(\bm\beta;\bm\beta^{(t)}):= l(\bm\beta) + P_\Theta(\varrho\bm\beta;\lambda) + \bm\Delta_\psi(\bm\beta,\bm\beta^{(t)}).
\end{equation}
Nonconvex iterative thresholding 
 and LLA are particular instances.

First, we must characterize the notion of  noise in this nonlikelihood setting, to take  into account the randomness of samples.
Assume $l_0$ is differentiable at point  $\bm X\bm\beta^*$ (but    not necessarily differentiable on all of $\mathbb R^n$)   and define the \textit{effective noise} by
\begin{equation} \label{noise-def}
\bm\epsilon = -\nabla l_0(\bm X\bm\beta^*).
\end{equation}
(An alternative assumption is  that  $\delta l_0(\bm X\bm\beta^*;\bm h)$ is a sub-Gaussian random variable with mean 0 and scale bounded by $c\sigma$ for any unit vector $\bm h$,   but we will not pursue further in the current paper.)

Typically,   $\mathbb E [\bm\epsilon] $ should be $0$, and so $\nabla \{\mathbb E[ l_0(\bm X\bm\beta^*)]\} = 0$ assuming the differentiation and expectation are exchangeable, which means  the statistical truth  makes the gradient of its risk vanish.
For a GLM  with $y_i$ ($1\le i \le n$) following a distribution in the exponential family that has cumulant function $b$ and canonical link function $g = (b')^{-1}$, the loss is then $l(\bm\beta) = l_0(\bm X\bm\beta) = - \langle\bm y, \bm X\bm\beta\rangle + \langle \bm 1,b(\bm X\bm\beta)\rangle$ (cf. \eqref{l0glmass} with $\sigma=1$), and so
\begin{equation} \label{def:noise}
\bm\epsilon = \bm y - g^{-1}(\bm X\bm\beta^*) = \bm y - \mathbb E(\bm y).
\end{equation}
Our effective noise, as a joint outcome of the loss and the response,  does not depend on the regularizer, and may differ from the raw noise.
For example, under $\bm y = \bm X\bm\beta^* + \bm\epsilon^{\mathrm{raw}}$,
$l(\bsb{\beta}) = l_{\mbox{\tiny Huber}}(\bm r) = \sum_{i:|r_i|\le a \sigma}r_i^2/2 + \sum_{i:|r_i|> a\sigma}(a|r_i|-a^2\sigma^2/2)$ with $\bm r = \bm y - \bm X\bm\beta$ 
\citep{huber1981book}, simple calculation gives $\epsilon_i = \epsilon_i^{\mathrm{raw}}1_{|\epsilon_i^{\mathrm{raw}}|\le a\sigma} + a \sigma 1_{|\epsilon_i^{\mathrm{raw}}|> a\sigma}$, which is   bounded by $ a\sigma$,  thereby sub-Gaussian, no matter what distribution     the raw noise follows. 
This nonparametricness is apparent for   any $l_0$ that is (globally) Lipschitz, for example,  the logistic deviance and  hinge loss for classification.

In this section, we assume that $\bm{\epsilon}$ is a  {sub-Gaussian} random vector with mean zero and scale bounded by $\sigma$, cf.  Definition \ref{def:subgauss}, where     $\epsilon_i$ are not required to be independent.
Examples include Gaussian random variables and bounded random variables such as Bernoulli. 

The support of $\bm{\beta}$ is denoted by $\mathcal J(\bm{\beta})= \{j: \beta_j \neq {0}\}$, and its cardinality is $J(\bm{\beta})=|\mathcal J(\bm{\beta})| = \|\bm{\beta}\|_0$. We abbreviate $J(\bm\beta^*)$ to $J^*$ and $J(\hat{\bm\beta})$ to $\hat J$.
In sparse learning, $J^*\ll n \ll p$   is typically true. The sparsity suggests the possibility of  obtaining a fast rate of convergence in statistical error.
The following penalty induced by the hard-thresholding $\Theta_H(t;\lambda) = t1_{|t|>\lambda}$ by \eqref{pendef} turns out to play a key role in the analysis\begin{equation} \label{def:PH}
P_H(t;\lambda) = (-t^2/2+\lambda|t|) 1_{|t|<\lambda} + (\lambda^2/2) 1_{|t|\ge \lambda}.
\end{equation}
An important fact is that $P_\Theta(t;\lambda)\ge P_H(t;\lambda)$ for any $t\in\mathbb R$ and any thresholding rule $\Theta$. This is simply because in shrinkage estimation, any  $\Theta(t;\lambda)$ with  $\lambda$ as the threshold is identical to zero as $ t\in [0, \lambda)$  and is bounded above by the identity line  for $t\ge \lambda$. 

\subsubsection{Statistical accuracy of fixed-point solutions}
\label{subsec:fixed}

The finally obtained solutions from a Bregman surrogate algorithm can be described as the fixed points of $g$ (recall \eqref{fpsetdef}),\begin{equation}\label{fixed-point}
\hat{\bm\beta} \in \mathop{\arg\min}_{\bm\beta} g(\bm\beta;\hat{\bm\beta}).
\end{equation}
We denote the set by   $\mathcal F$, and call such solutions the \textit{$F$-estimators}. When the objective function is convex,  an F-estimator   is necessarily  a globally optimal solution to the original problem by  Lemma \ref{lemma:degeneracy}, thus an M-estimator.
In general, however,  the lack of convexity and smoothness may make    $\hat {\bsb{\beta}}$  neither an M-estimator nor a Z-estimator \citep{van1996weak}, which  poses new and intriguing challenges to    statistical algorithmic  analysis. It is also worth mentioning that another important class of ``A-estimators'' that have \textit{alternative} optimality, typically arising from block coordinate descent (BCD) algorithms like in    Example \ref{ex:NMF}, can often be converted to F-estimators; see Section \ref{subsec:Aests}.

Nicely, if the problem is   regular, all  F-estimators defined through $g$ can       achieve essentially the best statistical precision in possibly high dimensions.  This is  nontrivial    since   even $f$'s locally optimal solutions do {not} all have the provable guarantee (cf. Remark \ref{fixedvslocal}).
   Theorem \ref{thm:errrate} and  Theorem \ref{thm:ora}   below only make use of the weak idempotence property; another notable feature is that  the conditions and conclusions below are \textit{regardless} of the form of     $\breg_{\psi}$. 

\begin{thm} \label{thm:errrate}
  Suppose there exist $\delta>0$, $\vartheta>0$ and large enough $K\geq 0$ so that the following inequality holds for any $\bm\beta \in\mathbb R^p$:
\begingroup
  \singlespacing
{\small
\begin{equation}
\label{R_0-errrate}
\begin{split}
&\varrho^{2}\mathcal L_\Theta\mathbf D_2( \bm{\beta}, \bm{\beta}^*) + \delta\mathbf D_2(\bm{X\beta},\bm{X\beta}^*) + \vartheta P_H(\varrho(\bm{\beta} - \bm{\beta}^*);\lambda)  +  P_\Theta(\varrho\bm{\beta}^*;\lambda) \\
\le\,& 2\sym{\bm{\Delta}}_{l}(\bm{\beta},\bm{\beta}^*)  + P_\Theta(\varrho\bm{\beta};\lambda) + K\lambda^2J(\bsb{\beta}^*), \end{split}
\end{equation}
}\endgroup
where  $\lambda = A\sigma\sqrt{\log(ep)}/\sqrt{(\delta\wedge\vartheta )\vartheta}$ with $A$ a sufficiently large constant.
Then
\begingroup
  \singlespacing
{\small
\begin{align}
&\Breg_2(\bm{X} \hat {\bsb{\beta}},\bm{X}\bsb{\beta}^*)  \le  \frac{2KA^2}{(\delta\wedge\vartheta )\delta\vartheta}\sigma^2 J^*  \log(ep), \label{esterrbnd-1}\\
&P_H(\varrho  (\hat {\bsb{\beta}} - \bsb{\beta}^*); \lambda)   \le  \frac{4KA^2}{(\delta\wedge\vartheta )\vartheta^2}\sigma^2 J^*  \log(ep), \label{esterrbnd-2}
\end{align}}
\endgroup
with probability at least $1-Cp^{-cA^2}$,
where $C,c$ are  positive constants.
\end{thm}


Moreover, an     \textit{oracle inequality} \citep{Donoho1994, Koltchinskii2011book} can be built to justify the estimators  even when   $ {\bsb{\beta}^*}$ is not exactly sparse. Toward this goal, recall the notion of a pseudo-metric $d$ (cf. Definition \ref{def:pseudometric}), that is,  $d$ is nonnegative, symmetric, and satisfies the triangle inequality, and  suppose  without loss of generality that $$\alpha d^2(\bsb{\eta}, \bsb{\eta}') \le \breg_{l_0}(\bsb{\eta}, \bsb{\eta}') \le L d^2(\bsb{\eta}, \bsb{\eta}' ), \forall \bsb{\eta}, \bsb{\eta}'$$ for some pseudo-metric $d$ with $-\infty\le \alpha \le  L\le +\infty$. For  regression $l(\bsb{\beta}) =l_0(\bsb{\eta} )= \| \bsb{y } -\bsb{\eta}\|_2^2/2$,  $\alpha =  L=1>0 $.

\begin{thm} \label{thm:ora}
Assume   for given $\bm\beta\in\mathbb R^p$, there exist   $r$:    $ 0\le r <1,   \alpha r /{L} \ge 0$, positive $\delta$, $\vartheta$,  and a large enough $K\geq 0$ so that
{\small\begin{align}\label{R_0-oracond-general}
\begin{split}
&\varrho^2\mathcal L_\Theta\mathbf D_2(\bm{\beta}, \bm{\gamma}) + \delta\mathbf D_2(\bsb{X}\bsb{\beta},\bsb{X}\bsb{\gamma}) + \vartheta P_H(\varrho(\bm{\beta} - \bm{\gamma});\lambda)  +  P_\Theta(\varrho\bm{\beta};\lambda)\\
\le\,& (1+ \frac{\alpha}{L}r)\bm{\Delta}_{l}(\bm{\beta},\bm{\gamma}) + P_\Theta(\varrho\bm{\gamma};\lambda) + K\lambda^2J(\bm{\beta}) \end{split}
\end{align}}for any $\bm\gamma\in\mathbb R^p$, where  $\lambda = A\sigma\sqrt{\log(ep)}/\sqrt{(\delta\wedge\vartheta )\vartheta}$ with $A$ a sufficiently large constant. The    oracle inequality below holds for some constant $C>0$, 
{\small\begin{align} \label{R_0-oraineq-general}
\begin{split}
\EE\bm\Delta_l(\hat{\bm\beta},\bm\beta^*) \le \,& \EE\Big\{  \Big(\frac{1+r}{1-r}\Big)^2  \bm\Delta_l(\bm\beta,\bm\beta^*)  +   \frac{(1+r)KA^2}{(1-r)(2\vartheta\wedge\delta )\vartheta}\sigma^2J(\bm\beta)\log(ep) \Big\} + \frac{C(1+r)}{(1-r)(2\vartheta\wedge\delta) } \sigma^2.
\end{split}
\end{align}}
\end{thm}

    Compared with     \eqref{R_0-errrate} which fixes $\bsb{\gamma}$ at $\bsb{\beta}^*$, \eqref{R_0-oracond-general} has $(1+ \frac{\alpha}{L}r)\bm{\Delta}_{l}$  in place of $2\sym{\bm{\Delta}}_{l}$ as   the first term on the right-hand side. Nonrigorously,  these conditions ask   $2\sym{\bm{\Delta}}_{l}$ or $ (1+ \frac{\alpha}{L}r)\bm{\Delta}_{l}$  to dominate $\varrho^2\mathcal L_\Theta\mathbf D_2$ in  a restricted sense; Remark \ref{rmk:regconds}   argues  that   \eqref{R_0-oracond-general} is not technically demanding compared with  many  other    regularity conditions in the literature.

When $r=0$,  the multiplicative constant proceeding $\bm\Delta_l(\bm\beta,\bm\beta^*)$   in \eqref{R_0-oraineq-general} is as small as   $1$, resulting in a {sharp} oracle inequality \citep{Koltchinskii2011book}.
If one sets $\bm\beta = \bm\beta^*$ in \eqref{R_0-oraineq-general}, the Bregman error $\bm\Delta_l(\hat{\bm\beta},\bm\beta^*)$ is of the order $\sigma^2J^*\log(ep)$ for any thresholding (when $\delta,\vartheta,K$ are treated as constants).
But the bias term $\bm\Delta_l(\bm\beta,\bm\beta^*)$ or $\bm\Delta_{l_0}(\bsb{X}\bm\beta,\bsb{X}\bm\beta^*)$   helps to handle \emph{approximately} sparse signals:  when $\bm\beta^*$ contains a number of small  nonzero elements,  rather than taking $\bm\beta = \bm\beta^*$,  a reference $\bm\beta$ with a reduced support will   yield an even smaller error bound  benefiting from the bias-variance tradeoff.

Unlike the optimization error bounds, the statistical error bounds  never  vanish (unless $\sigma\rightarrow 0$). We can similarly analyze  the set of  global minimizers,  in which case  the term  $\varrho^2\mathcal L_\Theta\mathbf D_2(\bm{\beta}, \bm{\beta}^*)$ is dropped from the regularity conditions, but the error bounds remain of  the  same order (cf. Remark \ref{staterr-exts} in Section \ref{proof:ora}).
In fact, for   sparse GLMs, by  Theorem \ref{thm:minimaxglm}, the  rate    $\sigma^2 J^* \log (ep)$ is   essentially minimax optimal (thus unbeatable) up to a logarithmic factor. 


\begin{remark}[Regularity condition comparison] \label{rmk:regconds}  
    The GBF-based regularity conditions \eqref{R_0-errrate}, \eqref{R_0-oracond-general}   are  no more demanding than some commonly used regularity conditions.
Assume that $P_\Theta$ is subadditive: $P_{\Theta}(t+s) \le P_{\Theta}(t) + P_{\Theta}(s)$, which holds when it is concave on $\mathbb R_+$. Let $\mathcal J = \mathcal J(\bm\beta)$, $J = |\mathcal J(\bsb{\beta})$\textbar, $\bm\gamma = \bm\beta'-\bm\beta$. Then, from $P_{\Theta}(\varrho\bm{\beta}_{\mathcal J}';\lambda) - P_{\Theta}(\varrho\bm{\beta}_{\mathcal J};\lambda)\le P_{\Theta}(\varrho(\bm{\beta}' - \bm{\beta})_{\mathcal J};\lambda)$ and $P_{\Theta}(\varrho\bm{\beta}_{\mathcal J^c}';\lambda) = P_{\Theta}(\varrho(\bm{\beta}' - \bm{\beta})_{\mathcal J^c};\lambda)$, \eqref{R_0-oracond-general}   is implied by
$P_{\Theta}(\varrho\bm{\gamma}_{\mathcal J}; \lambda) + \vartheta P_H(\varrho\bm{\gamma}_{\mathcal J}; \lambda) + \mathcal L_{\Theta}\mathbf{D}_2(\varrho\bm{\beta},\varrho\bm{\beta}') + \delta\|\bm X\bsb{\gamma}\|_2^2 /2
\le\,(2-\varepsilon ) \bm{\Delta}_{l}(\bm{\beta},\bm{\beta}') + K \lambda^2J + P_{\Theta}( \varrho\bm{\gamma}_{\mathcal J^c} ; \lambda) - \vartheta P_{H}(\varrho\bm{\gamma}_{\mathcal J^c}; \lambda)$,
or
$(1+ \vartheta)P_{\Theta}(\varrho\bm{\gamma}_{\mathcal J}; \lambda) +  \mathcal L_{\Theta}\mathbf{D}_2(\varrho\bm{\beta},\varrho\bm{\beta}') + \delta\|\bm X \bm{\gamma}\|_2^2/2
\le (2-\varepsilon )\bm{\Delta}_{l}(\bm{\beta},\bm{\beta}') + K \lambda^2J + (1-\vartheta) P_{\Theta}( \varrho\bm{\gamma}_{\mathcal J^c} ; \lambda)$ since $P_H \le P_{\Theta}$.

To get more intuition, let     $l(\bm{\beta}) = \|\bm{X\beta} - \bm{y}\|_2^2/2$.  Then the above condition simplifies to
$(1+\vartheta)P_\Theta(\varrho\bm{\gamma}_{\mathcal J};\lambda) +  {\mathcal L_\Theta}\|\varrho\bm{\gamma}\|_2^2/2
\le    (2-\varepsilon') \|\bm X \bm{\gamma}\|_2^2 /2+ K\lambda^2J + (1-\vartheta)P_\Theta(\varrho\bm{\gamma}_{\mathcal J^c};\lambda)$ with $\varepsilon' = \varepsilon +\delta$,  or the following sufficient condition   (with  $K$ redefined) for all $\bm\gamma\in\mathbb R^p$:  
{\small\begin{equation}
(1+\vartheta)P_\Theta(\varrho\bm{\gamma}_{\mathcal J};\lambda) + \frac{\mathcal L_\Theta}{2}\|\varrho\bm{\gamma}\|_2^2
\le   K \sqrt J \lambda  \|\bm X \bm{\gamma}\|_2    + (1-\vartheta)P_\Theta(\varrho\bm{\gamma}_{\mathcal J^c};\lambda).  \label{regcondition-comp}
\end{equation}}For  lasso, where  $P_\Theta(\bm\beta;\lambda) = \lambda\|\bm\beta\|_1$, there is  a rich collection of  regularity conditions  in the literature.  In this convex case,   $\mathcal L_\Theta = 0$ and  $\varrho$ can be arbitrarily large. \eqref{regcondition-comp} reduces to (with $\vartheta$ and $K$ redefined and $\lambda$ canceled)
{\small \begin{align}\label{rem1-4}
(1+\vartheta)\varrho\|\bm{\gamma}_{\mathcal J}\|_1 \le K\sqrt{J}\|\bm X\bm{\gamma}\|_2 + \varrho\|\bm{\gamma}_{\mathcal J^c}\|_1, \forall \bm{\gamma}
\end{align}}for some $K\ge 0, \vartheta>0$.
Taking $\varrho = c\| \bsb{X}\|_2$ results in scale invariance with respect to $\bsb{X}$. Let's compare   \eqref{rem1-4} with    the restricted eigenvalue (RE) condition and the compatibility condition \citep{Bickel2009,van2009conditions}. 
For given $\mathcal J$, the  two conditions assume that there exist positive numbers $\kappa$, $\vartheta_{RE}$ such that
$
J\| \bm X \bm{\gamma} \|_2^2 \ge \kappa\|\bm{\gamma}_{\mathcal J}\|_1^2
$  (compatibility)  
or more restrictively,
$
\| \bm X \bm{\gamma} \|_2^2 \ge \kappa \|\bm{\gamma}_{\mathcal J}\|_2^2
$ (RE), 
for all $\bm{\gamma}:  
(1+\vartheta_{RE})\|\bm{\gamma}_\mathcal{J}\|_1 \ge  \|\bm{\gamma}_{\mathcal J^c}\|_1$. 
 Therefore,   $(1+\vartheta)\varrho\|\bm{\gamma}_{\mathcal J}\|_1 \le K\sqrt{J}\|\bm X\bm{\gamma}\|_2 \vee     \varrho\|\bm{\gamma}_{\mathcal J^c}\|_1$   with   $K = (1+\vartheta_{RE})/(\varrho\sqrt{\kappa})$, $\vartheta = \vartheta_{RE}$. That is, the RE-type conditions are   more   demanding than  \eqref{rem1-4} (and  \eqref{R_0-oracond-general}).
 Another popular set of regularity conditions is based on  restricted strong convexity (RSC).  Under a version of  RSC  condition (and assuming $f$ is differentiable), \cite[Theorem 1]{Loh2015} showed that $\|\tilde{\bm\beta}-\bm\beta^*\|_2^2$  has a bound of order $\sigma^2(J^*\log p)/n$ for any stationary point $\tilde{\bm\beta}$. In the lasso case, the condition becomes $\|\bm X\bm\gamma\|_2^2 \ge \alpha\|\bm\gamma\|_2^2 - \tau\log p\|\bm\gamma\|_1^2/n$ for some constant $\alpha>0$ and $\tau\ge 0$, from which it follows that for any $\bm\gamma:(1+\vartheta_{RE})\|\bm{\gamma}_\mathcal{J}\|_1 \ge  \|\bm{\gamma}_{\mathcal J^c}\|_1$,
$
\|\bm X\bm\gamma\|_2^2  \ge \alpha\|\bm\gamma\|_2^2 - \tau(2+\vartheta_{RE})^2\frac{\log p}{n}\|\bm\gamma_{\mathcal J}\|_1^2
 \ge \alpha\|\bm\gamma\|_2^2 - \tau(2+\vartheta_{RE})^2 \frac{J\log p}{n} \|\bm\gamma_{\mathcal J}\|_2^2
 \ge \kappa'\|\bm\gamma_{\mathcal J}\|_2^2,
$ where $\kappa' = \alpha - \tau(2+\vartheta_{RE})^2(J\log p/n)$. Therefore, when $n\gg J\log p$, RSC implies   RE   and so is more restrictive than  \eqref{rem1-4}. See  Remark \ref{staterr-exts} in Section \ref{proof:ora} for  an extension to  general penalties.
\end{remark}

 \begin{remark}[Technical treatment] \label{advofoptbasedstatanal}  
A big difference between our work and   \cite{Loh2015} is that the latter enforces    an $\ell_1$-type side constraint, for example, $\|\bm\beta\|_1\le R$, in addition to the sparsity-inducing penalty $P$. The use of the constraint is a necessary ingredient of the proofs  and the constraint parameter $R$  appears in  the  minimum sample size condition and the  error bounds implicitly. However, few practically used  algorithms    seem to include such an additional $\ell_1$ constraint.

Our analysis   does not need any side constraint, and the resulting error bounds and the oracle inequality hold  with no minimum sample size requirement.
In fact, in dealing with a general  penalty that may be nonconvex, our treatment of the stochastic term is distinctive from the conventional  ``$\ell_1$ fashion'' via  H\"{o}lder's inequality: $\langle \bsb{\epsilon} , \bsb{X} \bsb{\beta}\rangle \le \|\bsb{X}^\top  \bsb{\epsilon}\|_{\infty} \| \bsb{\beta}\|_1$ (see, e.g., \cite{bunea2007sparsity,Bickel2009,Lounici2011}). More concretely,  applying the union bound to $\|\bsb{X}^\top  \bsb{\epsilon}\|_{\infty} $ will lead to a further upper bound      $\|\bsb{\beta}\|_2^2 +P(\bsb{\beta};\lambda)$ up to multiplicative factors   \cite{Loh2015}, while  we     can bound    $\langle \bsb{\epsilon} , \bsb{X} \bsb{\beta}\rangle$  by the sum of $\| \bsb{X} \bsb{\beta} \|_2^2/a$ and a light penalty $P_H(\bsb{\beta}; \lambda)/b$ for any   $a, b>0$, with a proper choice of $\lambda$. 
\end{remark}

\begin{remark}[Fixed points vs. local minimizers] \label{fixedvslocal}  
Targeting at the fixed points of the Bregman surrogate  instead of the local minimizers of the  original objective  seems more reasonable   from a statistical perspective.  Certainly, if $f$ is smooth,     $\mathcal F$  contains more valid solutions  (cf. Lemma \ref{lemma:degeneracy}).   But a more important reason is that $\mathcal F$   can adaptively exclude  bad local solutions  for some statistical learning problems with severe {nonsmoothness} and nonconvexity.

For instance,  each bridge $\ell_q$-penalty ($q:0\le q<1$) \citep{Frank1993Bridge} determines a  thresholding $\Theta_q$, which is however  the solution for  infinitely many penalties; picking the particular  one constructed from \eqref{pendef} that is the lowest and directionally differentiable   \cite{She2016}, one can  repeat the analysis in Theorems \ref{thm:errrate}, \ref{thm:ora}    to show  provable guarantees for all  the fixed points of the iterative $\Theta_q$ procedure. In contrast, as  pointed out by \cite{Loh2015}, the original optimization problem may contain   ``faulty''   local minimizers. In fact,  when  $q=0$, the $\ell_0$-penalized problem $\min_{\bm\beta}\|\bm X\bm\beta - \bm y\|_2^2/2+(\lambda^2/2)\|\bm\beta\|_0$ (not directionally differentiable)  \emph{always} has  $\bm 0$ as a local minimizer which is however  a poor estimator  as $\bsb{\beta}^*$ is large. \  Switching to   the surrogate's   fixed points successfully addresses the issue: $\hat{\bsb{\beta}} = \bsb{0}$ is a valid fixed point only when $\bm X^\top\bm y$ is properly small:   $\|\bm X^\top\bm y\|_\infty \le \lambda$, or the true signal is inconsequential   relative to the maximum noise level.     \end{remark}

\subsubsection{Statistical analysis of   the iterates from Bregman surrogates}
\label{sec:stat-seq}

We show  a nice  result for     \eqref{P_iterate} in the composite setting: under a   regularity  condition similar to those in Section \ref{subsec:fixed},  with high probability, the $t$-th iterate can approach the statistical target within the desired   precision geometrically fast, even when $p>n$. Specifically, we add   a mild  multiple of $\bm\Delta_\psi$ to the left-hand side of   \eqref{R_0-errrate} and assume that for some $\delta > 0$, $\varepsilon > 0$, $\vartheta > 0$ and large $K\geq 0$,    
{\small
\begin{align}\label{S}
\begin{split}
&\varepsilon\bm\Delta_\psi(\bm{\beta}^*,\bm{\beta}) + \delta\mathbf D_2(\bm{X\beta},\bm{X\beta}^*) + \vartheta P_H(\varrho(\bm{\beta} - \bm\beta^*);\lambda) + P_{\Theta}(\varrho\bm\beta^*;\lambda) \\
\le\,& (2\sym{\bm{\Delta}}_{l} - \varrho^2\mathcal L_{\Theta}\mathbf{D}_2)( \bm{\beta}, \bm{\beta}^*) + P_{\Theta}(\varrho\bm{\beta};\lambda) + K\lambda^2J(\bm{\beta}^*), \forall \bm\beta
\end{split}
\end{align}}\normalfont and   $\psi$ is   differentiable for simplicity.
Recall that  \eqref{strcvx-cond-1} in Proposition \ref{th:strongcvx} requires   $2\sym{\bm\Delta}_f$ to dominate     $\varepsilon\bm\Delta_\psi$; \eqref{S} gives    a large-$p$ extension of it.

\begin{thm}\label{th:stat_bound}
Under   the above regularity condition, for $\lambda = A\sigma\sqrt{\log(ep)}/\allowbreak\sqrt{(\delta\wedge\vartheta )\vartheta}$ with $A$ sufficiently large and $\kappa = 1/(1+\varepsilon)$, we have
{\small \begin{equation} \label{seq-linear}
\bm\Delta_\psi(\bm{\beta}^*,\bm{\beta}^{(t)}) \leq \kappa^t\bm\Delta_\psi(\bm{\beta}^*,\bm{\beta}^{(0)}) + \frac{\kappa}{1-\kappa}( K\lambda^2J^* - \min_{1\le s\le t}\bm\Delta_\psi(\bm\beta^{(s)},\bm\beta^{(s-1)}) )
\end{equation}}for any $t\ge 1$ with probability at least $1-Cp^{-cA^2}$, where $C,c$ are universal positive constants.
\end{thm}

The error measure $\bm\Delta_\psi(\bm{\beta}^*,\bm{\beta}^{(t)})$   in \eqref{seq-linear}    has $\bm\beta^*$ as its first argument and differs  from the $\bm\Delta_l(\hat{\bm\beta},\bm\beta^*)$  used in \eqref{R_0-oraineq-general}. 
According to the proof, \eqref{S}  only needs to hold for $\bm\beta= \bm\beta^{(s)}$  ($0\le s\le t$),  and so different starting values may give different values of $\kappa$.
With $\bm\Delta_\psi \ge 0$ (which can  be realized  by stepsize control),
the fast converging  statistical error  to  $\mathcal O(\sigma^2J^*\log(ep))$ implies that    over-optimization may be unnecessary.
As an example, consider the iterative thresholding procedures with $\breg_l \le L \Breg_2$ and $\varrho^2 > L$. Then \eqref{seq-linear} yields{\small\begin{align*}
\|\bm\beta^*-\bm\beta^{(t)}\|_2^2 \le \kappa^t \frac{   \varrho^2  }{ \varrho^2-L }\|\bm\beta^*-\bm\beta^{(0)}\|_2^2+ \frac{2\kappa K}{(1-\kappa)(\varrho^2-L)} \lambda^2J^*.
\end{align*}}So it is possible to terminate the iterative algorithm   before full computational convergence without sacrificing much statistical accuracy. The simulations in Section   \ref{subsec:simu_stat}  support this point.

\begin{remark}  
Theorem \ref{th:stat_bound} reveals the fast decay of  the \emph{direct} statistical error between $\bm\beta^{(t)}$ and $\bm\beta^*$.   \cite{Agarwal2012} and \cite{Loh2015} argued a similar point for gradient descent type algorithms, in a somehow indirect manner:   (i) $\bm\beta^{(t)}$ can approach any globally optimal solution $\tilde{\bm\beta}$ geometrically fast in computation under a combination of an RSC condition and an RSM condition, and  (ii) under some regularity conditions, every local minimum point is close enough to the authentic $\bm\beta^*$.
In the RSC condition  for (i), the factor proceeding the dominant term $\sym{\bm\Delta}_l$ is $1$ ({there are two different sets of RSC conditions used in Theorem 1 and Theorem 3 of \cite{Loh2015}, the factor $\alpha_1$ in the second set corresponding to \emph{half} of the $\alpha_1$ used in the first set}). But  \eqref{S} allows it to be $2$.
Moreover,  Theorem \ref{th:stat_bound} does not need the extra RSM condition and  applies to a broader class of algorithms.
For example,  we can  show that the statistical error of the {LLA} algorithm   reduces at a linear rate to the desired precision under some regularity conditions; see Proposition \ref{thm:stat_LLA} and Lemma \ref{lemma:Delta_Delta_P} in Section \ref{subsec:LLA-stat}. 
\end{remark}

\section{Two  acceleration schemes for generalized Bregman surrogates}\label{sec:acc}

 How to accelerate first-order algorithms without incurring much additional cost {per iteration}   has lately attracted lots of   attention in  big data applications.
In  convex optimization, Nesterov's momentum techniques prove to be quite effective in that the rate of convergence can be improved from $\mathcal O(1/t)$ to $\mathcal O(1/t^2)$, which is optimal when using  first-order methods on  smooth  problems \citep{Nesterov2004book,Tseng2008,Beck2009FISTA,Krich2015}.
This section attempts to extend Nesterov's \textit{first} and \textit{second} accelerations \citep{Nesterov1983, Nesterov1988} to   Bregman-surrogate algorithms. With a possible  lack of smoothness or convexity, carefully choosing the relaxation parameters and step sizes is the key, and we will see the benefit of   maximizing a quantity  $R_t/(\theta_t^2\rho_t)$ at the $t$-th iteration,  with $R_t$  appropriately defined via generalized Bregman notation. 
 We consider the following two broad scenarios to devise the acceleration schemes.

\paragraph*{Scenario 1} $g(\bm{\beta};\bm{\gamma}) = f(\bm{\beta}) - \bm\Delta_{\psi_0}(\bm{\beta},\bm{\gamma}) + \rho\mathbf D_2(\bm{\beta},\bm{\gamma})$.
This surrogate family includes gradient descent type algorithms. Often, if $\min_{\bm\beta} f(\bm\beta) + \bm\Delta_\psi(\bm\beta,\bm\gamma)$ is easy to solve, so is $\min_{\bm\beta} f(\bm\beta) + \bm\Delta_\psi(\bm\beta,\bm\gamma) + \rho\mathbf D_2(\bm\beta,\bm\gamma)$, in which case $\psi_0 = -\psi$.

\paragraph*{Scenario 2} $g(\bm{\beta};\bm{\gamma}) = f(\bm{\beta}) - \bm\Delta_{\psi_0}(\bm{\beta},\bm{\gamma}) + \rho\bm\Delta_\phi(\bm{\beta},\bm{\gamma})$.
This gives a more general class than the first one. 
\\

This section assumes that $f$,  $\psi_0$, $\phi$, $\bm\Delta_{\psi_0}(\cdot,\bm\gamma)$,  $\bm\Delta_{\psi_0}(\cdot,\bm\gamma)$ are directionally differentiable given  any $\bm\gamma$. We introduce a convenient notation $\mathbf C_\psi$ defined for any $\psi$ as follows
\begin{equation} \label{def:C}
\mathbf C_\psi(\bm{\alpha},\bm{\beta},\theta) = \theta \psi(\bm{\alpha}) + (1-\theta)\psi(\bm{\beta}) - \psi(\theta\bm{\alpha}+(1-\theta)\bm{\beta}), \end{equation}
where $0\le\theta\le 1$.
Like $\breg$, $\mathbf C$ is a linear operator of $\psi$ and its nonnegativity means   convexity. Some connections between $\breg$ and   $\mathbf C$ are given below.

\begin{lemma} \label{lemma:C}
  Let $\psi$ be directionally differentiable. (i)  $\mathbf C_\psi(\bm\alpha,\bm\beta,\theta) = (1-\theta)\bm\Delta_\psi(\bm\beta,\bm\alpha) - \bm\Delta_\psi(\theta\bm\alpha+(1-\theta)\bm\beta,\bm\alpha)$ for any $\bm\alpha,\bm\beta$ and  $\theta\in [0,1]$.
(ii) $\mathbf C_{\bm{\Delta}_\psi(\cdot,\bm\alpha)} = \mathbf C_\psi$  if   $\psi$ is differentiable at $\bsb{\alpha}$.
\end{lemma}

\paragraph*{An acceleration scheme of the second kind}
 Scenario 2 is of our primary interest since it applies more broadly. Below, we modify the surrogate and define   an iterative algorithm  (not a descent method)
 that involves three sequences $\bm\alpha^{(t)}$, $\bm\beta^{(t)}$, $\bm\gamma^{(t)}$ starting at $\bm\alpha^{(0)}=\bm\beta^{(0)}$:
\begin{subequations}\label{acc2-alg}
\begin{align}
&\bm{\gamma}^{(t)}   =  (1-\theta_t)\bm{\beta}^{(t)} + \theta_t\bm{\alpha}^{(t)} \label{acc2-alg1}\\
&\bm{\alpha}^{(t+1)}   {=}  \mathop{{\arg}{\min}} f(\bm{\beta}) {-} \bm\Delta_{\psi_0}(\bm{\beta},\bm{\gamma}^{(t)}) {+}\mu_0 \breg_{\phi}(\bm{\beta},\bm{\gamma}^{(t)}){+} \theta_t\rho_t\bm\Delta_\phi(\bm{\beta},\bm{\alpha}^{(t)})\label{acc2-alg2}\\
&\bm{\beta}^{(t+1)}  = (1-\theta_t)\bm{\beta}^{(t)} + \theta_t\bm{\alpha}^{(t+1)}, \label{acc2-alg3}
\end{align}
\end{subequations}
for some  $\mu_0\ge 0$,  $\theta_t\in (0, 1]$, $\rho_t>0$ ($\forall t\ge 0$), to be chosen later. Notice  the extra GBF term $\mu_0 \breg_{\phi}(\cdot,\bm{\gamma}^{(t)})$ in \eqref{acc2-alg2} in addition to $\breg_\phi(\cdot,\bm{\alpha}^{(t)})$.
The design of relaxation parameters $\theta_t$  and inverse step size parameters  $\rho_t, \mu_0$ holds the key to acceleration.
Let \begin{align}\bar \psi_0 = \psi_0 - \mu_0\phi. \label{psi0bardef}\end{align} We advocate the following line search criterion
{ 
\begin{subequations}\label{acc2_search}
\begin{align}
\begin{split}
&R_t:=\theta_t^2\rho_t\bm\Delta_\phi(\bm{\alpha}^{(t+1)},\bm{\alpha}^{(t)}) - \bm\Delta_{\bar\psi_0}(\bm{\beta}^{(t+1)},\bm{\gamma}^{(t)}) + (1-\theta_t)\bm\Delta_{\bar \psi_0}(\bm{\beta}^{(t)},\bm{\gamma}^{(t)}) \\
& \qquad \ + \mathbf C_{f(\cdot)-\bm\Delta_{\bar \psi_0}(\cdot,\bm\gamma^{(t)})}(\bm\alpha^{(t+1)},\bm\beta^{(t)},\theta_t)\ge 0,\end{split}\label{acc2_search1}
\\&\frac{\theta_t^2}{1-\theta_t} = \frac{\theta_{t-1}(\rho_{t-1}\theta_{t-1} + \mu_0)}{\rho_t},  \ t\ge 1. \label{acc2_search2}
\end{align}\end{subequations}}\normalfont
The update of the relaxation parameter involves $\rho$ and $\mu$ as well. 

Theorem \ref{th:comp_acc2} presents two error  bounds without assuming   convexity or smoothness, and shows     in general the  reasonability of \eqref{acc2_search1}. 

\begin{thm}\label{th:comp_acc2}
Let    $\rho_t$ be any positive sequence. Consider the algorithm defined by \eqref{acc2-alg1}--\eqref{acc2-alg3} and \eqref{acc2_search2}.
Let $\mathcal E_t(\bm\beta) = \bm\Delta_{\bar \psi_0}(\bm{\beta},\bm{\gamma}^{(t)}) + \bm\Delta_{f(\cdot)-\bm\Delta_{\psi_0}(\cdot,\bm\gamma^{(t)})}(\bm\beta,\bm\alpha^{(t+1)})+(\mu_0\bm\Delta_{\bm\Delta_\phi(\cdot,\bm\gamma^{(t)})-\phi(\cdot)}\allowbreak+\theta_t\rho_t\bm\Delta_{\bm\Delta_\phi(\cdot,\bm\alpha^{(t)})-\phi(\cdot)})(\bm\beta,\bm\alpha^{(t+1)})$.

(i) When $\mu_0=0$, for any $\bm{\beta}$ and $T\ge 0$,
{ \begin{equation} \label{acc2_result}
\begin{split}
&\frac{f(\bm{\beta}^{(T+1)})-f(\bm{\beta})}{\theta_T^2\rho_T} + T{\cdot}\mathop{\avg}_{0\le t\le T}\frac{\mathcal E_t(\bm\beta)}{\theta_t\rho_t} + T{\cdot}\mathop{\avg}_{0\le t\le T}\frac{R_t}{\theta_t^2\rho_t}  \\ \le \, &
\bm\Delta_\phi(\bm{\beta},\bm{\alpha}^{(0)}) - \bm\Delta_\phi(\bm{\beta},\bm{\alpha}^{(T+1)})+ \frac{1-\theta_0}{\theta_0^2\rho_0}\big[f(\bm{\beta}^{(0)})-f(\bm{\beta})\big] .
\end{split}\end{equation}}

(ii) Moreover, given any $\mu_0\ge 0$, {
\begin{equation}\label{eq:th_acc2-conclusion2}
\begin{split}
& f(\bm{\beta}^{(T+1)})-f(\bm{\beta} )   +\theta_T^2(\rho_T +\frac{\mu_0}{\theta_T}) \bm\Delta_\phi(\bm{\beta},\bm{\alpha}^{(T+1)})\\& + {\mbox{\large$\Sigma$}}_{t=0}^T \big(\Pi_{s=t+1}^T (1-\theta_s)\big) (R_t + \theta_t \mathcal E_t(\bsb{\beta}))\\
\le\,& \Big ({\mbox{\small$\prod$}}_{t=1}^T(1-\theta_t)\Big)\big[(1-\theta_0)(f(\bm{\beta}^{(0)})-f(\bm{\beta}))+\theta_0^2\rho_0 \bm\Delta_\phi(\bm{\beta},\bm{\beta}^{(0)})\big]
\end{split}
\end{equation}
}for all  $\bsb{\beta}$ and $T\ge 0$,
where by convention,  $\prod_{s = l}^u a_s=1$ as $l>u$.

\end{thm}

First, we make a discussion of the results  for convex optimization. Assume $\bm\Delta_\phi \ge \sigma\mathbf D_2$ for some $  \sigma>0$. With the additional knowledge that $f(\cdot)-\bm\Delta_{\bar \psi_0}(\cdot,\bm\gamma^{(t)})$ is convex and $\bm\Delta_{\bar \psi_0} \le L_{\bar \psi_0}\mathbf D_2$   for some $L_{\bar \psi_0} \ge 0$,  \eqref{acc2_search1} is implied by
\begin{equation}
\theta_t^2(\rho_t - L_{\bar \psi_0}/\sigma)\bm\Delta_\phi(\bm{\alpha}^{(t+1)},\bm{\alpha}^{(t)}) + (1-\theta_t)\bm\Delta_{\bar \psi_0}(\bm{\beta}^{(t)},\bm{\gamma}^{(t)}) \ge 0. \label{2ndacccvxline}
\end{equation}
So when $f$ is convex,  criterion \eqref{acc2_search} is satisfied by $\rho_t = \rho\ge L_{\bar \psi_0}/\sigma, \psi_0 = f, \mu_0=0$ and $\theta_{t+1} = (\sqrt{\theta_t^4 + 4\theta_t^2} - \theta_t^2)/2$,   degenerating to Nesterov's second method \citep{Nesterov1988,Tseng2008}, and the convergence rate is of order $\mathcal O(1/T^2)$ according to  \eqref{acc2_result} and   \eqref{eq:thetaratewithunivrho}.   The second conclusion tells  more  when strong convexity (or restricted strong convexity)  arises. Given a convex $f$ satisfying $\mu\Breg_\phi\le \breg_f\le L\Breg_\phi$ with $0<\mu\le L$ and    $\phi$ differentiable, taking    $\psi_0=f$, $\mu_{0} = \mu$, and $\rho_t = L -\mu$ ensures $\mathcal E_t(\bm\beta)=\bm\Delta_{f-\mu\phi}(\bm{\beta},\bm{\gamma}^{(t)}) + \bm\Delta_{f(\cdot)-\bm\Delta_{f}(\cdot,\bm\gamma^{(t)})}(\bm\beta,\bm\alpha^{(t+1)})
\ge\bm\Delta_{f-\mu\phi}(\bm{\beta},\bm{\gamma}^{(t)}) \ge 0$ and $ R_t\ge\theta_t^2\rho_t\Breg_\phi(\bm{\alpha}^{(t+1)},\bm{\alpha}^{(t)}) - \bm\Delta_{\bar\psi_0}(\bm{\beta}^{(t+1)},\bm{\gamma}^{(t)})  \ge \theta_t^2 (\rho_t+\mu_0-L) \sigma \Breg_2(\bm{\alpha}^{(t+1)},\bm{\alpha}^{(t)})=0 $. According to \eqref{acc2_search2},
the following choice\begin{equation}\theta_t = \theta_0= \frac{2}{\sqrt{4 \kappa -3} + 1} \mbox{ with }  \kappa = L/\mu\end{equation}
suffices, and the optimization problem to solve in \eqref{acc2-alg2} becomes \begin{align}\min f(\bm{\gamma}^{(t)})+ \delta f(\bm{\beta}; \bm{\beta} -\bm{\gamma}^{(t)})  + \mu\Breg_{\phi}(\bm{\beta},\bm{\gamma}^{(t)}){+} \frac{2 (L - \mu)}{\sqrt{4 \kappa -3} + 1} \Breg_\phi(\bm{\beta},\bm{\alpha}^{(t)}).  \end{align} From \eqref{eq:th_acc2-conclusion2}, both  $f(\bm{\beta}^{(T+1)})-f(\bm{\beta} ) $ and $ \Breg_\phi(\bm{\beta} ,\bm{\alpha}^{(T+1)})$  enjoy a linear convergence with rate parameter $\frac{\sqrt{4 \kappa -3} - 1}{\sqrt{4 \kappa -3} + 1}$, or an iteration complexity of $\mathcal O(\sqrt \kappa \log (1/\epsilon))$,
significantly faster than $\mathcal O(  \kappa \log (1/\epsilon))$ in Proposition \ref{th:strongcvx}. 
   Hence       \eqref{acc2-alg},  \eqref{acc2_search}  can achieve rate-optimality  in various convex scenarios. To the best of our knowledge, this is    the first   ``all-in-one''  form of the \textit{second} acceleration that  adapts. 

The proposed algorithm   can even go beyond convexity.
As  a  demonstration, let us apply the acceleration to   the  iterative quantile-thresholding procedure (cf. Example \ref{ex:thres}) for  solving the feature screening problem:  $\min l(\bsb{\beta})= \| \bsb{y} - \bsb{X} \bsb{\beta}\|_2^2/2  \mbox{ s.t. } \| \bsb{\beta}\|_0\le q$, which is  nonconvex. Here, $q$ is bounded above by $p$ but may be larger than $n$. Take $\phi = \|\cdot \|_2^2/2$, $\mu_0=0$ and  $\psi_0(\bsb{\beta})= l(\bsb{\beta})  -    \mathcal L  \phi(\bsb{\beta}) $ for some $\mathcal L\ge 0$.
Given any $s\le p$ and  $\bsb{X}$, define the restricted isometry number $\rho_+(s)$ \citep{Candes2005} that satisfies
$
\|\bsb{X} \bsb{\beta}\|_2^2 \le  \rho_+(s) \|\bsb{\beta}\|_2^2$, $ \forall   \bsb{\beta}: \| \bsb{\beta}\|_0 \le s
$, which can be much smaller than $ \| \bsb{X}\|_2 ^2$ as $s$ is small.

 \begin{cor} \label{cor-AccIQ} Assume   $q$ is set larger than the target $\| \bsb{\beta}^*\|_0$  with the ratio denoted by $r$. Then  for any  $\mathcal L \ge \rho_+(2q)/\sqrt r$,  there exists  a \emph{universal} $\rho_t$ ($\rho_t =\rho_+(2q) (1-1 /\sqrt r)$, say), thereby    $\theta_{t+1} = (\sqrt{\theta_t^4 + 4\theta_t^2} - \theta_t^2)/2$, such that the accelerated iterative quantile-thresholding according to \eqref{acc2-alg1}--\eqref{acc2-alg3} satisfies  $l(\bm{\beta}^{(T+1)})-l(\bm{\beta}^*) + \min_{0\le t \le T} \bm\Delta_{\psi_0}(\bm\beta^*,\bm\gamma^{(t)}) \le   A/T^2 $
 for all $ T\ge 0$, where $A$ is independent of $T$.
 \end{cor}

The proof of the corollary shows the power of an \textit{accumulative} $R_t$-control,  and applies   more generally: if the objective function $f(\bsb{\beta})$, possibly nonconvex, can be written as the sum of a convex function  $l(\bsb{\beta})$ with $\breg_l\le L \Breg_2$ and a function $P(\bsb{\beta})$ that can be lifted: $\breg_P+\mathcal L_0 \Breg_2 \ge 0$ for some finite $\mathcal L_0\ge 0$, then  one can utilize a $\psi_0$ as $ l -0.6 \mathcal L_0\| \cdot \|_2^2$    and a universal $\rho_t $ to fulfill $ T{\cdot}\mathop{\avg}_{ t\le T} {R_t}/({\theta_t^2\rho_t}) \ge 0$ in \eqref{acc2_result}  (although not every $R_t$ is necessarily nonnegative) so as to attain   an $\mathcal O(1/T^2)$  error bound. See  Remark \ref{rmk:acc2univrho} in Section  \ref{subsec:2ndaccANDcor}.

Of course, a  time-varying $\rho_t$   can provide finer control, and the theorem does not limit $\rho_t$ to be constant.
In fact, under $\mu_0=0$, as long as $\rho_t/\rho_{t-1}\ge 1-(at+ab+1)/(t+b-1)^2$ ($t\ge 1$) for some constants $a,b$: $a>-2, b\ge a+1$,     induction based on \eqref{acc2_search2} 
 gives $\theta_t \le (a+2)/(t+b)$ and $ \sum_{t = 0}^T \rho_T/(\rho_t\theta_t) \ge (T+c_1)^2/(a+2)^{2} + c_2$ (with constants $c_i$   dependent on $a,b$) for any $t\ge 1$, from which it follows that
\begin{align}
 \theta_T^2 = \mathcal O(1/T^2) \mbox{~~and~~} T\cdot \avg_{0\le t\le T}(1/(\rho_t\theta_t)) \ge \mathcal O(T^2/\rho_T). \label{eq:thetaratewithunivrho}
\end{align}
Now, under   $R_t\ge 0$  or just  $\sum_{t=0}^T R_t/(\theta_t^2\rho_t) \ge 0$, \eqref{acc2_result} gives
$f(\bm{\beta}^{(T+1)})-f(\bm{\beta}) + \min_{0\le t\le T}\mathcal E_t(\bm\beta) \allowbreak \le \mathcal O(\rho_T/T^2)$ for any $\bm\beta$.
Typically,  \eqref{acc2_search1}  
involves a line search. If the condition fails for the current value of $\rho_t$,  one can set $\rho_t = \alpha\rho_t$ for some $\alpha > 1$ and recalculate $\theta_t$, $\bm\gamma^{(t)}$, $\bm\alpha^{(t+1)}$ and $\bm\beta^{(t+1)}$ according to \eqref{acc2_search2} and \eqref{acc2-alg} to verify it again. In implementation, it is wise to limit the number of searches at each iteration (denoted by $M$) to control the per-iteration complexity. If  \eqref{acc2_search1} 
does not hold after $m$ times of search, we simply pick the $\rho_t$ that gives the largest $R_t/(\theta_t^2 \rho_t)$ based on Theorem \ref{th:comp_acc2}. 
 Some  details are in Algorithm \ref{alg:2stab}. In simulation studies,  letting $M = 3$, $\alpha = 2$ already shows excellent  performance; see Figure \ref{fig:NBEM_acc} and Figure \ref{fig:Reg_acc}.

\paragraph*{An  acceleration scheme of the first kind}

For the algorithms falling into Scenario 1, we can alternatively consider two sequences of iterates generated by 
\begin{subequations}\label{acc1-alg}
\begin{align}
\bm{\gamma}^{(t)} &= \bm{\beta}^{(t)} + \{\rho_{t-1}\theta_{t}(1 - \theta_{t-1})/(\rho_{t-1}\theta_{t-1} + \mu_0)\}(\bm{\beta}^{(t)}-\bm{\beta}^{(t-1)}),
 \label{acc1-alg1}\\\bm{\beta}^{(t+1)} &= \mathop{\arg\min}  f(\bm{\beta}) {-} \bm\Delta_{\psi_0}(\bm{\beta},\bm{\gamma}^{(t)}){+}\mu_0 \Breg_2(\bm{\beta},\bm{\gamma}^{(t)}) {+} \rho_t\mathbf D_2(\bm{\beta},\bm{\gamma}^{(t)}),\label{acc1-alg2}
\end{align}
\end{subequations}
for some $\mu_0\ge0$, $\theta_t\in (0, 1]$,  $\rho_t>0$ for all  $t\ge 0$, and we force $\bm{\gamma}^{(0)}  = \bm{\beta}^{(0)} $. \eqref{acc1-alg1}, \eqref{acc1-alg2} give a  new first type acceleration, and notably, the novel update of $\bsb{\gamma}^{(t)}$ involves $\rho_{t-1}$.   When $\bm\beta^{(t+1)} = \bm\gamma^{(t)}$ one  stops the algorithm and obtains a fixed point  with provable statistical guarantees as shown in Section \ref{subsec:fixed}.

Similar to \eqref{psi0bardef}, let $\bar \psi_0 = \psi_0 - \mu_0\| \cdot \|_2^2/2$. Define
the line search criterion
\begin{subequations}\label{acc1_search}
\begin{align}
&R_t:=(\rho_t\mathbf D_2 - \bm\Delta_{\bar\psi_0})(\bm{\beta}^{(t+1)},\bm{\gamma}^{(t)}) + (1-\theta_t)\bm\Delta_{\bar\psi_0}(\bm{\beta}^{(t)},\bm{\gamma}^{(t)}) \ge 0,\label{acc1_search1}\\
&\frac{\theta_t^2}{1-\theta_t} = \frac{\theta_{t-1}(\rho_{t-1}\theta_{t-1} + \mu_0)}{\rho_t},~\theta_t \ge 0,~ \rho_t > 0,~ t\ge 1.\label{acc1_search2}
\end{align}
\end{subequations}
 Note that  $R_t$ is  defined differently from  \eqref{acc2_search1}. 
The following theorem  
reveals the importance of maximizing   $R_t$ in each iteration step when performing  possibly nonconvex optimization.

\begin{thm}\label{th:comp_acc1}
Given any $\rho_t>0$ $(t\ge 0)$, consider the algorithm defined by \eqref{acc1-alg1}, \eqref{acc1-alg2}, and \eqref{acc1_search2}. Let $\mathcal E_t(\bm\beta) = \bm\Delta_{\bar\psi_0}(\bm\beta,\bm\gamma^{(t)}) + \{\mathbf C_{f(\cdot)-\bm\Delta_{ \psi_0}(\cdot,\bm\gamma^{(t)})}(\bm\beta,\bm\beta^{(t)},\theta_t) + \bm\Delta_{f(\cdot)-\bm\Delta_{\psi_0}(\cdot,\bm\gamma^{(t)})}\allowbreak(\theta_t\bm{\beta} + (1-\theta_t)\bm{\beta}^{(t)},\bm\beta^{(t+1)})\}/\theta_t$.

(i) When $\mu_0=0$, we have 
\begin{equation*} 
\begin{split}
&\frac{f(\bm{\beta}^{(T+1)}) - f(\bm{\beta})}{\theta_T^2\rho_T} + T{\cdot} \mathop{\avg}_{0\le t\le T}\frac{\mathcal E_t(\bm\beta)}{\theta_t\rho_t} + T{\cdot} \mathop{\avg}_{0\le t\le T}\frac{R_t}{\theta_t^2\rho_t}  \\ \le\,&  \mathbf D_2(\bm{\beta},\bm{\beta}^{(0)})+  \frac{1-\theta_0}{\theta_0^2\rho_0}\big[f(\bm{\beta}^{(0)}) - f(\bm{\beta})\big] \  \mbox{ for any }  \bsb{\beta} \mbox{ and }   T\ge 0.
\end{split}\end{equation*}
(ii) Moreover, given any  $\mu_0\ge 0$, for all $\bsb{\beta}$ and $T\ge 0$,
{ \begin{equation*}
\begin{split}
& f(\bm{\beta}^{(T+1)})-f(\bm{\beta} )   +\theta_T^2(\rho_T +\frac{ \mu_0}{\theta_T}) \Breg_2(\bm{\beta}, ({\bm{\gamma}^{(T+1)}}{-}(1{-}\theta_{T+1})\bsb{\beta}^{(T+1)})/{\theta_{T+1}} )\\& \quad+ {\mbox{\large$\Sigma$}}_{t=0}^T \big(\Pi_{s=t+1}^T (1-\theta_s)\big) (R_t + \theta_t \mathcal E_t(\bsb{\beta}))\\
\le\,& \Big ({\mbox{\small$\prod$}}_{t=1}^T(1-\theta_t)\Big)\big[(1-\theta_0)(f(\bm{\beta}^{(0)})-f(\bm{\beta}))+\theta_0^2\rho_0 \Breg_2(\bm{\beta},\bm{\beta}^{(0)})\big].
\end{split}
\end{equation*}}
\end{thm}

Again, the new proposal of   the iterate and parameter updates adapts to various situations, with  $\mu_0$ (which  can  be a  sequence $\mu_t$, cf. Remark \ref{rmk:relaxedrhot}) measuring the degree of convexity (or restricted convexity in a nonconvex composite problem).
For example, when $ f$ is convex and $L$-strongly smooth,  $\mu_0=0$,  $\rho_t = L$, $\psi_0 = f$, and $\theta_{t+1} = (\sqrt{\theta_t^4 + 4\theta_t^2} - \theta_t^2)/2$ make  \eqref{acc1_search}   hold,  corresponding to  Nesterov's first  method. Interestingly, if $f$ is   $\mu$-strongly convex, the associated standard  momentum update $ \bm{\gamma}^{(t)} = \bm{\beta}^{(t)} + \theta_t(  \theta_{t-1}^{-1}-1) (\bm{\beta}^{(t)}-\bm{\beta}^{(t-1)})$  only attains a linear rate at $1-1/\kappa $ ($\kappa=L/\mu$) (cf. Remark \ref{rmk:1staccth}), showing  \textit{no} theoretical advantage over the plain gradient descent.     \eqref{acc1-alg} fixes the issue: with   $\mu_0=\mu$,  $\rho_t=L-\mu$, $\theta_t = {2}/({\sqrt{4 \kappa -3} + 1})$,          an accelerated  linear  rate parameter is obtained as  $ ({\sqrt{4 \kappa -3} - 1})/\allowbreak({\sqrt{4 \kappa -3} + 1}) (\le 1- \sqrt{3/(4\kappa)})$. 
   (When $\mu_0$ is unknown,   \eqref{acc1-alg2}  based on the
split $L = \rho_t + \mu_t$  is still advantageous over the classical acceleration  with $\rho_t =L$.)  We proved these error bounds by use of GBFs, which is perhaps more straightforward than  Nesterov's ingenious  proof based on
the notion of estimate sequence, and more importantly, \eqref{acc1-alg}, \eqref{acc1_search} provide a universal ``all-in-one'' form, instead of  separate schemes in different situations \cite{Nesterov2004book}.

Theorem \ref{th:comp_acc1}  accommodates  diverse choices of the parameters  $\psi_0,\mu_0, \rho_t$, $\theta_t$ and is motivating  in the nonconvex composite setup. Consider, for example, $\min f(\bsb{\beta}) = \| \bsb{y} - \bsb{X} \bsb{\beta}\|_2^2 /2+ P_{\Theta}(\varrho\bsb{\beta}; \lambda)$. Because the objective    is  nonconvex when $p>n$ and $\mathcal L_{\Theta}>0$, how to accelerate the associated iterative thresholding procedure is an unconventional   problem. From the studies in Section \ref{subsec:stat}, we have learned that a sparsity-inducing penalty with a properly large threshold to suppress the noise   can result in  strong convexity in a restricted sense. We can then use   a surrogate $f(\bsb{\beta}) +  (\rho  \Breg_2- \breg_{\psi_0} )(\bsb{\beta}, \bsb{\beta}^-)$ where $\psi_0(\bsb{\beta})  =\| \bsb{y} - \bsb{X} \bsb{\beta}\|_2^2 /2 - \varrho^2\mathcal L_\Theta\| \bsb{\beta}\|_2^2/2$ and $\mu_0=0$. Since  $f(\cdot)-\bm\Delta_{\psi_0}(\cdot,\bm\gamma)$ is convex (cf. Lemma \ref{lemma:Lp}),   $\mathcal E_t(\bm\beta) \ge \bm\Delta_{\psi_0}(\bm\beta,\bm\gamma^{(t)})$.
Moreover, thanks to the sparsity in   $\bsb{\beta}^{(t)}$, and thus $\bsb{\gamma}^{(t)}$, $\bsb{X}(   \bsb{\beta}^{(t)}-\bsb{\gamma}^{(t)}) $ involves just a small number of features. So with an incoherent design,  a properly small   $\varrho$ can make        $\bm\Delta_{\psi_0}(\bsb{\beta}^{(t)},\bsb{\gamma}^{(t)})\ge 0$. Now, taking a constant  $\rho_t$  as large as, for instance,  $ \| \bsb{X}\|_2^2 -\varrho^2\mathcal L_\Theta  $, 
   may yield a   convergence rate  of order   $\mathcal O(1/t^2)$. (Actually,  linear convergence may result from the restricted strong convexity under some regularity conditions.)
More generally, different $\rho_t$'s are allowed  in the theorem:     \eqref{eq:thetaratewithunivrho}  is still  secured with  just, say,  $\rho_t/\rho_{t-1}\ge1- (t+3)/(t+1)^2$.  A line search can be used to determine a proper sequence $\rho_t$; see Algorithm \ref{alg:1stab} for more details.

The  proposed  accelerations of the first kind and of the second kind  can be utilized in a wide range of problems. Because they are   momentum based, the original algorithms need not be substantially modified  to have an improved
iteration complexity, and the two theorems proved in this section 
apply in any dimensions with no design   coherence restrictions. Another delightful fact is that our ``all-in-one'' forms update the iterates adaptively according to  the degree of  convexity $\mu_0\ge0 ,$ which can be relaxed to a sequence of local measures
$\mu_t$  (Remark \ref{rmk:relaxedrhot}). With a line search to get properly large $\mu_t$, this could be  helpful in high dimensional sparse learning problems which may or may \textit{not} have
 restricted strong convexity (the associated parameter often hard to determine in theory).

\section{Summary}\label{sec:summ}
This paper studied   the class of iterative  algorithms  derived from  GBF-defined surrogates with a possible lack of convexity and/or smoothness. These surrogates  differ from the MM surrogates  frequently used in statistical computation, in that they gain additional first-order   degeneracy and  may drop the majorization requirement.  GBFs  have interesting connections  to the densities in the exponential family and possess some   idempotence properties  that are useful for studying iterative algorithms.

The  GBF calculus built by the lemmas    not only facilitates   optimization error analysis  but can be bound to      the empirical process theory for  nonasymptotic statistical  analysis (cf. Sections \ref{subsec:stat} and \ref{subsec:genopt}).
In addition to  obtaining some insightful results   in the realm of  convex optimization,
we were able to build universal global convergence rates for a broad class of Bregman-surrogate algorithms  for nonsmooth  nonconvex optimization. Moreover,  in the nonconvex composite setting that is of great interest in high dimensional statistics, we found that the sequence of iterates generated by  Bregman surrogates can  approach the statistical truth at a linear rate even when $p>n$, and the obtained fixed points   enjoy oracle inequalities with essentially the optimal order of statistical accuracy,  under some regularity conditions less demanding than those used in the literature. 
Finally,    we devised two ``all-in-one''  acceleration schemes   with novel updates of the iterates and relaxation and stepsize parameters, 
and     some sharp  theoretical bounds  were shown without assuming  smoothness or convexity.

%
%
%

\appendix
\numberwithin{equation}{section}
\numberwithin{lemma}{section}
\numberwithin{thm}{section}
\numberwithin{cor}{section}
\numberwithin{defn}{section}
\numberwithin{remark}{section}
\renewcommand\thefigure{\thesection.\arabic{figure}}

\section{Proofs}\label{sec:proof}

We list some notation and symbols that are used in the proofs.
Given a directionally differentiable function $\psi$, $\bm\Delta_\psi(\bm\beta,\bm\gamma) = \psi(\bm\beta) - \psi(\bm\gamma) - \delta\psi(\bm\gamma;\bm\beta-\bm\gamma)$, $\sym{\bm{\Delta}}_\psi(\bm{\beta},\bm{\gamma}) = (\bm{\Delta}_\psi(\bm{\beta},\bm{\gamma}) + \bm{\Delta}_\psi(\bm{\gamma},\bm{\beta}))/2$, and ${\back{\bm\Delta}_\psi}(\bm\beta,\bm\gamma) = \bm\Delta_\psi(\bm\gamma,\bm\beta)$. We occasionally denote $\bm\Delta_\psi(\bm\beta,\bm\gamma)$ by $\bm\Delta(\bm\beta,\bm\gamma)$ when there is no ambiguity.
The classes of continuous functions and continuously differentiable functions are denoted by $\mathcal C^{0}$ and $\mathcal C^{1}$, respectively. Recall that all functions  are assumed to be defined on a vector space  unless otherwise mentioned.



\begin{defn}\label{def:subgauss}
We call $\xi$ a sub-Gaussian random variable if and only if there exist constants $C, c>0$ such that $\mathbb P\{|\xi|\geq t\} \leq C e^{-c t^2}, \forall t>0$.  The scale (or $\psi_2$-norm) of $\xi$ is defined by $\sigma( \xi) =
\inf \{\sigma>0: \mathbb E\exp(\xi^2/\sigma^2) \leq 2\}$. More generally, $\bm\xi\in \mathbb R^p$ is called a sub-Gaussian random vector with scale  bounded by $\sigma$ if all one-dimensional marginals $\langle \bm\xi, \bm\alpha \rangle$ are sub-Gaussian satisfying $\|\langle \bm\xi, \bm\alpha \rangle\|_{\psi_2}\leq \sigma \|\bm\alpha \|_2$, $\forall \bm\alpha\in \mathbb R^{p}$.
\end{defn}

\begin{defn}\label{def:pseudometric}
We call $d$  a pseudo-metric if it satisfies    $d(\bsb{\eta}_1, \bsb{\eta}_2) = d(\bsb{\eta}_2, \bsb{\eta}_1)\ge 0$ and $d(\bsb{\eta}_1, \bsb{\eta}_2) \le d(\bsb{\eta}_1, \bsb{\eta}_3) + d(\bsb{\eta}_2, \bsb{\eta}_3)$,  for all $ \bsb{\eta}_1, \bsb{\eta}_2, \bsb{\eta}_3$.
\end{defn}

We state a first-order optimality condition satisfied by all local minimizers of $f$ that is directionally differentiable. The result is basic and we omit the proof.
It holds the key to deriving the so-called ``basic inequality'' in a variety of statistical learning problems. \begin{lemma} \label{lemma:local-opt}
Let $f:\mathbb R^p\rightarrow \mathbb R$ be a real-valued function and $C\subset \mathbb R^p$ be a convex set. Suppose that $f$ is directionally differentiable at   $\bm\beta^o$ that is a local minimizer to the problem $\min_{\bm\beta\in C}f(\bm\beta)$. Then $\delta f(\bm\beta^o; \bm h) \ge 0$ with   $\bm h = \bsb{\beta} - \bsb{\beta}^o$ or   $f(\bsb{\beta}) - f(\bsb{\beta}^o) \ge \breg_f(\bsb{\beta}, \bsb{\beta}^o)$ for all $\bsb{\beta}\in  C$.
\end{lemma}

\subsection{Proof of Lemma \ref{lemma:Delta_linear}}
\label{proofofLem1}

(i) This property is straightforward by definition:
\begin{equation}\nonumber
\begin{split}
&\bm{\Delta}_{a\psi + b\varphi}(\bm{\beta},\bm{\gamma})\\
=\,&(a\psi + b\varphi)(\bm{\beta}) - (a\psi + b\varphi)(\bm{\gamma}) - \delta(a\psi + b\varphi)(\bm{\gamma}; \bm{\beta} - \bm{\gamma})\\
=\,&a\big[\psi(\bm{\beta}) - \psi(\bm{\gamma}) - \delta\psi(\bm{\gamma}; \bm{\beta}-\bm{\gamma}) \big] + b\big[\varphi(\bm{\beta}) - \varphi(\bm{\gamma}) - \delta\varphi(\bm{\gamma};\bm{\beta}-\bm{\gamma}) \big]\\
=\,&a\bm{\Delta}_{\psi}(\bm{\beta},\bm{\gamma}) + b\bm{\Delta}_{\varphi}(\bm{\beta},\bm{\gamma}).
\end{split}
\end{equation}

\vspace{2ex}
(ii)  From  \cite[Theorem 23.1]{Rockafellar1970}, the convexity of $\psi$ implies the directional differentiability of $\psi$ and the positively homogenous convexity of $\delta\psi(\bm\beta;\cdot)$ for any given $\bm\beta$, and  we can write
\begin{equation}\label{Gateaux-def2}
\delta \psi(\bm{\beta}; \bm h) = \inf_{\epsilon >0} \frac{\psi(\bm{\beta} + \epsilon \bm h) - \psi(\bm{\beta})}{\epsilon}.
\end{equation}
Putting $\epsilon = 1$ and $\bm h = \bm\gamma-\bm\beta$ in \eqref{Gateaux-def2} gives $\delta \psi(\bm{\beta}; \bm\gamma - \bm\beta) \le \psi(\bm\gamma) - \psi(\bm\beta)$, thus $\bm\Delta_\psi(\bm\gamma,\bm\beta)\ge 0$.

Conversely, suppose that  $\psi$ defined on $\mathbb R^n$ is directionally differentiable  ($\delta \psi$  exists {and} is finite),   thus radially continuous, and $\breg_\psi\ge 0$. For any    $\bsb{s}_{\bsb{\beta}}: \langle \bsb{s}_{\bsb{\beta}}, \bsb{h}\rangle \le \delta \psi (\bsb{\beta}; \bsb{h}),\bsb{s}_{\bsb{\gamma}}: \langle \bsb{s}_{\bsb{\gamma}}, \bsb{h}\rangle \le \delta \psi (\bsb{\gamma}; \bsb{h})  \ \forall \bsb{h}$,
\begin{align}
&\psi(\bm\beta) - \psi(\bm\gamma) - \langle\bsb{s}_{\bsb{\gamma}}, \bsb{\beta} - \bsb{\gamma} \rangle \ge \breg_\psi(\bsb{\beta}, \bsb{\gamma})\ge 0, \label{cvx-1}\\
&\psi(\bm\gamma) - \psi(\bm\beta) -\langle\bsb{s}_{\bsb{\beta}}, \bsb{\gamma} - \bsb{\beta} \rangle\ge \breg_\psi(\bsb{\gamma}, \bsb{\beta})\ge 0.\label{cvx-2}
\end{align}
Adding them together gives $\langle\bsb{s}_{\bsb{\beta}}-\bsb{s}_{\bsb{\gamma}}, \bsb{\beta} - \bsb{\gamma} \rangle\ge 0$. This indicates the monotone property of   the Clarke-Rockafellar subdifferential of $\psi$, thereby its convexity  according to  \cite{Correa1994}.

(iii)  To show the first result, notice that $\breg_{\psi \circ \varphi} (\bsb{\beta}, \bsb{\gamma}) - \breg_{\psi  } (\varphi(\bsb{\beta}), \varphi(\bsb{\gamma})) = \lim_{\epsilon \rightarrow 0+} \{ \psi (\varphi(\bsb{\gamma})\allowbreak+ \epsilon (\varphi(\bsb{\beta}) -\varphi( \bsb{\gamma}))) -\psi (\varphi(\bsb{\gamma}+ \epsilon (\bsb{\beta} - \bsb{\gamma}))) \}/\epsilon = \langle \nabla \psi (\varphi(\bsb{\gamma})), \varphi(\bsb{\beta}) -\varphi( \bsb{\gamma}) \rangle - \delta (\psi \circ \varphi)(\bsb{\gamma}; \bsb{\beta} - \bsb{\gamma})   $. From $\psi\in \mathcal C^{1}$ and $ \varphi\in \mathcal C^{0}$,
\begin{align*}\delta (\psi \circ \varphi)(\bsb{\gamma}; \bsb{\beta} - \bsb{\gamma}) = &\lim_{\epsilon\rightarrow 0+} \{\psi ( \varphi(\bsb{\gamma}) + (\varphi(\bsb{\gamma} + \epsilon ( \bsb{\beta} - \bsb{\gamma}))- \varphi(\bsb{\gamma}))) - \psi(\varphi(\bsb{\gamma}))\}/{\epsilon} \\ = &\lim_{\epsilon\rightarrow 0+}\langle \nabla \psi (\varphi(\bsb{\gamma})), \varphi(\bsb{\gamma} + \epsilon ( \bsb{\beta} - \bsb{\gamma}))- \varphi(\bsb{\gamma})/\epsilon \rangle  \\
= & \langle \nabla \psi (\varphi(\bsb{\gamma})), \delta \varphi(\bsb{\gamma}; \bsb{\beta} - \bsb{\gamma}) \rangle.
\end{align*}
Using the definition of $\breg_{\varphi}(\bsb{\beta}, \bsb{\gamma})$ (the componentwise extension), we obtain the conclusion.

Next, we  prove the second result. Let  $\varphi: \mathbb R^p \rightarrow \mathbb R^n$ be the linear function $\varphi(\bsb{\beta} )= \bsb{X} \bsb{\beta} + \bsb{\alpha}$ with its Jacobian matrix $ \rD  \varphi (\bsb{\beta})    := [\rD_j \varphi_i(x)] =\bsb{X}\in \mathbb R^{n\times p}$. By definition, $\breg_{\psi \circ \varphi} (\bsb{\beta}, \bsb{\gamma}) = \psi(\varphi (\bsb{\beta})) - \psi(\varphi (\bsb{\gamma})) - \delta (\psi\circ \varphi)(\bm\gamma; \bm\beta-\bm\gamma)$ and
\begin{align*}
\delta (\psi\circ \varphi)(\bm\gamma; \bm\beta-\bm\gamma)    = & \lim_{\epsilon\rightarrow 0+} \{\psi (\varphi(\bsb{\gamma}+ \epsilon (\bsb{\beta} - \bsb{\gamma}))) - \psi (\varphi(\bsb{\gamma}))\}/\epsilon  \\
= & \lim_{\epsilon\rightarrow 0+} \{\psi (\varphi(\bsb{\gamma})+ \epsilon \rD\varphi(\bsb{\gamma}) (\bsb{\beta} - \bsb{\gamma}))) - \psi (\varphi(\bsb{\gamma}))\}/\epsilon \\
 = &  \, \delta  \psi  (\varphi(\bm\gamma); \rD\varphi(\bsb{\gamma}) (\bsb{\beta} - \bsb{\gamma})) = \delta  \psi  (\varphi(\bm\gamma); \varphi(\bm\beta)-\varphi(\bm\gamma)),
 \end{align*} from which it follows that
$ \breg_{\psi \circ \varphi} (\bsb{\beta}, \bsb{\gamma})   = \breg_{\psi  } (\varphi(\bsb{\beta}), \varphi(\bsb{\gamma}))$.

(iv) From Theorem 11 in \cite{Dini2006}, for any continuous function $f$ with finite Dini derivative $D^+f(x):=\mathop{\lim\sup}_{\epsilon\rightarrow 0+}(f(x+\epsilon)-f(x))/\epsilon$, if $D^+f(x)$ is integrable over $[a,b]$, $f(b) - f(a) = \int_a^b D^+f(x)\text{d}x$.
By definition, $\psi$ is continuous when restricted to the line segment  $[\bsb{\beta}, \bsb{\gamma}]$ (radial continuity). It follows that
\begin{equation}\nonumber
\begin{split}
\psi(\bm\beta) - \psi(\bm\gamma) &= \psi\big(\bm\gamma + t(\bm\beta-\bm\gamma)\big)\Big|_{t=0}^1\\
&= \int_0^1 \lim_{\epsilon\rightarrow 0+}\frac{1}{\epsilon}\Big[\psi\big(\bm\gamma + (t+\epsilon)(\bm\beta-\bm\gamma)\big)-\psi\big(\bm\gamma + t(\bm\beta-\bm\gamma)\big)\Big]\text{d}t\\
&= \int_0^1 \lim_{\epsilon\rightarrow 0+}\frac{1}{\epsilon}\Big[\psi\big(\bm\gamma + t(\bm\beta-\bm\gamma) +\epsilon(\bm\beta-\bm\gamma)\big)-\psi\big(\bm\gamma + t(\bm\beta-\bm\gamma)\big)\Big]\text{d}t\\
&= \int_0^1 \delta\psi\big(\bm\gamma + t(\bm\beta-\bm\gamma);\,\bm\beta-\bm\gamma \big)\text{d}t.
\end{split}
\end{equation}
Hence, $\bm\Delta_\psi$ can be formulated by
\begin{equation}\nonumber
\bm\Delta_\psi(\bm\beta,\bm\gamma) = \int_0^1 \Big[\delta\psi\big(\bm\gamma + t(\bm\beta-\bm\gamma);\,\bm\beta-\bm\gamma \big) - \delta \psi(\bm\gamma; \bm\beta-\bm\gamma)\Big] \rd t.
\end{equation}

\subsection{Proof of Lemma \ref{lemma:Delta_Delta}} \label{subsec:lemma2}
(i)
First,    if $\delta\psi(\bm\alpha;\cdot-\bm\alpha)$ is directionally differentiable, then
\begin{equation} \label{Delta_diff}
\bm\Delta_\psi(\bm\beta,\bm\gamma) - \bm\Delta_{\bm\Delta_\psi(\cdot,\bm\alpha)}(\bm\beta,\bm\gamma) = \bm\Delta_{\delta\psi(\bm\alpha;\cdot-\bm\alpha)}(\bm\beta,\bm\gamma)
\end{equation}
for any $\bm\alpha,\bm\beta,\bm\gamma$.
In fact,  $\bm\Delta_\psi(\bm\beta,\bm\gamma) - \bm\Delta_{\bm\Delta_\psi(\cdot,\bm\alpha)}(\bm\beta,\bm\gamma) = \bm\Delta_{\psi(\cdot) - \bm\Delta_\psi(\cdot,\bm\alpha)}(\bm\beta,\bm\gamma) = \bm\Delta_{\psi(\bm\alpha)+\delta\psi(\bm\alpha;\cdot-\bm\alpha)}\allowbreak(\bm\beta,\bm\gamma)= \bm\Delta_{\delta\psi(\bm\alpha;\cdot-\bm\alpha)}(\bm\beta,\bm\gamma)$.

Accordingly, when $\psi$ is convex, which means  $\delta\psi(\bm\alpha;\cdot-\bm\alpha)$ is convex as well (cf. Section \ref{proofofLem1}),    $\bm\Delta_{\delta\psi(\bm\alpha;\cdot-\bm\alpha)}(\bm\beta,\bm\gamma) \ge 0$ by Lemma \ref{lemma:Delta_linear}. The result under concavity can be similarly proved.

(ii)
Let
\begin{align}
q(\cdot;\bm\alpha) = \delta\psi(\bm\alpha;\cdot-\bm\alpha).
\end{align} We want to show for $\bsb{\alpha} = \theta \bsb{\beta} + (1-\theta) \bsb{\gamma}$ with $\theta\le 0 $ or $\theta\ge 1$, $\bm\Delta_{q(\cdot;\bm\alpha)}(\bm\beta,\bm\gamma)
$
is well-defined and equals 0. This is  intuitive due to the linearity of  $q$ when restricted to $[\bsb{\beta}, \bsb{\gamma}]$, assuming $\bsb{\beta}-\bsb{\alpha}$ and $\bsb{\gamma}-\bsb{\alpha}$ are positively collinear.

To verify it,  by definition,
\begin{equation} \nonumber
\begin{split}
\delta q(\cdot;\bm\alpha)(\bm \gamma;\bm\beta-\bm\gamma) = &\lim_{\epsilon\rightarrow 0+}[q(\bm\gamma+\epsilon(\bm\beta-\bm\gamma);\bm\alpha) - q(\bm\gamma;\bm\alpha)]/\epsilon\\
= &\lim_{\epsilon\rightarrow 0+}[\delta\psi(\bm\alpha;\bm\gamma+\epsilon(\bm\beta-\bm\gamma)-\bm\alpha) - \delta\psi(\bm\alpha;\bm\gamma-\bm\alpha)]/\epsilon\\
= &\lim_{\epsilon\rightarrow 0+}[\delta\psi(\bm\alpha; (\theta - \epsilon)( \bm\gamma-\bm\beta)) - \delta\psi(\bm\alpha; \theta (\bm\gamma-\bm\beta))]/\epsilon\\
= &\lim_{\epsilon\rightarrow 0+}[\delta\psi(\bm\alpha; ( \epsilon-\theta)(\bm\beta-\bm\gamma)) - \delta\psi(\bm\alpha; (-\theta) (\bm\beta-\bm\gamma))]/\epsilon,\end{split}
\end{equation}
and so with $\theta>0$,
\begin{equation} \nonumber
\begin{split}
\delta q(\cdot;\bm\alpha)(\bm\gamma;\bm\beta-\bm\gamma) =  &\lim_{\epsilon\rightarrow 0+}[(\theta-\epsilon)\delta\psi(\bm\alpha;\bm\gamma-\bm\beta) -\theta \delta\psi(\bm\alpha;\bm\gamma-\bm\beta)]/\epsilon\\
= & -\delta\psi(\bm\alpha;\bm\gamma-\bm\beta) ,
\end{split}
\end{equation}
and  with $\theta\le 0$,
\begin{equation} \nonumber
\begin{split}
\delta q(\cdot;\bm\alpha)(\bm\gamma;\bm\beta-\bm\gamma) =  &\lim_{\epsilon\rightarrow 0+}[( \epsilon-\theta)\delta\psi(\bm\alpha; ( \bm\beta-\bm\gamma)) - (-\theta)\delta\psi(\bm\alpha;  (\bm\beta-\bm\gamma ))]/\epsilon\\
= &\,\delta\psi(\bm\alpha;\bm\beta-\bm\gamma) .
\end{split}
\end{equation}
The above derivation also guarantees the existence of $\bm\Delta_{\bm\Delta_\psi(\cdot,\bm\alpha)}(\bm\beta,\bm\gamma)$.
Now, as $\theta\ge 1$, $\langle \bsb{\beta} - \bsb{\alpha}, \bsb{\gamma} - \bsb{\beta}\rangle\ge 0$ and so $q(\bsb{\beta};\bm\alpha) - q(\bsb{\gamma};\bm\alpha) -\delta q(\cdot;\bm\alpha)(\bm\gamma;\bm\beta-\bm\gamma) = \delta \psi(\bm\alpha ;\bm\beta-\bm\alpha)  - \delta \psi(\bm\alpha ;\bm\gamma-\bm\alpha) + \delta \psi(\bm\alpha ;\bm\gamma-\bm\beta) = 0$. As  $\theta\le 0$,  $\langle \bsb{\beta} - \bsb{\alpha}, \bsb{\gamma} - \bsb{\beta}\rangle\le 0$ and $q(\bsb{\beta};\bm\alpha) - q(\bsb{\gamma};\bm\alpha) -\delta q(\cdot;\bm\alpha)(\bm\gamma;\bm\beta-\bm\gamma) = \delta \psi(\bm\alpha ;\bm\beta-\bm\alpha)  - \delta \psi(\bm\alpha ;\bm\gamma-\bm\alpha) - \delta \psi(\bm\alpha ;\bm\beta - \bm\gamma) = 0$.

(iii) By definition, we have
\begin{align*}
&\ \delta q(\cdot;\bm\alpha)(\bm z;\bm\beta-\bm\gamma)\\
=& \lim_{\epsilon_2\rightarrow 0+}\frac{1}{\epsilon_2} \{q ( \bm z + \epsilon_2(\bm\beta-\bm\gamma);\bm\alpha ) - q ( \bm z ;\bm\alpha)\}\\
=&  \lim_{\epsilon_2\rightarrow 0+}\frac{1}{\epsilon_2} \{\delta \psi (\bm\alpha; \bm z + \epsilon_2(\bm\beta-\bm\gamma) -\bm\alpha ) -\delta \psi (\bm\alpha; \bm z  -\bm\alpha ) \}.
\end{align*}

Under the restricted linearity condition $\delta\psi (\bsb{\alpha}; \bsb{h})=\langle g(\bsb{\alpha}), \bsb{h}\rangle,  \forall \bsb{h}\in [\bsb{\beta}-\bsb{\alpha}, \bsb{\gamma}-\bsb{\alpha}]$, for  $\bsb{z} = \bsb{\gamma}+ t(\bsb{\beta} - \bsb{\gamma})$ with $t\in [0,1)$, \begin{align*}
 \delta q(\cdot;\bm\alpha)(\bm z;\bm\beta-\bm\gamma)
= & \lim_{\epsilon_2\rightarrow 0+}\frac{1}{\epsilon_2}  \langle g (\bm\alpha), \bm z + \epsilon_2(\bm\beta-\bm\gamma) -\bm\alpha -
\bm z  +\bm\alpha  \rangle \\
=&\langle g (\bm\alpha), \bm\beta-\bm\gamma  \rangle.
\end{align*}
  Under the restricted continuity condition $\lim_{\epsilon\rightarrow 0+} \delta \psi(\bsb{\alpha}+\epsilon \bsb{h}; \bsb{\beta} - \bsb{\gamma}) =\delta \psi(\bsb{\alpha}; \bsb{\beta} - \bsb{\gamma}), \forall  \bsb{h}\in[\bsb{\beta}-\bsb{\alpha}, \bsb{\gamma}-\bsb{\alpha}] $, for  $\bsb{z} = \bsb{\gamma}+ t(\bsb{\beta} - \bsb{\gamma})$ with $t\in [0,1)$,

\begin{equation}\nonumber
\begin{split}
&\ \delta q(\cdot;\bm\alpha)(\bm z;\bm\beta-\bm\gamma)\\
=& \lim_{\epsilon_2\rightarrow 0+}\frac{1}{\epsilon_2}\bigg\{\lim_{\epsilon_1\rightarrow 0+}\frac{1}{\epsilon_1}\Big[\psi\big(\bm\alpha+\epsilon_1[\bm z + \epsilon_2(\bm\beta-\bm\gamma)-\bm\alpha]\big) - \psi(\bm\alpha)\Big]\\
&~~~~~~~~~~~~~~~ - \lim_{\epsilon_1\rightarrow 0+}\frac{1}{\epsilon_1}\Big[\psi\big(\bm\alpha+\epsilon_1(\bm z-\bm\alpha)\big) - \psi(\bm\alpha)\Big] \bigg\}\\
=& \lim_{\epsilon_2\rightarrow 0+} \lim_{\epsilon_1\rightarrow 0+} \frac{1}{\epsilon_1\epsilon_2}\Big[\psi\big((1-\epsilon_1)\bm\alpha + \epsilon_1\bm z-\epsilon_1\epsilon_2 \bm\gamma + \epsilon_1\epsilon_2\bm\beta\big) - \psi\big((1-\epsilon_1)\bm\alpha + \epsilon_1\bm z\big)\Big]\\
=& \lim_{\epsilon_2\rightarrow 0+} \lim_{\epsilon_1\rightarrow 0+} \frac{1}{\epsilon_1\epsilon_2} \int_0^1 \delta \psi\big((1-\epsilon_1)\bm\alpha + \epsilon_1\bm z+\epsilon_1\epsilon_2s(\bm\beta-\bm\gamma);\epsilon_1\epsilon_2(\bm\beta-\bm\gamma) \big)\rd s\\
=& \lim_{\epsilon_2\rightarrow 0+} \lim_{\epsilon_1\rightarrow 0+}   \int_0^1 \delta \psi\big((1-\epsilon_1)\bm\alpha + \epsilon_1\bm z+\epsilon_1\epsilon_2s(\bm\beta-\bm\gamma);  \bm\beta-\bm\gamma  \big)\rd s\\
=& \lim_{\epsilon_2\rightarrow 0+}    \int_0^1 \lim_{\epsilon_1\rightarrow 0+}\delta \psi\big( \bm\alpha + \epsilon_1(\bm z + \epsilon_2s(\bm\beta-\bm\gamma)-\bm \alpha);  \bm\beta-\bm\gamma  \big)\rd s\\
=&  \ \delta\psi(\bm\alpha;  \bm\beta-\bm\gamma),
\end{split}
\end{equation}
where we used  the positive homogeneity of $\delta \psi(\bsb{\alpha}; \cdot)$ and the dominated convergence theorem. (The integral is well-defined due to the boundedness  and Lebesgue measurability of the integrand.)

The two sets of conditions are not equivalent in multiple dimensions. But in either case, $\delta q(\cdot;\bm\alpha)(\bm z;\bm\beta-\bm\gamma)$ is  a term  independent of $\bm z$.
Hence  by  Lemma  \ref{lemma:Delta_linear} (iv),
$$\bm\Delta_{\bm\Delta_q(\cdot,\bm\alpha)}(\bm\beta,\bm \gamma) = \int_0^1 \Big[\delta q(\cdot;\bm\alpha)\big(\bm\gamma + t(\bm\beta-\bm\gamma);\,\bm\beta-\bm\gamma \big) - \delta q(\cdot;\bm\alpha)(\bm\gamma; \bm\beta-\bm\gamma)\Big] \text{d}t=0.$$

\subsection{Proof of Lemma \ref{lemma:glmBreg}}
\label{proofofexp}
%
 (i)
 Let $\varphi = b^*$.   Then for all subgradient $\bsb{g}\in \partial \varphi(\bsb{z})$,   $(\bsb{g}, \bsb{z})$ makes a conjugate pair and so $ \langle  \bsb{g},  \bsb{z}  \rangle =    b( \bsb{g}) + \varphi( \bsb{z} )$ (see, e.g.,  \cite{Rockafellar1970}). Using the shorthand notation $( \partial \varphi(\bsb{z}),  \bsb{z} ) $, we represent it as  $ \langle  \partial \varphi(   \bsb{z}),  \bsb{z}  \rangle =    b( \partial \varphi (   \bsb{z})) + \varphi( \bsb{z} )$. Therefore,
\begin{align*}
 \sigma^2 l_0(\bsb{\eta}; \bsb{z})  +b^*(\bsb{z})  & = - \langle \bsb{z}, \bsb{\eta} \rangle + b(\bsb{\eta}) +\varphi(\bsb{z})  \\ &  =  - \langle \bsb{z}, \bsb{\eta} \rangle + b(\bsb{\eta})+  \langle\bsb{z}, \partial \varphi(   \bsb{z})     \rangle -    b( \partial \varphi(   \bsb{z})) \\
&=  b(\bsb{\eta}) - b(\partial \varphi (   \bsb{z})) -  \langle \bsb{z}, \bsb{\eta} -  \partial \varphi(   \bsb{z})\rangle \\ &=  b(\bsb{\eta}) - b( \partial \varphi(   \bsb{z})) -  \langle \nabla b( \partial \varphi( \bsb{z})), \bsb{\eta} -  \partial \varphi(   \bsb{z})\rangle \\& = \breg_b (\bsb{\eta} ,  \partial \varphi(   \bsb{z})).
\end{align*}
When $p_{\bsb{\eta}}$ is minimal,  $\mathcal M$ is full-dimensional and  the canonical link $g = (\nabla b )^{-1}$ is well-defined on $\mathcal M^\circ$  (Proposition 3.1 and Proposition 3.2 in \citep{WainJordan08}  can be slightly modified to include the dispersion parameter), and so  $(g(\bsb{z}), \nabla b(g(\bsb{z})))$ or $(g(\bsb{z}),  \bsb{z} )$ makes a conjugate pair.

 (ii) Let $\bsb{\mu}(\bsb{\eta}) = \nabla b(\bsb{\eta})$ or $\bsb{\mu}$ for brevity. It follows that $\bsb{\eta}\in \partial \varphi (\bsb{\mu})$ and so
\begin{align*}
- \langle \bsb{z}, \bsb{\eta}\rangle  + b(\bsb{\eta}) + b^{*}(\bsb{z}) & =- \langle \bsb{z}, \bsb{\eta}\rangle +  \langle \bsb{\mu}, \bsb{\eta}\rangle - \varphi(\bsb{\mu}) + \varphi(\bsb{z})  \\
& = - \langle \bsb{z}- \bsb{\mu}, \bsb{\eta}\rangle   - \varphi(\bsb{\mu}) + \varphi(\bsb{z})\\
& \ge - \delta \varphi (\bsb{\mu};  \bsb{z}- \bsb{\mu})- \varphi(\bsb{\mu}) + \varphi(\bsb{z}) = \breg_{\varphi}(\bsb{z}, \bsb{\mu}),
\end{align*}
where  the  inequality is due to \cite[Theorem 23.2]{Rockafellar1970}. We claim that  the inequality is actually an equality.

Indeed, if there exist    $\bsb{\eta}_1, \bsb{\eta}_2\in \partial \varphi(\bsb{\mu})$ with $\bsb{\eta}_1\ne  \bsb{\eta}_2$, then $\sym{\breg}_b(\bsb{\eta}_2, \bsb{\eta}_1)=\langle \nabla b(\bsb{\eta}_2) - \nabla b(\bsb{\eta}_1),    \bsb{\eta}_2 - \bsb{\eta}_1 \rangle = 0$ and so ${\breg}_b(\bsb{\eta}_2, \bsb{\eta}_1)=0$ since $b$ is convex. Therefore, for any random vector $\bsb{y}$ following  $p_{\bsb{\eta}}$ in the exponential family, where  $\bsb{\eta} = t \bsb{\eta}_1 + (1-t) \bsb{\eta}_2$, $t\in (0,1)$, $$Var( (\bsb{\eta}_2 - \bsb{\eta}_1)^T \bsb{y}) = 0,$$ which can be obtained from Proposition 3.1 of \cite{WainJordan08}.  Because $\exp(  (\langle \cdot, \bsb{\eta}\rangle -  {b}(\bsb{\eta} ))/\sigma^{2})  >0$ for any $\bsb{\eta}\in \Omega$, we have      $\langle\bsb{\eta}_2 - \bsb{\eta}_1, \bsb{z}\rangle = c$  for almost every   $\bsb{z}\in \mathcal Y^{n}$ with respect to  $\nu$ (i.e.,      $p_{\bsb{\eta}}$ is not minimal). It follows that
$$
 \langle \bsb{z}- \bsb{\mu}, \bsb{\eta}_1 - \bsb{\eta}_2\rangle = 0.
$$
Finally, from $\delta \varphi(\bsb{\mu};  \bsb{h})=\sup\{\langle \bsb{g}, \bsb{h}\rangle:\bsb{g}\in \partial \varphi(\bsb{\mu} \} $   \cite[Theorem 23.4]{Rockafellar1970}, the claim is true.

In the case that       $p_{\bsb{\eta}}$ is also  minimal, $\varphi$ can be shown to be strictly convex and differentiable on $\mathcal M^\circ$ \cite[Theorem 26.4]{Rockafellar1970}.

 (iii) Let $\rd P_{\bsb{\eta}}  = p_{\bsb{\eta}}\rd \nu_0$. By definition, $$\mbox{KL}(  p_{\bsb{\eta}_1} , p_{\bsb{\eta}_2}) = \int \log(\rd P_{\bsb{\eta}_1}/\rd P_{\bsb{\eta}_2}) \rd P_{\bsb{\eta}_1} = \int p_{\bsb{\eta}_1}\log( p_{\bsb{\eta}_1}/\rd p_{\bsb{\eta}_2})  \rd \nu_0  $$ and so
\begin{align*}
\mbox{KL}(  p_{\bsb{\eta}_1} , p_{\bsb{\eta}_2}) & = \int \log {\{e^{ ( \langle \bsb{y},  \bsb{\eta}_1\rangle - b(\bsb{\eta}_1))/\sigma^2 - c(\bsb{y}, \sigma^2)}} / {e^{ ( \langle  \bsb{y},  \bsb{\eta}_2\rangle - b(\bsb{\eta}_2))/\sigma^2 - c(\bsb{y}, \sigma^2)}}\}  \rd P_{\bsb{\eta}_1} \\
 & = \frac{1}{\sigma^2}\int \langle \bsb{y},   \bsb{\eta}_1 - \bsb{\eta}_2\rangle -  b(\bsb{\eta}_1) + b(\bsb{\eta}_2) \rd P_{\bsb{\eta}_1} \\
& = \frac{1}{\sigma^2} \{b(\bsb{\eta}_2) - b(\bsb{\eta}_1)+ \int \langle  \bsb{y},   \bsb{\eta}_1 - \bsb{\eta}_2 \rangle   \rd P_{\eta_1}\} \\
& = \frac{1}{\sigma^2}\{ b(\bsb{\eta}_2) - b(\bsb{\eta}_1) + \langle \nabla b (\bsb{\eta}_1) , \bsb{\eta}_1 - \bsb{\eta}_2 \rangle\} \\
& = \breg_b(\bsb{\eta}_2, \bsb{\eta}_1) / \sigma^2,
\end{align*}
where the third equality is due to $\EE_{\bsb{y} \sim p_{\bsb{\eta}_1}} \bsb{y} = \nabla b (\bsb{\eta}_1)$ under $\bsb{\eta}_1\in \Omega^\circ$ (which can be derived from  Proposition 3.1 of \cite{WainJordan08}).
Moreover, from Lemma   \ref{lemma:Delta_linear},   $\sigma^2\breg_{l_0} (\bsb{\eta}_2, \bsb{\eta}_1) = \breg_b (\bsb{\eta}_2, \bsb{\eta}_1)$.


\subsection{Proof of Lemma \ref{lemma:degeneracy}}

In this proof, all directional derivatives are with respect with $\bm\beta$. The result of (i) is trivial from the construction of $g$. For (ii), by definition, we have $\delta g(\bm\beta;\bm\beta^-,\bm h) = \delta f(\bm\beta;\bm h) + \delta \psi(\bm\beta;\bm h)- \delta q(\bm\beta;\bm\beta^-,\bm h)$ with $q(\bm\beta; \bm\beta^-) = \delta\psi(\bm\beta^-;\bm\beta-\bm\beta^-)$. It follows from $q(\bm\beta; \bm\beta^-) = \lim_{\epsilon\rightarrow 0+}[\psi(\bm\beta^-+\epsilon(\bm\beta-\bm\beta^-)) - \psi(\bm\beta^-)]/\epsilon$ that
\[\begin{split}
\delta q(\bm\beta;\bm\beta^-,\bm h) = &\lim_{\epsilon'\rightarrow 0+}[q(\bm\beta+\epsilon'\bm h;\bm\beta^-) - q(\bm\beta;\bm\beta^-)]/\epsilon'\\
= &\lim_{\epsilon'\rightarrow 0+}\big\{(1/\epsilon')\lim_{\epsilon\rightarrow 0+}[\psi(\bm\beta^-+\epsilon(\bm\beta+\epsilon'\bm h-\bm\beta^-))-\psi(\bm\beta^-)]/\epsilon\big\}\\
&- \lim_{\epsilon'\rightarrow 0+}\delta\psi(\bm\beta^-;\bm\beta-\bm\beta^-)/\epsilon'.
\end{split}\]
When $\bm\beta^-=\bm\beta$, $\delta\psi(\bm\beta^-;\bm\beta-\bm\beta^-) = 0$ and so
\[\begin{split}
\delta q(\bm\beta;\bm\beta^-,\bm h)|_{\bm\beta^-=\bm\beta} &=\lim_{\epsilon'\rightarrow 0+}\lim_{\epsilon\rightarrow 0+}[\psi(\bm\beta+\epsilon(\epsilon'\bm h))-\psi(\bm\beta)]/(\epsilon\epsilon')\\
&=\lim_{\epsilon''\rightarrow 0+}[\psi(\bm\beta+\epsilon''\bm h)-\psi(\bm\beta)]/\epsilon''\\
&=\,\delta\psi(\bm\beta;\bm h).
\end{split}\]
The above argument also guarantees the existence of $\delta g(\bm\beta;\bm\beta^-,\bm h)|_{\bm\beta^-=\bm\beta}$.
Therefore, $\delta g(\bm\beta;\bm\beta^-,\bm h)|_{\bm\beta^-=\bm\beta} = \delta f(\bm\beta;\bm h)$ for any $\bm\beta$ and $\bm h$.

\subsection{Proof of Lemma \ref{lemma:C}}
\label{subsec:proofoflemma:C}
All results in  Lemma \ref{lemma:Delta_linear} and Lemma \ref{lemma:Delta_Delta} can be formulated for $\mathbf C$. For example,
$\psi$ is convex if and only if $\mathbf C_\psi \ge 0$,    $\mathbf C_{a\phi+b\varphi} = a\mathbf C_\phi + b\mathbf C_\varphi$, $\breg_\psi\ge \mu \Breg_2$ implies $\mathbf C_\psi\ge \mu \mathbf C_2$ since    $\mathbf C_2(\bm\alpha,\bm\beta,\theta):=\mathbf C_{\|\cdot\|_2^2/2}(\bm\alpha,\bm\beta,\theta)= \theta(1-\theta) \Breg_2(\bm\alpha,\bm\beta)$, and so on. To show (i), we have
\begin{equation} \nonumber
\begin{split}
&\mathbf C_\psi(\bm\alpha,\bm\beta,\theta) + \bm\Delta_\psi(\theta\bm\alpha + (1-\theta)\bm\beta, \bm\alpha) \\
=\,& \theta\psi(\bm\alpha)+(1-\theta)\psi(\bm\beta) - \psi(\theta\bm\alpha+(1-\theta)\bm\beta)\\
&+\psi(\theta\bm\alpha+(1-\theta)\bm\beta)-\psi(\bm\alpha)-\delta\psi(\bm\alpha;\theta\bm\alpha+(1-\theta)\bm\beta-\bm\alpha)\\
=\,& (\theta-1)\psi(\bm\alpha) + (1-\theta)\psi(\bm\beta) - \delta\psi(\bm\alpha;(1-\theta)(\bm\beta-\bm\alpha))\\
=\,& (1-\theta)\psi(\bm\beta) - (1-\theta)\psi(\bm\alpha) - (1-\theta)\delta\psi(\bm\alpha;\bm\beta-\bm\alpha)\\
=\,&(1-\theta)\bm\Delta_\psi(\bm\beta,\bm\alpha).
\end{split}
\end{equation}

Similar to the proof of   Lemma \ref{lemma:Delta_Delta}, let $q(\cdot;\bm\alpha) = \delta\psi(\bm\alpha;\cdot-\bm\alpha)$.   Then   $$
\mathbf C_{\bm{\Delta}_\psi(\cdot,\bm\alpha)}(\bm\beta,\bm\gamma,\theta) -\mathbf C_\psi(\bm\beta,\bm\gamma,\theta)=  \mathbf C_{\psi(\bm\alpha)+q(\cdot;\bm\alpha)}(\bm\beta,\bm\gamma,\theta)=  \mathbf C_{q(\cdot;\bm\alpha)}(\bm\beta,\bm\gamma,\theta), $$
 without requiring the directional differentiability of $q(\cdot;\bm\alpha) $.
We can show analogous results to Lemma \ref{lemma:Delta_Delta}.
For example, for any convex   $\psi$,   from   the   positively homogenous convexity of $q$,      $$\mathbf C_{\bm{\Delta}_\psi(\cdot,\bm\alpha)}\le \mathbf C_\psi$$ holds for any $\bsb{\alpha}$, and  for  $\bsb{\alpha}=(1-\theta) \bsb{\gamma} +  \theta  \bsb{\beta}$ with $\theta\not\in(0,1)$,   $$ \mathbf C_{\bm\Delta_\psi(\cdot,\bm\alpha)}(\bm\beta,\bm\gamma)= \mathbf C_\psi(\bm\beta,\bm\gamma)  $$ follows from the restricted linearity of $q$. In particular, when $\nabla \psi(\bsb{\alpha})$ exists,  $q$ is  linear and so $\mathbf C_{q(\cdot;\bm\alpha)}\equiv 0$ which gives the result in (ii).
\subsection{Proof of Theorem \ref{thm:minimaxglm}}
\label{subsec:minimaxglm}


The theorem can be proved  based on Theorem 6.1 of \cite{Lounici2011} and property (iii) of Lemma \ref{lemma:glmBreg}  in Section \ref{sec:Bregnotation}. 
We  give some details for the   second conclusion;  the proof of the first  follows similar lines and is easier. Consider a signal subclass
\begin{equation}\nonumber
\mathcal B^1 = \{{\bm\beta} : \beta_j\in\{0,\tau R\}, \|\bm\beta\|_0\le  s^*\},
\end{equation}
where $$R = [\sigma (\log(ep/s^*))^{1/2}/{  \overline \kappa  }^{1/2}]\wedge M$$ and $1> \tau>0$ is a small constant to be chosen later.
Clearly, $\mathcal B^1\in\mathcal B(s^* ,M)$. By Stirling's approximation, $\log |\mathcal B^1 | \ge \log {p \choose s^*} \ge s^*\log(p/s^*) \ge cs^*\log(ep/s^*)$
for some universal constant $c$.

Let $\rho({\bm\beta}_1,{\bm\beta}_2) = \|{\bm\beta}_1-{\bm\beta}_2\|_0$, the Hamming distance between ${\bm\beta}_1$ and ${\bm\beta}_2$. By Lemma A.3 in   \cite{Rigollet11}, there exists a subset $\mathcal B^{10}  \subset \mathcal B^1$ such that $\bsb{0}\in \mathcal B^{10} $ and
\begin{equation}\nonumber
\log |\mathcal B^{10}| \ge c_1s^*\log(ep/s^*), \rho({\bm\beta}_1,{\bm\beta}_2) \ge c_2s^*,\forall {\bm\beta}_1,{\bm\beta}_2\in \mathcal B^{10}, {\bm\beta}_1\ne {\bm\beta}_2
\end{equation}
for some universal constants $c_1,c_2>0$. Then \begin{align}\label{minimax_ineq-1}
\|\bsb{X}{\bm\beta}_1-\bsb{X}{\bm\beta}_2\|_2^2 \ge\underline \kappa\|{\bm\beta}_1-{\bm\beta}_2\|_2^2 = \underline \kappa \tau^2 R^2\rho({\bm\beta}_1,{\bm\beta}_2) \ge c_2\underline \kappa \tau^2R^2s^*\end{align} 
for any ${\bm\beta}_1,{\bm\beta}_2\in\mathcal B^{10}$, ${\bm\beta}_1\ne {\bm\beta}_2$.

By Lemma \ref{lemma:glmBreg} (iii), since $\Omega$ is open, for any ${\bm\beta}\in\mathcal B^{10}$, we have
\begin{equation}\nonumber
\mbox{KL}(p_{\bsb{\beta}}, p_{{\bm 0}} ) = \breg_{l_0} (    \bsb{0} , \bsb{X} { \bsb{\beta}}) \le   {\tau^2}\overline\kappa R^2s^*/(2\sigma^2).
\end{equation}
 Therefore,
\begin{equation} \label{minimax_ineq-2}
\frac{1}{|\mathcal B^{10}|-1}\sum_{ \bsb{\beta}\in\mathcal B^{10}\setminus \{\bsb{0}\}}\mbox{KL}( p_{ \bsb{\beta}}, p_{{\bm0} } ) \le {\tau^2}   s^*\log(ep/s^*).
\end{equation}

Combining \eqref{minimax_ineq-1} and \eqref{minimax_ineq-2} and choosing a sufficiently small value for $\tau$, we can apply Theorem 2.7 of \cite{Tsybakov2008} to get the desired lower bound.

\subsection{Proof of Proposition \ref{th:comp_funcval}} \label{proof:comp_funcval}

We first introduce a lemma.


\begin{lemma}\label{lemma:f_triangle}
 For the sequence of iterates $\{\bm\beta^{(t)}\}$
defined by \eqref{BregIter} starting from an arbitrary  point $\bm{\beta}^{(0)}$,  if  $f(\cdot)$ and $g(\cdot; \bsb{\beta}^{(t)})$ are  directionally differentiable,  the following inequality holds for any $\bm{\beta}$ and $t \geq 0$
\begin{align}\label{f_triangle-0}
\begin{split}
&f(\bm{\beta}) + \bm\Delta_\psi(\bm{\beta},\bm{\beta}^{(t)})   \\\ge \ &    f(\bm{\beta}^{(t+1)}) + \bm\Delta_\psi(\bm{\beta}^{(t+1)},\bm{\beta}^{(t)}) + (\bm\Delta_{\breg_\psi(\cdot; \bsb{\beta}^{(t)})}+\bm{\Delta}_f)(\bm{\beta},\bm{\beta}^{(t+1)}).
\end{split}\end{align}
\end{lemma}
It can be proved by Lemma \ref{lemma:local-opt} and Lemma \ref{lemma:Delta_linear} (details omitted).
Rearranging \eqref{f_triangle-0} gives
\begin{equation} \nonumber
\begin{split}
&f(\bm{\beta}^{(t+1)}) - f(\bm{\beta}) + \bm\Delta_\psi(\bm{\beta}^{(t+1)}, \bm{\beta}^{(t)}) + \bm{\Delta}_f(\bm{\beta},\bm{\beta}^{(t+1)})\\
\le\,&\bm\Delta_\psi(\bm{\beta}, \bm{\beta}^{(t)}) - \bm\Delta_{\breg_\psi(\cdot; \bsb{\beta}^{(t)})}(\bm{\beta}, \bm{\beta}^{(t+1)}).
\end{split}
\end{equation}
Under $\bm\Delta_\psi(\bm{\beta}^{(t+1)}, \bm{\beta}^{(t)}) + \bm{\Delta}_f(\bm{\beta},\bm{\beta}^{(t+1)}) \geq 0$, we have
\begin{equation} \label{funcval-1}
f(\bm{\beta}^{(t+1)}) - f(\bm{\beta}) \leq \bm\Delta_\psi(\bm{\beta}, \bm{\beta}^{(t)}) - \bm\Delta_{\breg_\psi(\cdot; \bsb{\beta}^{(t)})}(\bm{\beta}, \bm{\beta}^{(t+1)}).
\end{equation}
By Lemma \ref{lemma:Delta_Delta}, when $\psi$ is differentiable, $ \bm\Delta_{\breg_\psi(\cdot; \bsb{\beta}^{(t)})}$ is well-defined and equals $\breg_\psi$. Adding up the corresponding inequality for $t=0,1,\ldots, T$ leads to
\begin{equation} \nonumber
\sum_{t=0}^T[f(\bm{\beta}^{(t+1)}) - f(\bm{\beta})]  \leq \bm\Delta_\psi(\bm{\beta}, \bm{\beta}^{(0)}) - \bm\Delta_\psi(\bm{\beta}, \bm{\bm{\beta}}^{(T+1)}).
\end{equation}
Therefore,
\begin{equation} \nonumber
\mathop{\avg}_{0\le t\le T}f(\bm{\beta}^{(t+1)}) - f(\bm{\beta}) \leq \frac{1}{T+1}[\bm\Delta_\psi(\bm{\beta}, \bm{\beta}^{(0)}) - \bm\Delta_\psi(\bm{\beta}, \bm{\bm{\beta}}^{(T+1)})].
\end{equation}
 Note that under just the directional differentiability of  $\breg_\psi(\cdot; \bsb{\beta}^{(t)})$, \eqref{th1-condition} can be replaced by $\bm\Delta_\psi(\bm{\beta}^{(t+1)},\bm{\beta}^{(t)}) +( \bm{\Delta}_f+ \bm\Delta_{\breg_\psi(\cdot; \bsb{\beta}^{(t)})}- \bm\Delta_\psi)(\bm{\beta},\bm{\beta}^{(t+1)}) \geq 0,~0\le t\le T$.

In the specific case that both $f$ and $\psi$ are convex, \eqref{th1-condition} is always satisfied by Lemma \ref{lemma:Delta_linear} and letting $\bm\beta = \bm\beta^{(t)}$ in \eqref{funcval-1} gives
\begin{equation} \nonumber
f(\bm{\beta}^{(t+1)}) - f(\bm{\beta}^{(t)}) \le - \bm\Delta_\psi(\bm{\beta}^{(t)}, \bm{\beta}^{(t+1)}) \le 0.
\end{equation}
Hence
$f(\bm\beta^{(T+1)}) - f(\bm\beta) = \min_{0\le t\le T} f(\bm\beta^{(t+1)}) - f(\bm\beta) \le \mathop{\avg}_{0\le t\le T}f(\bm{\beta}^{(t+1)}) - f(\bm{\beta}).$
The proof is complete.

\subsection{Proof of Proposition \ref{th:strongcvx}}
\label{subsec:proofofscvx}

Substituting $\bm\beta^o$ for $\bm\beta$ in Lemma \ref{lemma:f_triangle} gives
\begin{equation}\label{strongcvx-1}
\begin{split}
& f(\bm\beta^{(t+1)}) - f(\bm\beta^o) + \bm\Delta_\psi(\bm\beta^{(t+1)},\bm\beta^{(t)}) + \bm\Delta_{\breg_\psi(\cdot; \bsb{\beta}^{(t)})}(\bm\beta^o, \bm\beta^{(t+1)}) \\
\le\,&\bm\Delta_\psi(\bm\beta^o,\bm\beta^{(t)}) - \bm\Delta_f(\bm\beta^o,\bm\beta^{(t+1)}).
\end{split}
\end{equation}
By Lemma \ref{lemma:local-opt}, we get
\begin{equation}\label{strongcvx-2}
f(\bm\beta^{(t+1)}) - f(\bm\beta^o) \ge \bm\Delta_f(\bm\beta^{(t+1)},\bm\beta^o).
\end{equation}
Combining \eqref{strongcvx-1} and \eqref{strongcvx-2} yields
\begin{align}
(2\sym{\bm\Delta}_f +  \bm\Delta_{\breg_\psi(\cdot; \bsb{\beta}^{(t)})})(\bm\beta^o, \bm\beta^{(t+1)}) +  \bm\Delta_\psi(\bm\beta^{(t+1)},\bm\beta^{(t)}) \le \bm\Delta_\psi(\bm\beta^o,\bm\beta^{(t)}).\label{scvx-basic-iter}
\end{align}
It follows from   the strong idempotence property  that
\begin{equation} \label{strongcvx-3}
(2\sym{\bm\Delta}_f + \bm\Delta_\psi)(\bm\beta^o, \bm\beta^{(t+1)}) \le \bm\Delta_\psi(\bm\beta^o,\bm\beta^{(t)})-\min _{0\le t \le  T} \breg_{\psi} (\bsb{\beta}^{(t+1)},\bsb{\beta}^{(t)} ),
\end{equation}
for any $0\le t\le T$,
and so \eqref{strcvx-res-1} can be obtained under $2\sym{\bm\Delta}_f \ge \varepsilon \bm\Delta_\psi$.

To show the first result, since $\bm\Delta_\phi = \bm\Delta_\psi + \bm\Delta_f$, \eqref{strongcvx-3} becomes
\begin{align*} 
(2\sym{\bm\Delta}_f +  \bm\Delta_{ \phi } -  \bm\Delta_{ f})(\bm\beta^o, \bm\beta^{(t+1)}) \le  (\bm\Delta_\phi - \bm\Delta_f)(\bm\beta^o,\bm\beta^{(t)}) - \min _{0\le t \le  T} \breg_{\psi} (\bsb{\beta}^{(t+1)},\bsb{\beta}^{(t)}).
\end{align*}
Because $\kappa>1$, \eqref{strcvx-cond-2} implies that $$\bm\Delta_f \ge (\kappa+1)\sym{\bm\Delta}_\phi/\kappa - {\back{\bm\Delta}_\phi}.$$
Applying the inequality twice,  we obtain   $((\kappa+1)/\kappa)\sym{\bm\Delta}_\phi(\bm\beta^o, \bm\beta^{(t+1)}) \le (2-(\kappa+1)/\kappa)\sym{\bm\Delta}_\phi(\bm\beta^o, \allowbreak \bm\beta^{(t)})- \min _{0\le t \le  T} \breg_{\psi} (\bsb{\beta}^{(t+1)},\bsb{\beta}^{(t)})$,
or
\begin{equation} \label{strongcvx-5}
\sym{\bm\Delta}_\phi(\bm\beta^o, \bm\beta^{(t+1)}) \le \frac{\kappa-1}{\kappa+1}\sym{\bm\Delta}_\phi(\bm\beta^o, \bm\beta^{(t)})-\frac{\kappa}{\kappa + 1}   \min _{0\le t \le  T} \breg_{\psi} (\bsb{\beta}^{(t+1)},\bsb{\beta}^{(t)}).
\end{equation}
The final conclusion can be obtained by applying \eqref{strongcvx-5} iteratively for $t = 0, 1, \ldots, T$.

\subsection{Proofs of   Theorem \ref{th:comp_nonconvex} and Corollary \ref{cor-mirror}}
The proof of the theorem follows from Section \ref{proof:comp_funcval}. In fact,  setting $\bm{\beta} = \bm{\beta}^{(t)}$ in \eqref{f_triangle-0} gives
\begin{align*}
 (\back\breg_\psi+  \bm\Delta_{\breg_\psi(\cdot; \bsb{\beta}^{(t)})}+\bm{\Delta}_f)(\bm{\beta}^{(t)},\bm{\beta}^{(t+1)}) \le      f(\bm{\beta}^{(t)})-f(\bm{\beta}^{(t+1)}),
\end{align*}
which, by the weak idempotence property (with $\bsb{\alpha} = \bsb{\beta}^{(t)}$), reduces to
\begin{equation}\label{th:comp_diff_bound-2}
(2\sym{\bm\Delta}_\psi + \bm{\Delta}_f)(\bm{\beta}^{(t)},\bm{\beta}^{(t+1)}) \leq f(\bm{\beta}^{(t)}) - f(\bm{\beta}^{(t+1)}).
\end{equation}
Summing up (\ref{th:comp_diff_bound-2}) over $t=0,1,\ldots, T$ gives the conclusion.\\

 Next, we prove a result slightly more general than Corollary \ref{cor-mirror}. Recall the surrogate
$$
g(\bsb{\beta}; \bsb{\beta}^-) = f(\bsb{\beta}) + (\rho \Breg_{\varphi}  - \breg_f)(\bsb{\beta}, \bsb{\beta}^-)
$$
where $\varphi\in \mathcal C^1$ is a strictly convex function, and $f$ is continuous and directionally differentiable but not necessarily convex or differentiable. Denote $  \arg\min g(\bsb{\beta}; \bsb{\beta}^{-})$ by $\mathcal T(\bsb{\beta}^{-})$.
\begin{manualcor}{\ref{cor-mirror}'}\label{cor-mirror-gen}
Suppose that  $\breg_f \le L \sym\Breg_{\varphi}$ for some $L>0$ and  the inverse stepsize parameter $\rho$ satisfies  $\rho > L/2$. Then $\mathop{\avg}_{0\leq t \leq T} (2\rho\bar{\mathbf D}_\varphi - {\back{\bm\Delta}_f}) \allowbreak(\bm{\beta}^{(t)}, \bm{\beta}^{(t+1)}) \leq \frac{f(\bm{\beta}^{(0)})}{(T+1)}  $ and so  $\mathop{\avg}_{0\leq t \leq T}\bar{\mathbf D}_\varphi (\bm{\beta}^{(t)}, \bm{\beta}^{(t+1)}) \leq \frac{f(\bm{\beta}^{(0)})}{(T+1)(2\rho-L)}  $.

Moreover, for any  accumulation point of $\bsb{\beta}^{(t)}$ at which $\mathcal T$ is continuous, it must be  a fixed point of $\mathcal T$. This is particularly true when   $f\in \mathcal C^1$.
\end{manualcor}

 \begin{proof}
Observe from \eqref{th:comp_diff_bound-2} that
\begin{align*}
f(\bm{\beta}^{(t)}) - f(\bm{\beta}^{(t+1)}) & \ge (2\rho \sym{\Breg}_\varphi - 2 \sym\breg_f + \bm{\Delta}_f)(\bm{\beta}^{(t)},\bm{\beta}^{(t+1)}) \\
  & \ge (2\rho \sym{\Breg}_\varphi -  \bm{\Delta}_f)(\bm{\beta}^{(t+1)}, \bm{\beta}^{(t)}) \\
&\ge  (2\rho-L)\sym{\Breg}_\varphi(\bm{\beta}^{(t+1)}, \bm{\beta}^{(t)})\ge 0.
\end{align*}
The error  bounds can be obtained. 

Let $\bsb{\beta}^o$ be the limit point of some subsequence $\bsb{\beta}^{t_l}$ as $l\rightarrow \infty$. Hence $f(\bsb{\beta}^{(t)})$ converges monotonically to   $\lim_{l\rightarrow \infty} f(\bsb{\beta}^{t_l}) =f(\bsb{\beta}^{o})$. It follows that$$
 \lim_{t\rightarrow +\infty} \sym{\Breg}_\varphi(\bm{\beta}^{(t+1)}, \bm{\beta}^{(t)} )= 0.
 $$
 $\mathcal T  $ is a well-defined function because of the strict convexity of the $g$-optimization problem. From   the   continuity assumptions,
 $$
 0 = \lim_{l\rightarrow +\infty} \sym{\Breg}_\varphi(\bm{\beta}^{(t_l+1)}, \bm{\beta}^{(t_l)} )= \sym{\Breg}_\varphi(\mathcal T(\bm{\beta}^o), \bm{\beta}^{o} )
 $$
 and thus $\mathcal T(\bm{\beta}^o)= \bm{\beta}^{o}$, i.e., $\bsb{\beta}^o$ is a fixed point of $\mathcal T$. 
\end{proof}

\subsection{Proof of Proposition \ref{pro:PTheta}} \label{sec:proof-Lp}

First, we show a  result when using the  Bregman surrogate  $g(\bm\beta;\bm\beta^-) = l(\bm\beta) + P(\varrho\bm\beta) + \bm\Delta_\psi(\bm\beta,\bm\beta^-)$ for solving  $\min_{\bm\beta} f(\bsb{\beta})= l(\bm\beta) + P(\varrho\bm\beta)$ where $l$ and $P$ directionally differentiable and can be  nonconvex. Define
\begin{equation} \label{def:Lp}
\mathcal L_P := \inf\{\mathcal L\in\mathbb R: \bm\Delta_P + \mathcal L\mathbf D_2 \ge 0\},
\end{equation}
which provides an index to characterize the degree of nonconvexity of $P$, c.f.\,\cite{Loh2015}.
Assume $\mathcal L_P > -\infty$. Then for $\bm\beta^{(t+1)}\in\arg\min_{\bm\beta}g(\bm\beta;\bm\beta^{(t)})$, the following inequality holds for all $T \ge 1$
\begin{equation*}
\mathop{\avg}_{0\leq t \leq T}(2\sym{\bm\Delta}_\psi + \bm{\Delta}_{l} - \varrho^2\mathcal L_P\mathbf{D}_2)(\bm{\beta}^{(t)}, \bm{\beta}^{(t+1)}) \leq \frac{1}{T+1}\big[f(\bm{\beta}^{(0)}) - f(\bm{\beta}^{(T+1)})\big].
\end{equation*}
The result can be proved from Theorem \ref{th:comp_nonconvex}, noticing the fact that
$\bm\Delta_f(\bm\beta,\bm\beta^-) =  \bm\Delta_l(\bm\beta,\bm\beta^-) + \bm\Delta_P(\varrho\bm\beta,\varrho\bm\beta^-)  \ge   \bm\Delta_l(\bm\beta,\bm\beta^-) - \mathcal L_P\mathbf D_2(\varrho\bm\beta,\varrho\bm\beta^-)  =   \bm\Delta_l(\bm\beta,\bm\beta^-) - \varrho^2\mathcal L_P\mathbf D_2(\bm\beta,\bm\beta^-)
$  for any $\bm\beta,\bm\beta^-$. The details are omitted.

It suffices to proving the following lemma to complete the proof of Proposition \ref{pro:PTheta}.

\begin{lemma} \label{lemma:Lp}
Given any thresholding function $\Theta$ satisfying Definition \ref{def:Theta}, let $P_\Theta$ be the $\Theta$-induced penalty  in \eqref{pendef}. Then $\mathcal L_\Theta$ as defined in \eqref{def:LTheta} equals $\mathcal L_{P_\Theta}$ that is given in \eqref{def:Lp}.
\end{lemma}

\begin{proof}

Since $\bm\Delta_{P_\Theta}(\bm\beta,\bm\gamma) = \sum_j\bm\Delta_{P_\Theta}(\beta_j,\gamma_j)$, it suffices to show the result in the univariate case.
Recall that $\Theta^{-1}(u;\lambda) := \sup\{t:\Theta(t;\lambda) \le u\}, \forall u>0$.
Since $P_\Theta(\gamma) = P_\Theta(|\gamma|) = \int_0^{|\gamma|}(\Theta^{-1}(u;\lambda)-u)\rd u$, we assume $\gamma \ge 0$ without loss of generality. We define $s(u;\lambda) = \Theta^{-1}(u;\lambda)-u$ for $u \ge 0$, and extend $s(\cdot)$ to $(-\infty, 0)$ by $s(-u) = -s(u), u > 0$. Clearly, $s'(u) = s'(|u|)$ a.e., and so $-\mathcal L_\Theta = \mathrm{ess\,inf}\{s'(u;\lambda):u\ne 0\}$.
By definition,
\begin{equation} \nonumber
\delta  P_\Theta(\gamma; \beta - \gamma) = \begin{cases} s(\gamma)(\beta - \gamma), &\text{ if } \gamma \ge 0, \\ s(0)|\beta|, &\text{ if } \gamma=0.\end{cases}
\end{equation}

When $\beta\ge 0$ and $\gamma \ne 0$, we get
\begin{equation} \nonumber
\begin{split}
&\,(\bm\Delta_{P_\Theta} + \mathcal L \mathbf D_2)(\beta,\gamma) \\
= &\,P_\Theta(\beta) - P_\Theta(\gamma) - \delta P_\Theta(\gamma; \beta - \gamma) + \mathcal L\mathbf D_2(\beta,\gamma)\\
= &\int_{\gamma}^{\beta}s(u)\rd u - s(\gamma)(\beta - \gamma) + \frac{1}{2}\mathcal L(\beta-\gamma)^2\\
= &\int_{\gamma}^{\beta}s(u)\rd u - \int_{\gamma}^{\beta} s(\gamma)\rd u + \mathcal L\int_{\gamma}^{\beta}(u-\gamma)\rd u\\
= &\int_{\gamma}^{\beta}\big[s(u) - s(\gamma) + \mathcal L(u-\gamma) \big]\rd u\\
= &\int_{\gamma}^{\beta}\int_{\gamma}^u  \big[s'(v) +\mathcal L\big] \rd v\rd u.
\end{split}
\end{equation}
When $\beta < 0$ and $\gamma \ne 0$,
\begin{equation} \nonumber
\begin{split}
&\,(\bm\Delta_{P_\Theta} + \mathcal L \mathbf D_2)(\beta,\gamma) \\
= &\,P_\Theta(\beta) - P_\Theta(\gamma) - \delta P_\Theta(\gamma; \beta - \gamma) + \mathcal L\mathbf D_2(\beta,\gamma)\\
= &\int_{\gamma}^{-\beta}s(u)\rd u - s(\gamma)(\beta - \gamma) + \frac{1}{2}\mathcal L(\beta-\gamma)^2\\
= &\int_{-\gamma}^{-\beta}s(u)\rd u - \int_{-\gamma}^{-\beta} s(-\gamma)\rd u + \mathcal L\int_{-\gamma}^{-\beta}(u+\gamma)\rd u\\
= &\int_{-\gamma}^{-\beta}\big[s(u) - s(-\gamma) + \mathcal L(u+\gamma) \big]\rd u\\
= &\int_{-\gamma}^{-\beta}\int_{-\gamma}^u  \big[s'(v) +\mathcal L\big] \rd v\rd u.
\end{split}
\end{equation}
Similarly, when $\gamma = 0$, $(\bm\Delta_{P_\Theta} + \mathcal L \mathbf D_2)(\beta,0) = \int_0^{|\beta|}\int_0^u[s(v)+\mathcal L]\rd v\rd u$.
It is then easy to verify that $\mathcal L_\Theta = \mathcal L_{P_\Theta}$.
\end{proof}

\subsection{Proof of Proposition \ref{pro:LLA_comp}}

Let $f(\bm\beta) = l(\bm\beta) + P(\varrho\bm\beta)$ and recall
\begin{equation} \nonumber
\begin{split}
g^{(t)}_\mathrm{LLA}(\bm\beta;\bm\beta^{(t)}) &= f(\bm\beta) + \bm\Delta_\mathrm{LLA}^{(t)}(\varrho\bm\beta,\varrho\bm\beta^{(t)})\\
&= f(\bm\beta) + \sum_j(\alpha_j^{(t)}\bm\Delta_1 - \bm\Delta_P)(\varrho\beta_j,\varrho\beta_j^{(t)}).
\end{split}
\end{equation}
The proof is similar to that of Theorem \ref{th:comp_nonconvex} and we give some details for completeness.
The important fact  $\beta^{(t+1)} \in \mathop{\arg\min}_{\bm\beta}g^{(t)}_\mathrm{LLA}(\bm\beta;\bm\beta^{(t)})$ as shown in Example \ref{ex:LLA} implies
$
  \bm\Delta_{g^{(t)}_\mathrm{LLA}(\cdot;\bm\beta^{(t)})}(\bm\beta^{(t)},\bm\beta^{(t+1)})
\le  f(\bm\beta^{(t)}) - f(\bm\beta^{(t+1)}) + \bm\Delta_\mathrm{LLA}^{(t)}(\varrho\bm\beta^{(t)},\varrho\bm\beta^{(t)}) - \bm\Delta_\mathrm{LLA}^{(t)}(\varrho\bm\beta^{(t+1)},\varrho\bm\beta^{(t)})
$ or \begin{equation*} 
\begin{split}
&f(\bm\beta^{(t)}) - f(\bm\beta^{(t+1)})\\
\ge\,&\bm\Delta_f(\bm\beta^{(t)},\bm\beta^{(t+1)}) + \bm\Delta_{\bm\Delta_\mathrm{LLA}^{(t)}(\varrho\cdot,\varrho\bm\beta^{(t)})}(\bm\beta^{(t)},\bm\beta^{(t+1)}) +  \bm\Delta_\mathrm{LLA}^{(t)}(\varrho\bm\beta^{(t+1)},\varrho\bm\beta^{(t)})\\
=\,&\bm\Delta_f(\bm\beta^{(t)},\bm\beta^{(t+1)}) + \bm\Delta_{\bm\Delta_\mathrm{LLA}^{(t)}(\cdot,\varrho\bm\beta^{(t)})}(\varrho\bm\beta^{(t)},\varrho\bm\beta^{(t+1)}) +  \bm\Delta_\mathrm{LLA}^{(t)}(\varrho\bm\beta^{(t+1)},\varrho\bm\beta^{(t)})\\
=\,&\bm\Delta_f(\bm\beta^{(t)},\bm\beta^{(t+1)}) +\breg_{\mathrm{LLA}}^{(t)}(\varrho\bm\beta^{(t)},\varrho\bm\beta^{(t+1)})+  \bm\Delta_\mathrm{LLA}^{(t)}(\varrho\bm\beta^{(t+1)},\varrho\bm\beta^{(t)}).\\
   = \, & \bm\Delta_f(\bm\beta^{(t)},\bm\beta^{(t+1)}) + 2\sym{\bm\Delta}_{\mathrm{LLA}}^{(t)}(\varrho\bm\beta^{(t)},\varrho\bm\beta^{(t+1)}).
\end{split}
\end{equation*}
The conclusion follows from summing up this inequality for $t = 0,1,\ldots,T$.

\subsection{Proofs of   Theorem \ref{thm:errrate} and Theorem \ref{thm:ora}} \label{proof:ora}

Let $f(\bm\beta) = l(\bm\beta) + P_{\Theta}(\varrho\bm\beta;\lambda)$ and recall $g(\bm\beta;\bm\beta^-) = f(\bm\beta) + \bm\Delta_\psi(\bm\beta,\bm\beta^-)$. We first introduce a lemma.

\begin{lemma}\label{lemma:Theta_triangle_fixed}
Let $\hat{\bm\beta} \in \mathcal F$. Then for any $\bm\beta\in\mathbb R^p$, we have the following inequality regardless of the specific form of $\psi$
\begin{equation}\label{fix_tri_ineq}
\begin{split}
&(\bm\Delta_l - \varrho^2\mathcal L_{\Theta}\mathbf D_2)(\bm\beta, \hat{\bm\beta}) + \bm\Delta_l(\hat{\bm\beta}, \bm\beta^*) + P_\Theta(\varrho\hat{\bm\beta};\lambda)\\
\le\,&\breg_l(\bm\beta,\bm\beta^*) + \langle\bm\epsilon, \bm X \hat{\bm\beta} - \bm X\bm\beta \rangle + P_\Theta(\varrho\bm\beta;\lambda).
\end{split}
\end{equation}
\end{lemma}

\begin{proof}
Denote $\hat g(\bm\beta) := g(\bm\beta;\hat{\bm\beta}) = l(\bm\beta) + P_\Theta(\varrho\bm\beta;\lambda) + \bm\Delta_\psi(\bm\beta,\hat{\bm\beta})$.
Since $\hat{\bm\beta}$ is a minimizer of $\hat g(\cdot)$, Lemma \ref{lemma:local-opt} shows that for any $\bsb{\beta}$, $\bm\Delta_{\hat g}(\bm\beta,\hat{\bm\beta})   \le \hat g(\bm\beta) - \hat g(\hat{\bm\beta})$. On the one hand,
\begin{align*}
&\,\hat g(\bm\beta) - \hat g(\hat{\bm\beta})\\
=&\,l(\bm\beta) - l(\hat{\bm\beta}) + P_\Theta(\varrho\bm\beta;\lambda) - P_\Theta(\varrho\hat{\bm\beta};\lambda) + \bm\Delta_\psi(\bm\beta,\hat{\bm\beta})\\
=&\,l(\bm\beta) - l(\bm\beta^*) - (l(\hat{\bm\beta}) - l(\bm\beta^*)) + P_\Theta(\varrho\bm\beta;\lambda) - P_\Theta(\varrho\hat{\bm\beta};\lambda) + \bm\Delta_\psi(\bm\beta,\hat{\bm\beta})\\
=&\,\bm\Delta_l(\bm\beta,\bm\beta^*) + \langle\nabla l(\bm\beta^*),\bm\beta-\bm\beta^*\rangle - (\bm\Delta_l(\hat{\bm\beta},\bm\beta^*) + \langle\nabla l(\bm\beta^*),\hat{\bm\beta}-\bm\beta^*\rangle)\\
 &\,+ P_\Theta(\varrho\bm\beta;\lambda) - P_\Theta(\varrho\hat{\bm\beta};\lambda) + \bm\Delta_\psi(\bm\beta,\hat{\bm\beta})\\
=&\,\bm\Delta_l(\bm\beta,\bm\beta^*) - \bm\Delta_l(\hat{\bm\beta},\bm\beta^*) + \langle \bm\epsilon, \bm X\hat{\bm\beta} - \bm X\bm\beta\rangle+ P_\Theta(\varrho\bm\beta;\lambda) - P_\Theta(\varrho\hat{\bm\beta};\lambda) + \bm\Delta_\psi(\bm\beta,\hat{\bm\beta}).
\end{align*}
On the other hand, by Lemma \ref{lemma:Delta_linear}, Lemma \ref{lemma:Delta_Delta}, and Lemma \ref{lemma:Lp},
\begin{equation} \nonumber
\begin{split}
\bm\Delta_{\hat g}(\bm\beta,\hat{\bm\beta}) =&\,\bm\Delta_l(\bm\beta,\hat{\bm\beta}) + \bm\Delta_{P_\Theta(\varrho\cdot)}(\bm\beta,\hat{\bm\beta}) +  \bm\Delta_{\bm\Delta_\psi(\cdot,\hat {\bm\beta})} (\bm\beta,\hat{\bm\beta})\\
=&\, \bm\Delta_l(\bm\beta,\hat{\bm\beta}) + \bm\Delta_{P_\Theta(\varrho\cdot)}(\bm\beta,\hat{\bm\beta})  + \bm\Delta_\psi(\bm\beta,\hat{\bm\beta})
\\\ge&\, \bm\Delta_l(\bm\beta,\hat{\bm\beta}) - \varrho  ^{2}\mathcal L_\Theta\mathbf D_2(\bm\beta,\hat{\bm\beta}) + \bm\Delta_\psi(\bm\beta,\hat{\bm\beta}).
\end{split}
\end{equation}
The conclusion follows.
\end{proof}

To handle the stochastic term $\langle\bm\epsilon, \bm X \hat{\bm\beta} - \bm X\bm\beta \rangle$ in \eqref{fix_tri_ineq}, we introduce the following result. %

\begin{lemma} \label{lemma:phostochastic}
Let $\bm{X}\in\mathbb{R}^{n\times p}$,   $\bsb{\epsilon}$ be a sub-Gaussian random vector with mean 0 and scale bounded by $\sigma$, and $\lambda^o=\sigma\sqrt{\log (ep)}$. Suppose that  $\varrho\ge \| \bsb{X}\|_2$. Then  there exist universal constants $A_0, C, c>0$ such that for any $a\ge 2b>0$ and $A_1 \ge A_0$, the following event
$$
\sup_{\bm\beta \in \mathbb R^p} \Big\{2\langle  \bsb{\epsilon},    \bm X\bm \beta \rangle - \frac{1}{a} \|\bm X \bm\beta  \|_2^2 - \frac{1}{b} [  P_{H}(\varrho\bm\beta; \sqrt{ab}{A_1\lambda^o}) ]\Big\} \geq a  \sigma^2 t
$$ 
occurs with probability at most $C\exp(-ct)p^{-cA_1^2}$.
\end{lemma}
The lemma can be proved by Lemma 4 of   \cite{She2016} based on a scaling argument.

Let
$
R = \sup_{\bm\beta, \hat{\bm\beta}\in \mathbb R^p} \{\langle \bm\epsilon, \bm X \hat{\bm\beta} - \bm X\bm\beta \rangle - \frac{1}{2a} \|\bm X \hat{\bm\beta} - \bm X\bm\beta  \|_2^2 - \frac{1}{2b} [ P_H(\varrho(\hat{\bm\beta} - \bm\beta); \sqrt{ab}{A\lambda^o}) ]\}
$
with $\lambda^o = \sigma\sqrt{\log(ep)}$.
Plugging the bound into \eqref{fix_tri_ineq} gives
\begin{equation} \label{intermediateRes}
\begin{split}
&(\bm\Delta_l - \varrho^2 \mathcal L_{\Theta}\mathbf D_2)(\bm\beta, \hat{\bm\beta}) + \bm\Delta_l(\hat{\bm\beta}, \bm\beta^*) + P_\Theta(\varrho\hat{\bm\beta};\lambda) - P_\Theta(\varrho\bm\beta;\lambda)\\
\le\,&\bm\Delta_l(\bm\beta,\bm\beta^*) +  \frac{1}{2a} \|\bm X \hat{\bm\beta} - \bm X\bm\beta  \|_2^2 + \frac{1}{2b} [ P_H(\varrho(\hat{\bm\beta} - \bm\beta); \sqrt{ab}{A\lambda^o}) ] + R
\end{split}
\end{equation}
with $\EP(2R\ge a\sigma^2 t)\le C\exp(-ct)p^{-cA ^2}$ for any $a\ge 2 b >0$ and $A$  large.

To prove Theorem \ref{thm:errrate}, substitute $\bsb{\beta}^*$ for $\bsb{\beta}$ in \eqref{intermediateRes} and combine it with the regularity condition \eqref{R_0-errrate}, resulting in
\begin{align*}
\big(  \delta - \frac{1}{a} \big)\mathbf D_2(\bm X \hat{\bm\beta},\bm X {\bm{\beta}}^*) + \big(\vartheta - \frac{1}{2b}\big)P_H(\varrho (\hat {\bsb{\beta}} - {\bsb{\beta}}^*; \lambda)  \le      K\lambda^2J(\bm\beta^*) + R
\end{align*}
where $\lambda = \sqrt{ab}A\lambda^o$, $a\ge 2b>0$, and $A \ge A_0$ with $A_0$ given in Lemma \ref{lemma:phostochastic}. Setting $a = 2/(\delta\wedge(2\vartheta)), b = 1/(2\vartheta)$ or   $a = 2/((2\delta)\wedge\vartheta), b = 1/\vartheta$  bounds    $\Breg_2(\bm{X} \hat {\bsb{\beta}},\bm{X}\bsb{\beta}^*)$ or $P_H(\varrho  (\hat {\bsb{\beta}} - \bsb{\beta}^*); \lambda)$. Finally, by Lemma \ref{lemma:phostochastic}, $\mathbb P(R \le 0) \ge 1 - Cp^{-cA^2}$.

For Theorem \ref{thm:ora}, we combine  \eqref{intermediateRes} and   \eqref{R_0-oracond-general} with $\bsb{\gamma}=\hat {\bsb{\beta}}$:
\begin{align*}
\bm\Delta_l(\hat{\bm\beta}, \bm\beta^*) + \big(\delta - \frac{1}{a}   \big)\mathbf D_2(\bm X \bm\beta,\bm X\hat{\bm{\beta}})\le\,&   \frac{\alpha}{L}r\bm{\Delta}_{l}(\bm{\beta}, \hat {\bm{\beta}})
 + \bm\Delta_l(\bm\beta, \bm\beta^*)   + K\lambda^2J(\bm\beta) + R.
\end{align*}
Take the same choice for   $\lambda $ and  set $a = 2/(\delta\wedge(2\vartheta)), b = 1/(2\vartheta)$.

Case (i):       $\alpha r/L=0$. The conclusion follows easily, and does not need any restriction on $\alpha$ or $L$.

Case (ii):  $\alpha r/L>0$. Then $0<r<1$ and $\alpha/L>0$. If $\alpha < 0$, $ ({\alpha}r/{L})     \bm{\Delta}_{l_0}(\bsb{X}\bm{\beta},\bsb{X} \hat {\bm{\beta}}) \le \alpha  r d^2(\bsb{X}\bm{\beta},\bsb{X} \hat {\bm{\beta}}) \le 0 $,   reducing to the first case. Assume  $\alpha>0$. Then
\begin{align*}
  \frac{\alpha}{L}r\bm{\Delta}_{l}(\bm{\beta}, \hat {\bm{\beta}}) & =   \frac{\alpha}{L}r\bm{\Delta}_{l_0}(\bm{\bsb{X}\beta},\bsb{X} \hat {\bm{\beta}}) \le     \alpha  r d^2(\bsb{X}\bm{\beta},\bsb{X} \hat {\bm{\beta}}) \\
   & \le  \alpha  r (d(\bsb{X}\bm{\beta}, \bsb{\bsb{X}\beta}^*) + d(\bsb{X}\hat {\bm{\beta}}, \bsb{\bsb{X}\beta}^*))^2 \\
& \le \alpha r (1 +1/M)d^2(\bsb{X}\bm{\beta}, \bsb{\bsb{X}\beta}^*) + \alpha r (1 +M)d^2(\bsb{X}\hat{\bm{\beta}},\bsb{X} \bsb{\beta}^*) \\
& \le   r (1 +1/M)\bm{\Delta}_{l}(\bm{\beta}, \bsb{\beta}^*) +   r (1 +M)\bm{\Delta}_{l}(\hat{\bm{\beta}}, \bsb{\beta}^*),
\end{align*}
for any $M>0$. Take $M= (1- r)/(1+r)$. Then $$1/ \{1 - r (1 +M)\} =  1+ r (1 +1/M) = (1+r)/(1-r).$$
So we obtain
\begin{align*}
&(\frac{1-r}{1+r} \breg_{l_0} + \frac{\delta}{2}\Breg_2)(\bsb{X}\hat{\bm\beta},\bsb{X}\bm\beta^*) \\ \le \ & \frac{1+r}{ 1-r}\bm{\Delta}_{l_0}(\bsb{X}\bm{\beta}, \bsb{X}\bsb{\beta}^*) +  \frac{KA^2}{((2\vartheta)\wedge\delta )\vartheta}\sigma^2J(\bm\beta)\log(ep) +R.
\end{align*}
Finally, from  $\EP(2R\ge a\sigma^2 t)\le C\exp(-ct)p^{-cA ^2}\le C\exp(-ct)$, we have $\EE R \le Ca \sigma^2$. The oracle inequality is proved.   In fact,  we also get
\begin{align*}
\EE[\Breg_2(\bsb{X}\hat{\bm\beta},\bsb{X}\bm\beta^*)] \le    \frac{2(1+r)}{ (1-r)\delta}\EE [\bm{\Delta}_{l_0}(\bsb{X}\bm{\beta}, \bsb{X}\bsb{\beta}^*) ] \\ +  \frac{2 KA^2}{(\vartheta\wedge\delta )\vartheta \delta}\sigma^2J(\bm\beta)\log(ep) + \frac{C}{( \vartheta\wedge\delta)\delta } \sigma^2
\end{align*}
 under the same condition.


\begin{remark} \label{staterr-exts}
 Recall $J^*=J(\bsb{\beta}^*), \mathcal J^*=\mathcal J(\bsb{\beta}^*)$ and       $P( \bm\beta;\lambda)=\sum_j P( \beta_j; \lambda) $. When $\hat {\bsb{\beta}}$ is a global minimizer, applying the bound of the stochastic term proved in Lemma \ref{lemma:phostochastic} gives the same conclusions \eqref{esterrbnd-1}, \eqref{esterrbnd-2}, under
\begin{align}
\begin{split}
&  \vartheta P_H(\varrho(\bm{\beta} - \bm{\beta}^*);\lambda)  +  P_\Theta(\varrho\bm{\beta}^*;\lambda) \\
\le\,& (2\sym{\bm{\Delta}}_{l_0} - \delta\mathbf D_2)(\bm{X\beta},\bm{X\beta}^*)  + P_\Theta(\varrho\bm{\beta};\lambda) + K\lambda^2J^*,  \forall \bsb{\beta} \end{split}\label{R_0-errrate-cond-global}
\end{align}
for some  $\delta>0$, $\vartheta>0$ and large enough $K\geq 0$. Assuming  $P_{\Theta}$  is subadditive, we can follow the arguments in  Remark \ref{rmk:regconds} to show that    \eqref{R_0-errrate-cond-global} is implied by $(1+\vartheta)  P_\Theta(\varrho(\bm{\beta}-\bsb{\beta}^*)_{\mathcal J^*};\lambda) \le   (2\sym{\bm{\Delta}}_{l_0} - \delta\mathbf D_2)(\bm{X\beta},\bm{X\beta}^*)  + K\lambda^2J^*+(1-\vartheta) P_\Theta(\varrho(\bm{\beta}-\bsb{\beta}^*)_{\mathcal J^{* c}};\lambda)
$. Furthermore, when $l_0$ is $\mu$-strongly convex as in regression, one can take $\delta = \mu$ and  the regularity condition is implied by
\begin{align}
(1+\bar \vartheta)  P_\Theta(\varrho \bm{\gamma}_{\mathcal J^*};\lambda) \le    \bar K\lambda \sqrt {J^*} \| \bsb{X} \bsb{\gamma}\|_2 +  P_\Theta(\varrho\bm{\gamma}_{\mathcal J^{* c}};\lambda), \forall \bsb{\gamma} \label{rem1-4-gen}
\end{align}
for some $\bar\vartheta \,(=\!\frac{2\vartheta}{1-\vartheta})>0$ and $\bar K  \big(\kern-4pt=\! \frac{\sqrt  {2 (2 \mu-\delta)K}}{1 - \vartheta}\big)\ge 0$, or the constrained forms
\begin{align}
[ P_{\Theta} (\varrho \bsb{\gamma}_{\mathcal J^*}; \lambda)]^2 \le \tilde K \lambda^2 J^* \| \bsb{X} \bsb{\gamma}\|_2^2,\ \, \forall \bsb{\gamma}: P_{\Theta}(\varrho \bsb{\gamma}_{\mathcal J^{*c}}; \lambda)\le (1+\bar \vartheta) P_{\Theta}(\varrho \bsb{\gamma}_{\mathcal J^{*}}; \lambda) \label{Compacond-genp}
\\\sum _{j\in \mathcal J^*} P_{\Theta}^2 (\varrho \gamma_j; \lambda) \le \tilde K \lambda^2 \| \bsb{X} \bsb{\gamma}\|_2^2,\ \, \forall \bsb{\gamma}: P_{\Theta}(\varrho \bsb{\gamma}_{\mathcal J^{*c}}; \lambda)\le (1+\bar \vartheta) P_{\Theta}(\varrho \bsb{\gamma}_{\mathcal J^{*}}; \lambda) \label{REcond-genp}
\end{align}
for some $\bar\vartheta>0$ and $\tilde K\ge 0$. The conclusions and conditions can also be formulated in the oracle inequality setup of Theorem \ref{thm:ora}. \eqref{rem1-4-gen}, \eqref{Compacond-genp} and \eqref{REcond-genp}
  extend the comparison condition \eqref{regcondition-comp},  compatibility condition and  RE condition to a more general penalty.

\end{remark}

\subsection{Proof of Theorem \ref{th:stat_bound}}
\label{subsec:proofseqbound}
We prove the result under a more relaxed assumption: $l$, $g$ are merely directionally differentiable, and \eqref{S} is replaced by
{
\begin{align*}
\begin{split}
& ( \varrho^2\mathcal L_{\Theta}\mathbf{D}_2+\varepsilon\bm\Delta_\psi)( \bm{\beta}^*, \bm{\beta}) +(\breg_\psi - \breg_{\breg_\psi(\cdot; \bsb{\alpha})})(\bm{\beta}^*,\bm{\beta}) + \vartheta P_H(\varrho(\bm{\beta} - \bm\beta^*);\lambda) + P_{\Theta}(\varrho\bm\beta^*;\lambda) \\
\le\,& (2\sym{\bm{\Delta}}_{l_0} -\delta\mathbf D_2)(\bm{X\beta},\bm{X\beta}^*) + P_{\Theta}(\varrho\bm{\beta};\lambda) + K\lambda^2J(\bm{\beta}^*), \forall \bsb{\beta}, \bsb{\alpha}
\end{split}
\end{align*}}
for some $\delta > 0$, $\varepsilon > 0$, $\vartheta > 0$ and large $K\geq 0$.

 Recall the objective function $f(\bm\beta) = l(\bm\beta) + P_\Theta(\varrho\bm\beta;\lambda)$ and the surrogate function $g(\bm\beta;\bm\beta^-)=f(\bm\beta) + \bm\Delta_\psi(\bm\beta,\bm{\beta}^-)$. From Lemma \ref{lemma:f_triangle} and   Lemma \ref{lemma:Lp}, we obtain
\begin{align*}
&f(\bm{\beta}) + \bm\Delta_\psi(\bm{\beta},\bm{\beta}^{(t)}) \\
\ge\,& f(\bm{\beta}^{(t+1)}) + \bm\Delta_\psi(\bm{\beta}^{(t+1)},\bm{\beta}^{(t)}) + (  \bm\Delta_{\breg_\psi(\cdot; \bsb{\beta}^{(t)})}+\bm{\Delta}_l + \bm{\Delta}_{P_\Theta(\varrho\cdot)})(\bm{\beta},\bm{\beta}^{(t+1)})\\
\ge\,& f(\bm{\beta}^{(t+1)}) + \bm\Delta_\psi(\bm{\beta}^{(t+1)},\bm{\beta}^{(t)}) + (  \bm\Delta_{\breg_\psi(\cdot; \bsb{\beta}^{(t)})}+\bm{\Delta}_l - \varrho^2\mathcal L_\Theta\mathbf{D}_2)(\bm{\beta},\bm{\beta}^{(t+1)}).
\end{align*}
Substituting $\bm{\beta}^*$ for $\bm{\beta}$   yields
\begin{equation} \nonumber
\begin{split}
&\bm\Delta_\psi(\bm{\beta}^*,\bm{\beta}^{(t)})\\
\ge\,&\bm\Delta_\psi(\bm{\beta}^{(t+1)},\bm{\beta}^{(t)}) + (  \bm\Delta_{\breg_\psi(\cdot; \bsb{\beta}^{(t)})}+\bm{\Delta}_l)(\bm{\beta}^*,\bm{\beta}^{(t+1)}) - \mathcal L_{\Theta}\mathbf D_2(\varrho\bm{\beta}^*,\varrho\bm{\beta}^{(t+1)})\\
&+ l(\bm\beta^{(t+1)}) - l(\bm\beta^*) + P_{\Theta}(\varrho\bm{\beta}^{(t+1)};\lambda) - P_{\Theta}(\varrho\bm{\beta}^*;\lambda)\\
=\,&\bm\Delta_\psi(\bm{\beta}^{(t+1)},\bm{\beta}^{(t)}) + \bm\Delta_{\breg_\psi(\cdot; \bsb{\beta}^{(t)})}(\bm{\beta}^*,\bm{\beta}^{(t+1)}) - \mathcal L_{\Theta}\mathbf D_2(\varrho\bm{\beta}^*,\varrho\bm{\beta}^{(t+1)})\\
& 2\sym{\bm\Delta}_l(\bm\beta^*,\bm\beta^{(t+1)})- \langle\bm{\epsilon}, \bm{X\beta}^{(t+1)} - \bm{X\beta}^*\rangle + P_{\Theta}(\varrho\bm{\beta}^{(t+1)};\lambda) - P_{\Theta}(\varrho\bm{\beta}^*;\lambda).
\end{split}
\end{equation}
From the above regularity condition,
\begin{equation}\nonumber
\begin{split}
&\varepsilon\bm\Delta_\psi(\bm{\beta}^*,\bm{\beta}^{(t+1)}) + \mathcal L_{\Theta}\mathbf{D}_2(\varrho\bm{\beta}^*,\varrho\bm{\beta}^{(t+1)}) + \delta\mathbf D_2(\bm{X\beta}^*,\bm{X\beta}^{(t+1)})\\
& + \vartheta P_H(\varrho(\bm{\beta}^{(t+1)} - \bm{\beta}^*);\lambda) + P_{\Theta}(\varrho\bm{\beta^*};\lambda)+(\bm \Delta _\psi- \bm\Delta_{\breg_\psi(\cdot; \bsb{\beta}^{(t)})})(\bm{\beta}^*,\bm{\beta}^{(t+1)})\\
\le\,& 2\sym{\bm{\Delta}}_{l}(\bm{\beta}^*,\bm{\beta}^{(t+1)}) + P_{\Theta}(\varrho\bm{\beta}^{(t+1)};\lambda) + K\lambda^2J(\bm{\beta}^*),
\end{split}
\end{equation}
and so we obtain
\begin{equation} \label{stat-triangle-3}
\begin{split}
&(\varepsilon+1)\bm\Delta_\psi(\bm{\beta}^*,\bm{\beta}^{(t+1)}) + \bm\Delta_\psi(\bm{\beta}^{(t+1)},\bm{\beta}^{(t)})\\
& + \delta\mathbf D_2(\bm{X\beta}^*,\bm{X\beta}^{(t+1)}) + \vartheta P_H(\varrho(\bm{\beta}^{(t+1)}-\bm{\beta}^*);\lambda)\\
\le\,& \bm\Delta_\psi(\bm{\beta}^*,\bm{\beta}^{(t)}) + \langle \bm{\epsilon},\bm{X\beta}^{(t+1)}-\bm{X\beta}^*\rangle + K\lambda^2J^*.
\end{split}
\end{equation}

According to  Lemma \ref{lemma:phostochastic}, as long as $A\ge A_0/\sqrt{2}$, choosing $\lambda = \sqrt{2ab}A\lambda^o$, $b = 1/(2\vartheta), a = 1/(\delta\wedge\vartheta)$ guarantees that the probability of the following inequality occurring for all $t$,
\begin{align*}  \langle \bm{\epsilon},\bm{X\beta}^{(t+1)}-\bm{X\beta}^*\rangle - \delta\mathbf D_2(\bm{X\beta}^*,\bm{X\beta}^{(t+1)}) - \vartheta P_H(\varrho(\bm{\beta}^{(t+1)}-\bm{\beta}^*);\lambda) \ge 0,\end{align*}
is no greater than $Cp^{-cA^2}$. Together with \eqref{stat-triangle-3}, with probability $1 - Cp^{-cA^2}$
\begin{equation} \label{stat-triangle-4}
\bm\Delta_\psi(\bm{\beta}^*,\bm{\beta}^{(t+1)}) \le \frac{1}{\varepsilon+1} [\bm\Delta_\psi(\bm{\beta}^*,\bm{\beta}^{(t)}) - \bm\Delta(\bm\beta^{(t+1)},\bm\beta^{(t)}) + K\lambda^2J^*]
\end{equation}
for all $t$.
The desired inequality can be shown by iteratively applying \eqref{stat-triangle-4} for $t= 0,1,2, \ldots$.

\subsection{Proofs of Theorem \ref{th:comp_acc2} and Corollary \ref{cor-AccIQ}}
\label{subsec:2ndaccANDcor}

First we prove Theorem \ref{th:comp_acc2}. Note that \eqref{acc2-alg2}, \eqref{acc2_search2} have additional terms involving $\mu_0$.
The  first result  for $\mu_0=0$  can be shown based on a GBF translation of the proof of Proposition 1 in \cite{Tseng2008}.
For convenience, let $h_{t}(\bm{\beta}) \allowbreak= f(\bm{\beta}) - \bm\Delta_{\bar\psi_0}(\bm{\beta},\bm{\gamma}^{(t)})= f(\bm{\beta}) - \bm\Delta_{ \psi_0}(\bm{\beta},\bm{\gamma}^{(t)})+\mu_0\breg_{\phi}(\bm{\beta},\bm{\gamma}^{(t)})$. Applying Lemma \ref{lemma:f_triangle} to \eqref{acc2-alg2} yields     $(\bm{\Delta}_f - \bm\Delta_{\bm\Delta_{\bar\psi_0}(\cdot,\bm\gamma^{(t)})} + \theta_t\rho_t\bm\Delta_{\bm\Delta_\phi(\cdot,\bm\alpha^{(t)})})(\bm{\beta}, \bm{\alpha}^{(t+1)})
\le h_{t}(\bm{\beta}) + \theta_t\rho_t\bm\Delta_\phi(\bm{\beta},\bm{\alpha}^{(t)}) - h_{t}(\bm{\alpha}^{(t+1)}) - \theta_t\rho_t\bm\Delta_\phi(\bm{\alpha}^{(t+1)},\bm{\alpha}^{(t)})$   $\forall \bm\beta$, or
\begin{equation} \label{tri_acc2}
\begin{split}
&h_{t}(\bm{\alpha}^{(t+1)}) -h_{t}(\bm{\beta}) + \theta_t\rho_t\bm\Delta_\phi(\bm{\alpha}^{(t+1)},\bm{\alpha}^{(t)})\\
\le\,& \theta_t\rho_t\bm\Delta_\phi(\bm{\beta},\bm{\alpha}^{(t)}) -  (\theta_t\rho_t\bm\Delta_{\bm\Delta_\phi(\cdot,\bm\alpha^{(t)})} + \bm{\Delta}_{f(\cdot) - \bm\Delta_{\bar\psi_0}(\cdot,\bm\gamma^{(t)})})(\bm{\beta}, \bm{\alpha}^{(t+1)}).
\end{split}
\end{equation}
By   definition, 
$\mathbf C_{h_t}(\bm{\alpha}^{(t+1)},\bm{\beta}^{(t)},\theta_t)= \theta_t h_{t}(\bm{\alpha}^{(t+1)}) + (1-\theta_t)h_{t}(\bm{\beta}^{(t)}) - h_{t}(\bm{\beta}^{(t+1)})
$; 
   adding it to   \eqref{tri_acc2} multiplied by $\theta_t$   gives\begin{equation}\nonumber
\begin{split}
&h_{t}(\bm{\beta}^{(t+1)}) - (1-\theta_t)h_{t}(\bm{\beta}^{(t)}) - \theta_t h_{t}(\bm\beta) \\& + \theta_t^2\rho_t\bm\Delta_\phi(\bm{\alpha}^{(t+1)},\bm{\alpha}^{(t)})  + \mathbf C_{h_t}(\bm{\alpha}^{(t+1)},\bm{\beta}^{(t)},\theta_t) \\&+ \theta_t^2\rho_t\bm\Delta_{\bm\Delta_\phi(\cdot,\bm\alpha^{(t)})-\phi(\cdot)}(\bm\beta,\bm\alpha^{(t+1)})\\
\le\,&\theta_t^2\rho_t\bm\Delta_\phi(\bm{\beta},\bm{\alpha}^{(t)}) - (\theta_t^2\rho_t\bm\Delta_\phi + \theta_t\bm{\Delta}_{f(\cdot)-\bm\Delta_{\bar\psi_0}(\cdot,\bm\gamma^{(t)})})(\bm{\beta},\bm{\alpha}^{(t+1)}),
\end{split}
\end{equation}
and so
\begin{equation}\label{befrec1_acc2}
\begin{split}
&f(\bm{\beta}^{(t+1)})-f(\bm{\beta}) - (1-\theta_t)[f(\bm{\beta}^{(t)})-f(\bm{\beta})]  +\theta_t\bm\Delta_{\bar\psi_0}(\bm{\beta},\bm{\gamma}^{(t)})  \\
&  + \theta_t \{ (\bm\Delta_{f(\cdot)-\bm\Delta_{\bar\psi_0}(\cdot,\bm\gamma^{(t)})}+\theta_t\rho_t\bm\Delta_{\bm\Delta_\phi(\cdot,\bm\alpha^{(t)})-\phi(\cdot)})(\bm\beta,\bm\alpha^{(t+1)})\}+ R_t
\\
\le\,&\theta_t^2\rho_t(\bm\Delta_\phi(\bm{\beta},\bm{\alpha}^{(t)}) - \bm\Delta_\phi(\bm{\beta},\bm{\alpha}^{(t+1)})), \forall t\ge 0
\end{split}
\end{equation}
where  $R_t$ is given by $ \theta_t^2\rho_t\bm\Delta_\phi(\bm{\alpha}^{(t+1)},\bm{\alpha}^{(t)}) - \bm\Delta_{\bar\psi_0}(\bm{\beta}^{(t+1)},\bm{\gamma}^{(t)}) + (1-\theta_t)\bm\Delta_{\bar\psi_0}(\bm{\beta}^{(t)},\bm{\gamma}^{(t)}) + \mathbf C_{f(\cdot)-\bm\Delta_{\bar\psi_0}(\cdot,\bm\gamma^{(t)})}(\bm\alpha^{(t+1)},\bm\beta^{(t)},\theta_t)$.

Under $\mu_0=0$,   \eqref{befrec1_acc2} implies that
\begin{equation}
\begin{split}
&\frac{1}{\theta_t^2\rho_t}\big[f(\bm{\beta}^{(t+1)})-f(\bm{\beta})\big] - \frac{1-\theta_t}{\theta_{t}^2\rho_{t}}\big[f(\bm{\beta}^{(t)})-f(\bm{\beta})\big] + \frac{\mathcal E_t(\bm\beta)}{\theta_t\rho_t} + \frac{R_t}{\theta_t^2\rho_t}\\
\le\,&\bm\Delta_\phi(\bm{\beta},\bm{\alpha}^{(t)}) - \bm\Delta_\phi(\bm{\beta},\bm{\alpha}^{(t+1)}).
\end{split}
\end{equation}
Since in this case  \eqref{acc2_search2} gives $ ({1-\theta_t})/{\theta_{t}^2\rho_{t}}=  {1}/{\theta_{t-1}^2 \rho_{t-1}}$ for any $t\ge 1$, we obtain the first conclusion
\begin{equation} \nonumber
\begin{split}
&\frac{1}{\theta_T^2\rho_T}\big[f(\bm{\beta}^{(T+1)})-f(\bm{\beta})\big] - \frac{1-\theta_0}{\theta_0^2\rho_0}\big[f(\bm{\beta}^{(0)})-f(\bm{\beta})\big] + \sum_{t=0}^T \Big(\frac{\mathcal E_t(\bm\beta)}{\theta_t\rho_t} + \frac{R_t}{\theta_t^2\rho_t}\Big)\\
\le\,&\bm\Delta_\phi(\bm{\beta},\bm{\alpha}^{(0)}) - \bm\Delta_\phi(\bm{\beta},\bm{\alpha}^{(T+1)}).
\end{split}
\end{equation}
%

On the other hand, given  $\mu_0\ge 0$, \eqref{befrec1_acc2} can be written as
\begin{equation}\label{befrec1_acc2-2}
\begin{split}
&f(\bm{\beta}^{(t+1)})-f(\bm{\beta}) - (1-\theta_t)[f(\bm{\beta}^{(t)})-f(\bm{\beta})]    \\
& + R_t +\theta_t\bm\Delta_{\bar\psi_0}(\bm{\beta},\bm{\gamma}^{(t)})+ \theta_t\bm\Delta_{f(\cdot)-\bm\Delta_{\psi_0}(\cdot,\bm\gamma^{(t)})}(\bm\beta,\bm\alpha^{(t+1)}) \\ & + \theta_t  (\mu_0\bm\Delta_{\bm\Delta_\phi(\cdot,\bm\gamma^{(t)})-\phi(\cdot)}+\theta_t \rho_t\bm\Delta_{\bm\Delta_\phi(\cdot,\bm\alpha^{(t)})-\phi(\cdot)})(\bm\beta,\bm\alpha^{(t+1)})
\\
\le\,&\theta_t^2\rho_t\bm\Delta_\phi(\bm{\beta},\bm{\alpha}^{(t)}) -\theta_t^2(\rho_t +\frac{\mu_0}{\theta_t}) \bm\Delta_\phi(\bm{\beta},\bm{\alpha}^{(t+1)}).
\end{split}
\end{equation}
Therefore, we have
\begin{equation}\nonumber
\begin{split}
& f(\bm{\beta}^{(t+1)})-f(\bm{\beta})   +\theta_t^2(\rho_t +\frac{\mu_0}{\theta_t}) \bm\Delta_\phi(\bm{\beta},\bm{\alpha}^{(t+1)})+ \theta_t \mathcal E_t(\bm\beta)  +  R_t \\
\le\,& (1-\theta_t)[f(\bm{\beta}^{(t)})-f(\bm{\beta})]+\theta_t^2\rho_t \bm\Delta_\phi(\bm{\beta},\bm{\alpha}^{(t)}), \forall t\ge 0
\end{split}
\end{equation}
and from  \eqref{acc2_search2},
\begin{equation}\nonumber
\begin{split}
& f(\bm{\beta}^{(t+1)})-f(\bm{\beta})   +\theta_t^2(\rho_t +\frac{\mu_0}{\theta_t}) \bm\Delta_\phi(\bm{\beta},\bm{\alpha}^{(t+1)})+ \theta_t \mathcal E_t(\bm\beta)  +  R_t \\
\le\,& (1-\theta_t)[f(\bm{\beta}^{(t)})-f(\bm{\beta})+\theta_{t-1}^2(\rho_{t-1} +\frac{\mu_0}{\theta_{t-1}})\Delta_\phi(\bm{\beta},\bm{\alpha}^{(t)})], \forall t\ge 1
\end{split}
\end{equation}
The second conclusion can be obtained  by a recursive argument and  $ R_T +\theta_T \mathcal E_T(\bm\beta )  +(1-\theta_T)(R_{T-1} + \theta_{T-1} \mathcal E_{T-1}(\bm\beta)  )+ \cdots (1-\theta_T)\cdots (1-\theta_1)(R_{0}+ \theta_{0} \mathcal E_{0}(\bm\beta )) = \sum_{t=0}^T (\prod_{s=t+1}^T (1-\theta_s))( R_t + \theta_t \mathcal E_t(\bsb{\beta} )$).
\begin{remark}\label{rmk:relaxedrhot}
With the  `$=$' in \eqref{acc2_search2}   replaced by `$\le $',  \eqref{eq:th_acc2-conclusion2} still holds
when $\breg_\phi\ge 0$ (or $\phi$ is convex), and  \eqref{acc2_result}    still holds
  if we set $\bsb{\beta} $ to be a  minimizer of $f$.   But the   equality form of \eqref{acc2_search2}   makes   our conclusions applicable to say the noise-free statistical truth    $\bsb{\beta}= \bsb{\beta}^*$, which may not be a minimizer of the sample-based objective. The same comment applies to Theorem \ref{th:comp_acc1}.

Also, it is trivial to see that the conclusions extend to  a varying sequence of    $\mu_t$. (Concretely, the $\mu_0$ in  \eqref{acc2-alg2}, \eqref{psi0bardef}, and  $\mathcal E_t(\bm\beta) $ becomes $\mu_t$, and the $\mu_0$  in  \eqref{acc2_search2}, \eqref{eq:th_acc2-conclusion2} becomes $\mu_{t-1}, \mu_T$, respectively.) One can add  backtracking for $\mu_t$ in the algorithm to further reduce its iteration complexity.
\end{remark}

Finally, we prove  Corollary \ref{cor-AccIQ-gen}    which implies
Corollary \ref{cor-AccIQ} and applies to any convex  $l_0$ in \eqref{eq:featurescreening} below. The proof is based on an accumulative $R_t$ bound that can be derived in a more general setup; see \eqref{eq:sumRcontrol} in Remark \ref{rmk:acc2univrho}.

Here,  the   optimization problem of interest in ``variable screening'' is
\begin{align}\label{eq:featurescreening}
\min l(\bsb{\beta})=l_0(\bsb{X}\bsb{\beta}) \mbox{ s.t. } \| \bsb{\beta}\|_0\le q,\end{align}  to estimate the target   $\bsb{\beta}^*$ satisfying the  strict inequality   $\| \bsb{\beta}^*\|_0< q$. Take $\mu_0=0$,   $\phi = \|\cdot \|_2^2/2$, $\psi_0(\bsb{\beta})= l(\bsb{\beta})  -    \mathcal L  \phi(\bsb{\beta}) $ for some $\mathcal L \ge 0$.

Given $l_0$, $\bsb{X}$, and $s\le p$, we  extend the notion of
restricted isometry numbers $\rho_+, \rho_-$ \citep{Candes2005}:
\begin{equation}\label{eq:genrip}
\rho_-(s) \Breg_2(\bsb{\beta}, \bsb{\gamma}) \le \breg_{l_0}(\bsb{X} \bsb{\beta}, \bsb{X} \bsb{\gamma})\le  \rho_+(s) \Breg_2(\bsb{\beta}, \bsb{\gamma}),  \forall   \bsb{\beta}, \bsb{\gamma}: \| \bsb{\beta} -\bsb{\gamma}\|_0 \le s.
\end{equation}
(The dependence on $\bsb{X}$ and $l_0$ is dropped for the sake of brevity.)

\begin{manualcor}{\ref{cor-AccIQ}'}\label{cor-AccIQ-gen}
 Let $l_0$ be any convex function and  $q$ be any nonnegative  integer no more than $p$. As long as $r = q/\|\bsb{\beta}^*\|_0>1$, for any  $\mathcal L \ge \rho_+(2q)/\sqrt r$, there must  exist a universal $\rho_t =\rho_0,  \forall t$, e.g.,$$\rho_0 = (1- 1/\sqrt r )\rho_+(2q), $$    and thus $\theta_{t+1} = (\sqrt{\theta_t^4 + 4\theta_t^2} - \theta_t^2)/2$, so that the accelerated iterative quantile-thresholding according to \eqref{acc2-alg1}--\eqref{acc2-alg3} satisfies
\begin{align*}
\frac{l(\bm{\beta}^{(T+1)})-l(\bm{\beta}^*)}{\theta_T^2 } + T\cdot\mathop{\avg}_{0\le t\le T}\frac{\breg_{\psi_0}(\bm\beta^*,\bm\gamma^{(t)})}{\theta_t}   + \frac{\mathcal L(1-\theta_T)  }{\theta_T^3}\Breg_2 (\bm{\beta}^{(T+1)},\bm{\beta}^{(T)}) \\ \le \rho_0 \Breg_2 (\bsb{\beta}^*, \bsb{\alpha}^{(0)}), \mbox{ for all }
T\ge 0,
\end{align*}
or $ l(\bm{\beta}^{(T+1)})-l(\bm{\beta}^*) + \min_{0\le t\le T} \breg_{\psi_0}(\bm\beta^*,\bm\gamma^{(t)}) \le A / T^2 $ with $A$ independent of $T$.
\end{manualcor}

Seen from the error measure,  $r$ should be appropriately large  (but cannot be too large from the perspective of   statistical accuracy.)

To prove the corollary, we first introduce a useful result \citep[Lemma 9]{She2017RRRR}. \begin{lemma} \label{lemma:bound without design}
 Given  $ q\le p$,  $\hat{\bm\beta}=\Theta^{\#}(\bm y; q)$ is a globally optimal solution to  $\min_{\bm\beta\in\mathbb R^p}l(\bm\beta) = \|\bm y - \bm\beta\|_2^2/2 $ s.t. $\|\bm\beta\|_0\le q$. Let $\mathcal J = \mathcal J(\bm\beta)$, $\hat{\mathcal J} = \mathcal J(\hat{\bm\beta})$ and assume      $J(\hat{\bm\beta}) = q$. Then, for any $\bm\beta$ with $J(\bm\beta)\le s = q/r$ and $r\ge 1$,   $$l(\bm\beta)-l(\hat{\bm\beta}) \ge \{1-\mathcal L(\mathcal J, \hat{\mathcal J})\}\Breg_2 (\hat{\bm\beta},  \bm\beta),$$ where $\mathcal L(\mathcal J,\hat{\mathcal J}) = (|\mathcal J \backslash \hat{\mathcal J}| / |\hat{\mathcal J}\backslash \mathcal J|)^{1/2} \le (s/q)^{1/2} = r^{-1/2}$.
\end{lemma}

Set  $\bsb{\beta} = \bsb{\beta}^*$ in the previous proof, and apply,  instead of Lemma \ref{lemma:f_triangle}, Lemma \ref{lemma:bound without design} (where $\| \bm{\alpha}^{(t+1)}\|_0=q$ due to the no-tie-occurring assumption) to   \eqref{acc2-alg2}.    \eqref{tri_acc2} is then replaced by
\begin{align*} 
\begin{split}
&l(\bm{\alpha}^{(t+1)}) - \bm\Delta_{\psi_0}(\bm{\alpha}^{(t+1)}, \bm{\gamma}^{(t)}) - l(\bm{\beta}^*) + \bm\Delta_{\psi_0}(\bm{\beta}^*, \bm{\gamma}^{(t)}) + \theta_t\rho_t\Breg_2  (\bm{\alpha}^{(t+1)},\bm{\alpha}^{(t)})\\
\le\,& \theta_t\rho_t\Breg_2  (\bm{\beta}^*,\bm{\alpha}^{(t)}) -  (\theta_t\rho_t   + \mathcal L  )(1-\frac{1}{\sqrt r})\Breg_2(\bm{\beta}^*, \bm{\alpha}^{(t+1)}). \end{split}
\end{align*}
 Accordingly, \eqref{befrec1_acc2} becomes
\begin{equation}\nonumber
\begin{split}
&l(\bm{\beta}^{(t+1)})-l(\bm{\beta}^*) - (1-\theta_t)(l(\bm{\beta}^{(t)})-l(\bm{\beta}^*)) + R_t \\
&+ \theta_t\Big\{\bm\Delta_{\psi_0}(\bm{\beta}^*,\bm{\gamma}^{(t)}) + \big[  \mathcal L   (1-\frac{1}{\sqrt r}) -\frac{\theta_t\rho_t}{\sqrt r}\big] \Breg_2(\bm\beta^*,\bm\alpha^{(t+1)})\Big\} \\
\le\,&\theta_t^2\rho_t(\Breg_2(\bm{\beta}^*,\bm{\alpha}^{(t)}) - \Breg_2  (\bm{\beta}^*,\bm{\alpha}^{(t+1)})),
\end{split}
\end{equation}
where $R_t$ is  the same as before. Based on  Lemma \ref{lemma:Delta_linear}, Lemma \ref{lemma:Delta_Delta}, Lemma \ref{lemma:C} (together with some  results in its proof), \eqref{eq:genrip},  and the following facts
\begin{align*}
&\bm{\beta}^{(t)}-\bm{\gamma}^{(t)} = \theta_t (\bm{\beta}^{(t)}-\bm{\alpha}^{(t)})= \theta_t (1-\theta_{t-1}) (\bm{\beta}^{(t-1)}-\bm{\alpha}^{(t)}), \forall t\ge 1\\
&\bm{\beta}^{(t+1)}  -\bm{\gamma}^{(t)}  =\theta_t(\bm{\alpha}^{(t+1)}
-\bm{\alpha}^{(t)}), \forall t\ge 0
\end{align*}
we obtain for all $t\ge 1$,
\begin{align*}
R_t \ge  \, &  \theta_t^2\rho_t\Breg_2  (\bm{\alpha}^{(t+1)},\bm{\alpha}^{(t)}) - \bm\Delta_{\psi_0}(\bm{\beta}^{(t+1)},\bm{\gamma}^{(t)}) + (1-\theta_t)\bm\Delta_{\psi_0}(\bm{\beta}^{(t)},\bm{\gamma}^{(t)})\\& + \mathcal L \theta_t(1-\theta_t) \Breg_2(\bm\alpha^{(t+1)},\bm\beta^{(t)}) \\ = \, &\theta_t^2 ( \rho_t +\mathcal L)\Breg_2  (\bm{\alpha}^{(t+1)},\bm{\alpha}^{(t)}) -\bm\Delta_{l_0}(\bsb{X}\bm{\beta}^{(t+1)},\bsb{X}\bm{\gamma}^{(t)})+ (1-\theta_t)\bm\Delta_{l}(\bm{\beta}^{(t)},\bm{\gamma}^{(t)}) \\ &  + \mathcal L \{\theta_t(1-\theta_t) \Breg_2(\bm\alpha^{(t+1)},\bm\beta^{(t)})-(1-\theta_t)\Breg_2 (\bm{\beta}^{(t)},\bm{\gamma}^{(t)}) \} \\ = \, &\theta_t^2 ( \rho_t +\mathcal L)\Breg_2  (\bm{\alpha}^{(t+1)},\bm{\alpha}^{(t)}) -\bm\Delta_{l_0}(\bsb{X}\bm{\beta}^{(t+1)},\bsb{X}\bm{\gamma}^{(t)})+ (1-\theta_t)\bm\Delta_{l}(\bm{\beta}^{(t)},\bm{\gamma}^{(t)}) \\ &  + \mathcal L \{\theta_t(1-\theta_t) \Breg_2(\bm\alpha^{(t+1)},\bm\beta^{(t)})-(1-\theta_t)  \theta_t^2 (1-\theta_{t-1})^2 \Breg_2 (\bm{\alpha}^{(t)},\bm{\beta}^{(t-1)}) \}\\ \ge \, &\theta_t^2 ( \rho_t +\mathcal L)\Breg_2  (\bm{\alpha}^{(t+1)},\bm{\alpha}^{(t)}) -\theta_t^2\rho_+(2q) \Breg_2  (\bm{\alpha}^{(t+1)},\bm{\alpha}^{(t)})  \\ &
+ \mathcal L\theta_t(1-\theta_t) \{ \Breg_2(\bm\alpha^{(t+1)},\bm\beta^{(t)})-  \theta_t  (1-\theta_{t-1})^2 \Breg_2 (\bm{\alpha}^{(t)},\bm{\beta}^{(t-1)}) \},
\end{align*}
  and  $R_t \ge  \theta_t^2 \{  \rho_t +\mathcal L-\rho_+(2q)\}\Breg_2  (\bm{\alpha}^{(t+1)},\bm{\alpha}^{(t)})  $ as  $t = 0$.
 It follows that
\begin{align*}
\sum_{t=0}^T \frac{R_t}{\theta_t^2\rho_t} \ge  \,&\sum_{t=0}^T \frac{ \rho_t +\mathcal L-\rho_+(2q)}{\rho_t}\Breg_2  (\bm{\alpha}^{(t+1)},\bm{\alpha}^{(t)})+   \mathcal L\frac{1-\theta_T}{\theta_T\rho_T}\Breg_2 (\bm{\alpha}^{(T+1)},\bm{\beta}^{(T)}) \\& + \mathcal L\sum_{t=0}^{T-1} \Big\{\frac{1-\theta_t}{\theta_t \rho_t}-\frac{(1-\theta_t)^{2}(1-\theta_{t+1})}{\rho_{t+1}} \Big\} \Breg_2(\bm\alpha^{(t+1)},\bm\beta^{(t)}) \\=\,& \sum_{t=0}^T \frac{ \rho_t +\mathcal L-\rho_+(2q)}{\rho_t}\Breg_2  (\bm{\alpha}^{(t+1)},\bm{\alpha}^{(t)})+ \frac{\mathcal L(1-\theta_T)  }{\theta_T^3\rho_T}\Breg_2 (\bm{\beta}^{(T+1)},\bm{\beta}^{(T)}) \\&+ \mathcal L\sum_{t=0}^{T-1} \frac{1-\theta_t}{\theta_t \rho_t}\Big\{1-(1-\theta_t)\theta_{t+1}\frac{\theta_{t+1}}{\theta_{t}} \Big\} \Breg_2(\bm\alpha^{(t+1)},\bm\beta^{(t)}) .
\end{align*}
Therefore, choosing a  universal $\rho_t = \rho_0\ge \rho_+(2q) - \mathcal L$ (which implies $\theta_t\downarrow $) ensures
 $\sum_{t=0}^T  {R_t}/({\theta_t^2\rho_t}) \ge 0$.

Moreover,  $\mathcal E_t(\bm\beta^*) \ge  \bm\Delta_{\psi_0}(\bm{\beta}^*,\bm{\gamma}^{(t)})$ holds
 under $
\big\{  \mathcal L   (1-\frac{1}{\sqrt r}) -\frac{\theta_t\rho_t}{\sqrt r}\big\} \Breg_2(\bm\beta^*,\allowbreak\bm\alpha^{(t+1)})\}\ge 0$ or $L   (1- {1}/{\sqrt r}) \ge  {\rho_t}/\sqrt r$.
It is easy to see that as long as $r>1$, there exist positive  $\mathcal L, \rho_0$ satisfying \begin{align}
\begin{cases}
\mathcal L (\sqrt r - 1)\ge \rho_0 \\
\rho_0 \ge \rho_+(2q) - \mathcal L.
\end{cases}
\end{align}
Furthermore,  for  any $\mathcal L \ge  \rho_+(2q) /\sqrt r$, we can always choose $\rho_0 = (1- 1/\sqrt r )\rho_+(2q)$.
The rest of the proof proceeds as before.

\begin{remark}\label{rmk:acc2univrho}
The idea of controlling the overall $\sum_{ t\le T}\frac{R_t}{\theta_t^2\rho_t} $ can be extended with a proper choice of $\psi_0$  to a general problem   $\min f(\bsb{\beta})$ that may be nonconvex. In fact, if   $f(\bsb{\beta})$ can be  decomposed as $l(\bsb{\beta})+P(\bsb{\beta})$ with $0\le \breg_l\le L \Breg_2$ and $ \breg_P+\mathcal L_0 \Breg_2 \ge 0$ for some finite $\mathcal L_0\ge 0$, then setting $\mu_0=0$,  $\psi_0 = l - \mathcal L \| \cdot \|_2^2/2$ with $\mathcal L \ge \mathcal L_0$ and repeating the previous arguments, we obtain
\begin{equation}\label{eq:sumRcontrol}
\begin{split}
 \sum_{t=0}^T  \frac{R_t}{\theta_t^2\rho_t} &  \ge
  \sum_{t= 0}^T \frac{1}{\rho_t} (\rho_t\bm\Delta_\phi +\mathcal L \Breg_2 - L\Breg_2)(\bm{\alpha}^{(t+1)},\bm{\alpha}^{(t)}) \\ +\sum_{t=0}^{T-1}& \frac{ 1}{\theta_{t}^2\rho_{t}}\bm\Delta_{l}(\bm{\beta}^{(t+1)},\bm{\gamma}^{(t+1)})  + \frac{(\mathcal L-\mathcal L_0)(1-\theta_T)  }{\theta_T^3\rho_T}\Breg_2 (\bm{\beta}^{(T+1)},\bm{\beta}^{(T)}) \\+ \sum_{t=0}^{T-1}& \mathcal L\frac{1-\theta_t}{\theta_t \rho_t}\Big\{\frac{\mathcal L - \mathcal L_0}{\mathcal L}-(1-\theta_t)\theta_{t+1}\frac{\theta_{t+1}}{\theta_{t}} \Big\} \Breg_2(\bm\alpha^{(t+1)},\bm\beta^{(t)}) .
\end{split}\end{equation}
With  $\rho_t = \rho_0$, it can be shown that
$(1-\theta_t)\theta_{t+1}\frac{\theta_{t+1}}{\theta_{t}}$ achieves the maximum value $0.1612$ at $t = 3$ and so $\mathcal L \ge 1.2 \mathcal L_0$ makes the last term nonnegative. (A varying $\rho_t$ may reduce $\mathcal L$ further.)  Therefore, under $\breg_\phi\ge \sigma\Breg_2$,  we can choose any $\rho_0 \ge (L -\mathcal L)/\sigma$ and  $\mathcal L \ge 1.2 \mathcal L_0$  so that for any $\bsb{\beta}$, \begin{align*}
\frac{f(\bm{\beta}^{(T+1)})-f(\bm{\beta} )}{\theta_T^2 } + T\cdot\mathop{\avg}_{0\le t\le T}\frac{\mathcal E_t(\bm\beta)}{\theta_t}   + \frac{(\mathcal L-\mathcal L_0)(1-\theta_T)  }{\theta_T^3}\Breg_2 (\bm{\beta}^{(T+1)},\bm{\beta}^{(T)}) \\ \le \rho_0 \Breg_2 (\bm{\beta} , \bsb{\alpha}^{(0)}), \forall  T\ge 0.
\end{align*}
Hence for some $A$ independent of $T$,
 $ f(\bm{\beta}^{(T+1)})-f(\bm{\beta} ) + \min_{0\le t\le T} \mathcal E_t(\bm\beta) \le A / T^2 $, or    $ f(\bm{\beta}^{(T+1)})-f(\bm{\beta} ) + \min_{0\le t\le T} \{\breg_{\psi_0}(\bm\beta ,\bm\gamma^{(t)}) + \bm\Delta_{f(\cdot)-\bm\Delta_{\psi_0}(\cdot,\bm\gamma^{(t)})} (\bm\beta, \allowbreak\bm\alpha^{(t+1)})\}\le A / T^2 $ when $\phi$ is differentiable.

\end{remark}

\subsection{Proof of Theorem \ref{th:comp_acc1}}

The construction of the new acceleration scheme and the proof are motivated  by Proposition 2 of  \cite{Tseng2008}, with the use of GBF calculus. First, from  Lemma \ref{lemma:f_triangle}, given any $\bm\beta'_t$,
\begin{equation}\label{eq:th_acc1-1}
\begin{split}
&f(\bm{\beta}^{(t+1)}) - f(\bm{\beta}'_t)+ (\rho_t\mathbf D_2-\bm\Delta_{\bar\psi_0})(\bm{\beta}^{(t+1)},\bm{\gamma}^{(t)})+ \bm{\Delta}_f(\bm{\beta}'_t,\bm{\beta}^{(t+1)})  \\
\le &(\rho_t\mathbf D_2-\bm\Delta_{\bar\psi_0})(\bm{\beta}'_t,\bm{\gamma}^{(t)}) - (\rho_t\mathbf D_2- \bm\Delta_{\bm\Delta_{\bar\psi_0}(\cdot,\bm\gamma^{(t)})})(\bm{\beta}'_t,\bm{\beta}^{(t+1)})
\end{split}
\end{equation}
Let   $\bm{\beta}'_t = \theta_t\bm{\beta} + (1-\theta_t)\bm{\beta}^{(t)}$ with $\theta_t$ to be determined. Define $h_t(\cdot) = f(\cdot) - \bm\Delta_{\bar\psi_0}(\cdot,\bm\gamma^{(t)})$. By the definition of $\mathbf C$,
\begin{equation} \nonumber
\begin{split}
- f(\bm\beta'_t) =\,& \theta_t\bm\Delta_{\bar\psi_0}(\bm{\beta},\bm{\gamma}^{(t)}) + (1-\theta_t)\bm\Delta_{\bar\psi_0}(\bm{\beta}^{(t)},\bm{\gamma}^{(t)}) - \bm\Delta_{\bar\psi_0}(\bm{\beta}'_t,\bm{\gamma}^{(t)})\\
&-\theta_t f(\bm\beta) - (1-\theta_t)f(\bm\beta^{(t)}) + \mathbf C_{h_t}(\bm\beta,\bm\beta^{(t)},\theta_t).
\end{split}
\end{equation}
Plugging the last equality into \eqref{eq:th_acc1-1} yields
\begin{align*}
\begin{split}
&f(\bm{\beta}^{(t+1)}) - f(\bm{\beta}) - (1-\theta_t)(f(\bm{\beta}^{(t)}) - f(\bm{\beta}))\\& + (\rho_t\mathbf D_2-\bm\Delta_{\bar\psi_0})(\bm{\beta}^{(t+1)},\bm{\gamma}^{(t)})+ \mathbf C_{h_t}(\bm\beta,\bm\beta^{(t)},\theta_t)  \\
&+ \theta_t\bm\Delta_{\bar\psi_0}(\bm{\beta},\bm{\gamma}^{(t)}) + (1-\theta_t)\bm\Delta_{\bar\psi_0}(\bm{\beta}^{(t)},\bm{\gamma}^{(t)})  \\
\le\,& (\rho_t\mathbf D_2-\bm\Delta_{\bar\psi_0})(\bm\beta'_t,\bm\gamma^{(t)}) - (\rho_t\mathbf D_2 - \bm\Delta_{\bm\Delta_{\bar\psi_0}(\cdot,\bm\gamma^{(t)})})(\bm\beta'_t,\bm\beta^{(t+1)})  \\
& + \bm\Delta_{\bar\psi_0}(\bm\beta'_t,\bm\gamma^{(t)}) - \bm\Delta_f(\bm\beta'_t,\bm\beta^{(t+1)})  \\
=\,&\rho_t[\mathbf D_2(\bm{\beta}'_t,\bm{\gamma}^{(t)}) - \mathbf D_2(\bm{\beta}'_t,\bm{\beta}^{(t+1)})] - \bm\Delta_{h_t}(\bm\beta'_t,\bm\beta^{(t+1)}) %
\end{split}\end{align*}
and based on the definition of $\bar \psi_0$ and $R_t$,
\begin{align*}
\begin{split}
&f(\bm{\beta}^{(t+1)}) - f(\bm{\beta})+R_t  \\& +\rho_t \mathbf D_2(\bm{\beta}'_t,\bm{\beta}^{(t+1)})+ \mathbf C_{ \mu_0\Breg_{2}(\cdot,\bm\gamma^{(t)})}(\bm\beta,\bm\beta^{(t)},\theta_t)+ \mu_0\Breg_{2}(\bm\beta'_t,\bm\beta^{(t+1)}) \\& +\theta_t\bm\Delta_{\bar\psi_0}(\bm{\beta},\bm{\gamma}^{(t)}) + \mathbf C_{f(\cdot) - \bm\Delta_{\psi_0}(\cdot,\bm\gamma^{(t)})}(\bm\beta,\bm\beta^{(t)},\theta_t)+\bm\Delta_{f(\cdot) - \bm\Delta_{\psi_0}(\cdot,\bm\gamma^{(t)})}(\bm\beta'_t,\bm\beta^{(t+1)})   \\
\le \,&(1-\theta_t)(f(\bm{\beta}^{(t)}) - f(\bm{\beta}))+\rho_t\mathbf D_2(\bm{\beta}'_t,\bm{\gamma}^{(t)}).  %
\end{split}\end{align*}
From Section \ref{subsec:proofoflemma:C},  $\mathbf C_{ \mu_0\Breg_{2}(\cdot,\bm\gamma^{(t)})}(\bm\beta,\bm\beta^{(t)},\theta_t)=\mu_0 \mathbf C_{2 }(\bm\beta,\bm\beta^{(t)},\theta_t)=\mu_0 \theta_t(1-\theta_t)\Breg_{2 }(\bm\beta,\bm\beta^{(t)})$ and so
\begin{align}\label{eq:th_acc1-keybef}
\begin{split}
&f(\bm{\beta}^{(t+1)}) - f(\bm{\beta})+R_t  +   \theta_t \mathcal E_t(\bsb{\beta})\\ & +(\rho_t +\mu_0)\mathbf D_2(\bm{\beta}'_t,\bm{\beta}^{(t+1)})+ \mu_0 \theta_t(1-\theta_t)\Breg_{2 }(\bm\beta,\bm\beta^{(t)})    \\
\le \,&(1-\theta_t)(f(\bm{\beta}^{(t)}) - f(\bm{\beta}))+\rho_t\mathbf D_2(\bm{\beta}'_t,\bm{\gamma}^{(t)}).  %
\end{split}\end{align}

We would like to write $(\rho_t +\mu_0)\mathbf D_2( \theta_t\bm{\beta} + (1-\theta_t)\bm{\beta}^{(t)},\bm{\beta}^{(t+1)})+ \mu_0 \theta_t(1-\theta_t)\Breg_{2 }(\bm\beta,\bm\beta^{(t)})$ into the form of  a multiple of  $\mathbf D_2(\bm{\beta},\bm{\nu}^{(t+1)}) $ for some $\bsb{\nu}^{(t+1)}$. This can be done by  solving the gradient equation with respect to $\bsb{\beta}$:
\begin{align}
\bsb{\nu}^{(t+1)}= \frac{(\rho_t+\mu_0)\theta_t \bsb{\beta}^{(t+1)}-\rho_t\theta_t(1- \theta_t) {\bsb{\beta}}^{(t)}}{\rho_t \theta_t^2 + \mu_0 \theta_t}.\label{eq:nu++}
\end{align}
On the other hand,   $\nabla \rho_t\mathbf D_2( \theta_t\bm{\beta} + (1-\theta_t)\bm{\beta}^{(t)},\bm{\gamma}^{(t)})=\bsb{0}$
 gives
\begin{align}
\bsb{\nu}^{(t)}= \frac{\bm{\gamma}^{(t)}}{\theta_t} - \frac{ 1-\theta_t}{\theta_t}\bm{\beta}^{(t)}.\label{eq:nu+}  \end{align}
Combining \eqref{eq:nu++} and \eqref{eq:nu+} results in  \begin{equation}\bm{\gamma}^{(t)} = \bm{\beta}^{(t)} + \frac{\rho_{t-1} \theta_t(1-\theta_{t-1})}{\rho_{t-1}\theta_{t-1} + \mu_0} (\bm{\beta}^{(t)}-\bm{\beta}^{(t-1)}),\end{equation} as in \eqref{acc1-alg1}. Therefore,
 \eqref{eq:th_acc1-keybef} becomes \begin{align}\label{acc1-res}
\begin{split}
&f(\bm{\beta}^{(t+1)}) - f(\bm{\beta}) +(\theta_t^2\rho_t +\mu_0 \theta_t) \mathbf D_2(\bm\beta,\bm\nu^{(t+1)})+R_t  +   \theta_t \mathcal E_t(\bsb{\beta})    \\
\le \,&(1-\theta_t)(f(\bm{\beta}^{(t)}) - f(\bm{\beta}))+\theta_t^2\rho_t\mathbf D_2(\bm\beta,\bm\nu^{(t)}).  %
\end{split}
\end{align}

  Let $\mu_0=0$. It follows from    \eqref{acc1-res} that
\begin{equation}
\begin{split}
&\frac{1}{\theta_t^2\rho_t}\big[f(\bm{\beta}^{(t+1)})-f(\bm{\beta})\big] + \mathbf D_2(\bm{\beta},\bm{\nu}^{(t+1)}) + \frac{\mathcal E_t(\bm\beta)}{\theta_t\rho_t} + \frac{R_t}{\theta_t^2\rho_t}\\
\le\,&\frac{(1-\theta_t)}{\theta_t^2\rho_t}\big[f(\bm{\beta}^{(t)}) - f(\bm{\beta})\big] + \mathbf D_2(\bm{\beta},\bm{\nu}^{(t)}), \forall t\ge 0.
\end{split}\label{eq:th_acc1-7}
\end{equation}
Under \eqref{acc1_search2}, we have
\begin{align}
\begin{split}
&\frac{1}{\theta_t^2\rho_t}\big[f(\bm{\beta}^{(t+1)})-f(\bm{\beta})\big] + \mathbf D_2(\bm{\beta},\bm{\nu}^{(t+1)}) + \frac{\mathcal E_t(\bm\beta)}{\theta_t\rho_t} + \frac{R_t}{\theta_t^2\rho_t}\\
\le\,&\frac{1}{\theta_{t-1}^2\rho_{t-1}}\big[f(\bm{\beta}^{(t)}) - f(\bm{\beta})\big]+ \mathbf D_2(\bm{\beta},\bm{\nu}^{(t)}),~~\forall t\ge 1.
\end{split}\label{eq:th_acc1-8}
\end{align}
Summing \eqref{eq:th_acc1-8} for $t = T,\ldots,1$ and \eqref{eq:th_acc1-7} for $t = 0$ gives
\begin{equation}\nonumber
\begin{split}
&\frac{1}{\theta_T^2\rho_T}[f(\bm{\beta}^{(T+1)}) - f(\bm{\beta})] + \sum_{t=0}^T\Big(\frac{\mathcal E_t(\bm\beta)}{\theta_t\rho_t} + \frac{R_t}{\theta_t^2\rho_t}\Big)\\
\le\,& \frac{1-\theta_0}{\theta_0^2\rho_0}\big[f(\bm{\beta}^{(0)}) - f(\bm{\beta})\big] + \mathbf D_2(\bm{\beta},\bm\nu^{(0)}) - \mathbf D_2(\bm{\beta},\bm\nu^{(T+1)}), 
\end{split}
\end{equation}
and so the first bound  noticing that  $\bm\nu^{(0)} = \bm\gamma^{(0)} = \bm\beta^{(0)}$.

Moreover, given any $\mu_0\ge 0$, from \eqref{acc1_search2},  \eqref{acc1-res} implies for any $t\ge 1$,
\begin{align}\label{acc1-res-genmu0}
\begin{split}
&f(\bm\beta^{(t+1)}) - f(\bm\beta)   +(\theta_t^2\rho_t +\mu_0 \theta_t)\mathbf D_2(\bm\beta,\bm\nu^{(t+1)})  +R_t+\theta_t\mathcal E_t(\bm\beta)\\
\le\,&(1-\theta_t)[f(\bm\beta^{(t)})-f(\bm\beta)+(\theta_{t-1}^2\rho_{t-1} +\mu_0 \theta_{t-1})\mathbf D_2(\bm\beta,\bm\nu^{(t)}) ].
\end{split}
\end{align}
Similar to the proof of Theorem \ref{th:comp_acc2},  a recursive argument using \eqref{acc1-res-genmu0} and \eqref{acc1-res} gives the second bound.
\begin{remark}\label{rmk:1staccth}
Compared with the proof of Theorem \ref{th:comp_acc2}, the proof here needs to perform a finer analysis of $ \mathbf C_{h_t}$  (the proof of  Corollary \ref{cor-AccIQ-gen} uses a similar treatment). Otherwise one would  get $ \bm{\gamma}^{(t)} = \bm{\beta}^{(t)} + \theta_t(  \theta_{t-1}^{-1}-1) (\bm{\beta}^{(t)}-\bm{\beta}^{(t-1)})$ and  $\rho_t {\theta_t^2}/({1-\theta_t}) =  {\theta_{t-1}^2 (\rho_{t-1}+\mu_0)
}$
in place of \eqref{acc1-alg1}, \eqref{acc1_search2}, respectively. Following the same proof, we can show that the resultant algorithm does result in a linear rate when $\mu_0=\mu>0$, but  offers no acceleration  ($\theta_0 = 1/\kappa$) in strongly smooth and  convex optimization.

Finally, Remark \ref{rmk:relaxedrhot} still applies. For example, the second conclusion holds when the `$=$' in \eqref{acc1_search2} is replaced by `$\le$', and it is straightforward to see that  $\mu_0$ can be  similarly replaced by a sequence of varying $\mu_t$ to speed the convergence.
\end{remark}

\subsection{Statistical accuracy of LLA iterates} \label{subsec:LLA-stat}

In this subsection, assume  $f(\bm\beta) = l(\bm\beta) + P(\varrho\bm\beta)$,  $l(\bm\beta) = l_0(\bm X\bm\beta)$, $P(\varrho\bm\beta) = \sum_jP(\varrho\beta_j)$ (by a slight  abuse of notation), $P(0) = 0$,  $P'_+(0)<+\infty$,   $P(t)=P(-t)\ge 0$,  $P(t)$ is differentiable for any $t>0$, and    $P$ is concave on $(0,+\infty)$. Recall $\bm\Delta^{(t)}_{\mathrm{LLA}} = \bm\Delta_{\|\bm\alpha^{(t)}\text{\tiny$\circ$}(\cdot)\|_1 - P(\cdot)}$ which does \emph{not} satisfy the strong idempotence.\\ 

\noindent\textsc{Assumption} $\mathcal A(\varepsilon,\delta,\vartheta,K,\bm\alpha,\bm\beta)$ Given $\bm X, \bm\alpha, \bm\beta$, there exist $\varepsilon > 0, \delta > 0, \vartheta > 0, K \ge 0$ such that the following inequality holds
\begin{align}
&(1+\varepsilon)\bm\Delta_{\|\bm\alpha\text{\tiny$\circ$}(\cdot)\|_1 - P(\cdot)}(\varrho\bm{\beta}^*,\varrho\bm{\beta}) + \delta\mathbf D_2(\bm{X\beta}^*,\bm{X\beta}) + \vartheta P_H(\varrho(\bm{\beta} - \bm{\beta}^*);\lambda) \nonumber \\
\le\,& 2\sym{\bm{\Delta}}_{l}(\bm\beta^*,\bm\beta) + P(\varrho\bm{\beta};\lambda) - P(\varrho\bm{\beta}^*;\lambda) + K\lambda^2J^*. \nonumber 
\end{align}

\begin{pro}\label{thm:stat_LLA}
Assume that for any given $T\ge 1$, $\mathcal A(\varepsilon,\delta,\vartheta,K,\bm\alpha^{(t)},\bm\beta^{(t)})$ ($1\le t\le T$) is satisfied for some $\varepsilon > 0, \delta>0, \vartheta>0, K\geq 0$. Let $\lambda = A\sigma\sqrt{\log(ep)}/\sqrt{(\delta\wedge\vartheta )\vartheta}$.
Then the following inequality holds with probability at least $1-Cp^{-cA^2}$
\begin{equation} \nonumber 
\bm\Delta_\mathrm{LLA}^{(T)}(\varrho\bm{\beta}^*,\varrho\bm{\beta}^{(T)}) \leq \kappa^T\bm\Delta_\mathrm{LLA}^{(0)}(\varrho\bm{\beta}^*,\varrho\bm{\beta}^{(0)}) + \frac{\kappa}{1-\kappa} K\lambda^2J^*,
\end{equation}
where $\kappa = 1/(1+\varepsilon)$ and $C,c$ are universal positive constants.
\end{pro}

\begin{proof}
From the proof of Proposition \ref{pro:LLA_comp}, for any $\bm\beta$,
\begin{equation} \nonumber
\begin{split}
&\bm\Delta_f(\bm\beta,\bm\beta^{(t+1)}) + \bm\Delta_{\bm\Delta_{\mathrm{LLA}}^{(t)}(\cdot,\varrho\bm\beta^{(t)})}(\varrho\bm\beta,\varrho\bm\beta^{(t+1)})\\
\le\,& f(\bm\beta) - f(\bm\beta^{(t+1)}) + \bm\Delta_\mathrm{LLA}^{(t)}(\varrho\bm\beta,\varrho\bm\beta^{(t)}) - \bm\Delta_\mathrm{LLA}^{(t)}(\varrho\bm\beta^{(t+1)},\varrho\bm\beta^{(t)}).
\end{split}
\end{equation}
Using the definition of $\bm\Delta^{(t)}_{\mathrm{LLA}}$, we have
\begin{align}\label{stat-LLA-2}
\begin{split}
& \bm\Delta_{P}(\varrho\bm\beta,\varrho\bm\beta^{(t+1)})-\bm\Delta_{\bm\Delta_{P(\cdot,\varrho\bm\beta^{(t)})}}(\varrho\bm\beta,\varrho\bm\beta^{(t+1)})\\
&+\bm\Delta_l(\bm\beta,\bm\beta^{(t+1)}) +\sum_j\alpha_j^{(t)}\bm\Delta_{\bm\Delta_1(\cdot,\varrho\beta_j^{(t)})}(\varrho\beta_j,\varrho\beta_j^{(t+1)})  \\
\le\,& f(\bm\beta) - f(\bm\beta^{(t+1)}) + \bm\Delta_\mathrm{LLA}^{(t)}(\varrho\bm\beta,\varrho\bm\beta^{(t)}) ,
\end{split}
\end{align}
where we used  $\bm\Delta_{\mathrm{LLA}}^{(t)}(\varrho\bm\beta^{(t+1)},\varrho\bm\beta^{(t)}) \ge 0$ since  $P(\cdot)$ is concave on $(0,+\infty)$.


\begin{lemma}\label{lemma:Delta_Delta_P}
For any $P(\cdot)$ which is differentiable on $(0,+\infty)$ and satisfies $P(t) = P(-t)\ge 0, P(0)=0$ and $P'_+(0)<+\infty$, we have $\bm\Delta_{\bm\Delta_P(\cdot,\alpha)}(\beta,\gamma) = \bm\Delta_P(\beta,\gamma) - P'_+(0)\bm\Delta_1(\beta,\allowbreak\gamma) 1_{\alpha=0}$ for any $\alpha, \beta,\gamma\in\mathbb R$.
In particular, $\bm\Delta_{\bm\Delta_1(\cdot,\alpha)}(\beta,\gamma) = \bm\Delta_1(\beta,\gamma)1_{\alpha\ne 0}$.
\end{lemma}

The result can be shown from the proof of Lemma \ref{lemma:Delta_Delta}. Indeed, from \eqref{Delta_P},
\begin{equation} \nonumber
\bm\Delta_P(\cdot, \alpha) = \begin{cases}P(\cdot) - P(\alpha) - P'(\alpha)(\cdot - \alpha),&\alpha \ne 0\\ P(\cdot) - P'_+(0)|\cdot|,& \alpha = 0.\end{cases}
\end{equation}
When $\alpha \ne 0$, by Lemma \ref{lemma:Delta_linear} and Lemma \ref{lemma:Delta_Delta}, $\bm\Delta_{\bm\Delta_P(\cdot,\alpha)}(\beta,\gamma) = \bm\Delta_P(\beta,\gamma) - \bm\Delta_{P(\alpha) + P'(\alpha)(\cdot - \alpha)}(\beta,\gamma)\allowbreak= \bm\Delta_P(\beta,\gamma)$.
When $\alpha = 0$, $\bm\Delta_{\bm\Delta_P(\cdot,\alpha)}(\beta,\gamma) = \bm\Delta_P(\beta,\gamma) - P'_+(0)\bm\Delta_1(\beta,\gamma)$.
Combining the two cases gives
\begin{equation} \nonumber
\bm\Delta_{\bm\Delta_P(\cdot,\alpha)}(\beta,\gamma) =
\bm\Delta_P(\beta,\gamma)-P'_+(0)\bm\Delta_1(\beta,\gamma)1_{\alpha=0}.
\end{equation}
When $P(\bm\beta) = \|\bm\beta\|_1$, $\bm\Delta_{\bm\Delta_1(\cdot,\alpha)}(\beta,\gamma) = \bm\Delta_1(\beta,\gamma)-\bm\Delta_1(\beta,\gamma)1_{\alpha=0} = \bm\Delta_1(\beta,\gamma)1_{\alpha\ne 0}$.

\vspace{2ex}
From Lemma \ref{lemma:Delta_Delta_P},
\[\bm\Delta_{P}(\varrho\bm\beta,\varrho\bm\beta^{(t+1)})-\bm\Delta_{\bm\Delta_{P(\cdot,\varrho\bm\beta^{(t)})}}(\varrho\bm\beta,\varrho\bm\beta^{(t+1)}) = \sum_{j:\beta_j^{(t)}=0}P'_+(0)\bm\Delta_1(\varrho\beta_j,\varrho\beta_j^{(t+1)})\]
and
\[\sum_j\alpha_j^{(t)}\bm\Delta_{\bm\Delta_1(\cdot,\varrho\beta_j^{(t)})}(\varrho\beta_j,\varrho\beta_j^{(t+1)}) = \sum_{j:\beta_j^{(t)}\ne 0}\alpha_j^{(t)}\bm\Delta_1(\varrho\beta_j,\varrho\beta_j^{(t+1)}).\]
Plugging these into
\eqref{stat-LLA-2} gives $\bm\Delta_l(\bm\beta,\bm\beta^{(t+1)}) + \sum_{j:\beta_j^{(t)}=0}P'_+(0)\bm\Delta_1(\varrho\beta_j,\varrho\beta_j^{(t+1)}) + \sum_{j:\beta_j^{(t)}\ne 0}\allowbreak \alpha_j^{(t)}\bm\Delta_1(\varrho\beta_j,\varrho\beta_j^{(t+1)}) \le f(\bm\beta) - f(\bm\beta^{(t+1)}) + \bm\Delta_\mathrm{LLA}^{(t)}(\varrho\bm\beta,\varrho\bm\beta^{(t)})$.

Together with $\alpha_j^{(t)} = |P'_+(\beta_j^{(t)})| \le P'_+(0)$, we have
\begin{align}
\bm\Delta_l(\bm\beta,\bm\beta^{(t+1)}) + \sum_j\alpha_j^{(t)}\bm\Delta_1(\varrho\beta_j,\varrho\beta_j^{(t+1)}) \le f(\bm\beta) - f(\bm\beta^{(t+1)}) + \bm\Delta^{(t)}_\mathrm{LLA}(\varrho\bm\beta,\varrho\bm\beta^{(t)}).
\end{align}
Letting $\bm{\beta} = \bm{\beta}^*$ and  using  the definition of $\bm\epsilon$,
we obtain\begin{equation} \nonumber 
\begin{split}
&2\sym{\bm\Delta}_l(\bm\beta^*,\bm\beta^{(t+1)}) + \bm\Delta_{\|\bm\alpha^{(t)}\text{\tiny$\circ$}(\cdot)\|_1}(\varrho\bm\beta^*,\varrho\bm\beta^{(t+1)}) + P(\varrho\bm\beta^{(t+1)};\lambda)\\
\le\,& \bm\Delta^{(t)}_\mathrm{LLA}(\varrho\bm\beta^*,\varrho\bm\beta^{(t)}) + \langle\bm\epsilon,\bm X\bm\beta^{(t+1)}-\bm X\bm\beta^*\rangle + P(\varrho\bm\beta^*;\lambda).
\end{split}
\end{equation}
From the regularity condition,
\begin{align*}
&(1+\varepsilon)\bm\Delta_{\mathrm{LLA}}^{(t+1)}(\varrho\bm{\beta}^*,\varrho\bm{\beta}^{(t+1)}) + \delta\mathbf D_2(\bm{X\beta}^*,\bm{X\beta}^{(t+1)})\allowbreak + \vartheta P_H(\varrho(\bm{\beta}^{(t+1)} - \bm{\beta}^*);\lambda)
\\ \le & 2\sym{\bm{\Delta}}_{l}(\bm\beta^*,\bm\beta^{(t+1)}) + \bm\Delta_{\|\bm\alpha^{(t)}\text{\tiny$\circ$}(\cdot)\|_1}(\varrho\bm\beta^*,\varrho\bm\beta^{(t+1)}) + P(\varrho\bm{\beta}^{(t+1)};\lambda) - P(\varrho\bm{\beta}^*;\lambda) + K\lambda^2J^*\end{align*}
for $1\le t\le T$.
The final conclusion can be proved by combining the last two inequalities and then applying a similar probabilistic argument as in Theorem \ref{th:stat_bound}.
\end{proof}

\subsection{A-estimators as F-estimators}
\label{subsec:Aests}
In this part, we show that an important class of   \textbf{A}-estimators that has \textit{alternative} optimality, typically arising from block coordinate descent (BCD) algorithms, can often be converted to F-estimators, and analyzed in a similar way. Let ${ \bsb{\beta}} = [{ \bsb{\beta}}_{[1]}^T,\ldots,  { \bsb{\beta}}_{[K]}^T]^T$ where ${ \bsb{\beta}}_{[k]}$ is  the $k$th block, $1\le k\le K$, and we use ${ \bsb{\beta}}_{[-k]}$ to denote the subvector after removing the $k$th block. Assume $$f = l+ P$$ where $l$ is  differentiable, and $P$ is separable: $P(\bsb{\beta})=\sum P_k(\bsb{\beta}_{[k]})$.  When viewed as a function of $\bsb{\beta}_{[k]}$ only, $f$ is denoted by $f(\bsb{\beta}_{[k]};\bsb{\beta}_{[-k]}) $.  
We say  $\hat{ \bsb{\beta}}$ has alternative optimality or is an A-estimator if \begin{equation}\hat{ {\bsb{\beta}}}_{[k]}  \in \arg\min_{ {\bsb{\beta}}_{[k]} } f({\bsb{\beta}}_{[k]} ; \hat{\bsb{\beta}}_{[-k]} ), 1\le k \le K. \label{Aest:ch1}\end{equation}
\begin{lemma}\label{Lemma:AtoF-gen}
Let $\hat \beta$ be an A-estimator of  $\min f(\bsb\beta)$.   Construct a surrogate function:
\begin{align}
g_{\bsb{\rho}}(\bsb{\beta}; \bsb{\beta}^-) = f(\bsb{\beta}) - \breg_l (\bsb{\beta}, \bsb{\beta}^-) + \sum \rho_k \Breg_2 (\bsb{\beta}_{[k]}, \bsb{\beta}_{[k]}^-)
\end{align}
where $\bsb{\rho}=(\rho_1, \ldots, \rho_K)$ with $\rho_k\ge 0$.

(i) If $P_k$ are directionally differentiable and
\begin{align}\label{eq:AtoFpencond}\breg_{P_k} + \mathcal L_k \Breg_2\ge 0
\end{align} for some $\mathcal L_k\ge 0$, then for any $\rho_k\ge \mathcal L_k$, $\hat {\bsb{\beta}}$ must satisfy
\begin{equation}\hat{ {\bsb{\beta}}}   \in \arg\min_{ {\bsb{\beta}} } g_{\bsb{\rho}}(\bsb{\beta}; \bsb{\beta}^-) |_{\bsb{\beta}^- =  \hat{\bsb{\beta}}} . \label{Aest:ch2}
\end{equation}
(ii) If $l$ as a function of $\bsb{\beta}_k$ satisfies $\breg_{  l (\cdot; \bsb{\beta}_{[-k]})} \le L_k \Breg_2, \forall \bsb{\beta}_{[-k]}$,  $1\le k \le K$, or less restrictively,
\begin{align}\label{eq:AtoFlosscond} \breg_{\hat l_k(\cdot)}(\bsb{\beta}_{[k]}, \hat{ \bsb{\beta}}_{[k]})\le L_k \Breg_2(\bsb{\beta}_{[k]}, \hat{ \bsb{\beta}}_{[k]}), \ \forall \bsb{\beta}_{[k]}\end{align}   where  $\hat l_k(\bsb{\beta}_{[k]})$   denotes $l_k(\bsb{\beta}_{[k]}; \hat {\bsb{\beta}}_{[-k]})$, then for any $\rho_k>L_k$, \eqref{Aest:ch2} still holds.
In addition, if $\hat{ {\bsb{\beta}}}_{[k]}  $ is the unique solution to \eqref{Aest:ch1}, then  $\rho_k\ge L_k$ suffices.
\end{lemma}
Overall, \eqref{Aest:ch2}  provides a useful  \textit{joint} optimization form that can be used as the  so-called ``basic inequality'' in empirical process theory, and so with the lemma, A-estimators can analyzed   like F-estimators.  Moreover, the quality of the initial point can be incorporated in the analysis; see \cite{SheetalPIQ}. 

\begin{proof}

(i) The condition \eqref{eq:AtoFpencond} means that  $g_{\bsb{\rho}}$ is convex in,  $\bsb{\beta}_{[k]}$. By Lemma \ref{lemma:degeneracy} and Lemma \ref{lemma:Delta_linear}, and the fact that  $g_{\bsb{\rho}}$ is separable in $\bsb{\beta}_{[k]}$, $1\le k\le K$,   we immediately know that $\hat { \bsb{\beta}}$ is necessarily a solution to $\min_{ {\bsb{\beta}} } g_{\bsb{\rho}}(\bsb{\beta}; \hat{\bsb{\beta}})$.

(ii) We use a shorthand notation   $\hat g_{k}({\bsb{\beta}}_{[k]})$ to denote   $g_{\bsb{\rho}}(\bsb{\beta}; \hat{\bsb{\beta}}) $  as a function of ${\bsb{\beta}}_{[k]}$ when  $  {\bsb{\beta}}_{[-k]} =  \hat{\bsb{\beta}}_{[-k]}$. Let $\tilde {\bsb{\beta}}_{[k]} \in \arg\min  \hat g_{k}({\bsb{\beta}}_{[k]})$. It suffices to show $\tilde {\bsb{\beta}}_{[k]} = \hat {\bsb{\beta}}_{[k]}$. Because of the separability of  $g_{\bsb{\rho}}$, $$
f(\tilde {\bsb{\beta}}_{[k]} ; \hat{\bsb{\beta}}_{[-k]}) + (\rho_k - L_k)\Breg_2(\hat {\bsb{\beta}}_{[k]} , \tilde {\bsb{\beta}}_{[k]} ) \le f(\hat {\bsb{\beta}}_{[k]} ; \hat{\bsb{\beta}}_{[-k]})
$$
and so $(\rho_k - L_k)\Breg_2(\hat {\bsb{\beta}}_{[k]} , \tilde {\bsb{\beta}}_{[k]} ) =0$. The conclusion follows.
\end{proof}

Some conclusions like (i) can be extended to functions defined on Riemannian manifolds. It is also worth mentioning that
in the regression setup, which is of primary interest in many statistical applications, we can use    some surrogates with $\rho_k=1$, regardless of the design or penalty, to convert alternative optimality to joint optimality. The following lemma exemplifies the point  in matrix regression, and is condition free.
\begin{lemma} \label{Lemma:AtoF-l2}
Let $l_0(\bsb{A}; \bsb{Y}) = \| \bsb{Y}-\bsb{A}\|_F^2/2$, and $\bsb A$ be defined differently as follows.

(i) Let $\bsb{A}=\sum \bsb{X}_k\bsb{B}_k$ with $\bsb{B} = (\bsb{B}_1, \ldots, \bsb{B}_K) $, where the dependence of $\bsb B$ (and $\bsb{X}_k$) is dropped for simplicity. Consider   the problem
\begin{align}\label{regprob:sum}
\min_{\bsb{B}_1, \ldots, \bsb{B}_K} l_0(\bsb{A}; \bsb{Y}) + \sum P_k(\bsb{B}_k) \mbox{ s.t. } \bsb{A}  =\sum \bsb{X}_k\bsb{B}_k.
\end{align}
Then the set of A-estimators of \eqref{regprob:sum} is exactly the set of F-estimators associated with the following surrogate
\begin{align}
g(\bsb{B},\bsb{B}^-) = l_0(\bsb{A}; \bsb{Y}) -\Breg_{l_0} (\bsb{A},\bsb{A}^-) + \sum P_k(\bsb{B}_k)+ \sum \Breg_2 (\bsb{X}_k \bsb{B}_k, \bsb{X}_k \bsb{B}_k^-).
\end{align}

(ii) Let $\bsb{A}=  \bsb{X} \bsb{B}_1 \cdots \bsb{B}_K$ with $\bsb{B} = (\bsb{B}_1, \ldots, \bsb{B}_K) $, where the dependence of $\bsb B$ (and $\bsb{X}_k$) is dropped for simplicity.  Consider
\begin{align}\label{regprob:prod}
\min_{\bsb{B}_1, \ldots, \bsb{B}_K}   l_0(\bsb{A}; \bsb{Y}) + \sum P_k(\bsb{B}_k) \mbox{ s.t. } \bsb{A}=  \bsb{X} \bsb{B}_1 \cdots \bsb{B}_K.
\end{align}
Redefine  $l_0$ as a function $\bar l_0$ of $\bsb B$ and introduce a discrepancy measure $d^2$ as follows \begin{align*}
&l_0(\bsb{A};   \bsb{Y}) = \bar l_0(\bsb{B};\bsb{X}, \bsb{Y})\\
& d^2(\bsb{B}, \bsb{B}^-) = \frac{1}{2}\sum_{k=1}^K \|  \bsb{X} \bsb{B}_1^-\cdots \bsb{B}_{k-1} ^-(\bsb{B}_k-\bsb{B}_k^-) \bsb{B}_{k+1}^-\cdots \bsb{B}_{K}   ^-\|_F^2.
\end{align*}
Then the set of A-estimators of \eqref{regprob:prod} is exactly the set of F-estimators associated with  the surrogate
\begin{align}
\begin{split}
g(\bsb{B},\bsb{B}^-) & = \bar l_0(\bsb{B})-\breg_{\bar l_0} (\bsb{B},\bsb{B}^-) + \sum P_k(\bsb{B}_k)+ d^2( \bsb{B},  \bsb{B}^-)\\
 & = \frac{1}{2}\| \bsb{A}^- - \bsb{Y}\|_F^2     + \langle \bsb{X} \bsb{B}_1^-\cdots \bsb{B}_K^- - \bsb{Y},  \\ & \quad\sum_{k=1}^K \bsb{X} \bsb{B}_1^-\cdots \bsb{B}_{k-1} ^-(\bsb{B}_k-\bsb{B}_k^-) \bsb{B}_{k+1}^-\cdots \bsb{B}_{K}^-\rangle+ \sum P_k(\bsb{B}_k)\\&  \quad\frac{1}{2} \sum_{k=1}^K\|  \bsb{X} \bsb{B}_1^-\cdots \bsb{B}_{k-1} ^-(\bsb{B}_k-\bsb{B}_k^-) \bsb{B}_{k+1}^-\cdots \bsb{B}_{K}   ^-\|_F^2.
 \end{split}
\end{align}
\end{lemma}
The lemma can be directly proved by   the definition of GBF and matrix differentiation and its proof is omitted.
For the application of the first result (i), see \cite{She2017RRRR} for example. The second result can be used to study bilinear problems or NMF like matrix decomposition problems. One could show a statistical accuracy result in terms of $d^1$ (which satisfies $d^1\le K d^2$),  $$
d^1(\bsb{B}, \bsb{B}^-) = \frac{1}{2}\big\| \sum_{k=1}^K \bsb{X} \bsb{B}_1^-\cdots \bsb{B}_{k-1} ^-(\bsb{B}_k-\bsb{B}_k^-) \bsb{B}_{k+1}^-\cdots \bsb{B}_{K}^-  \big \|_F^2,
$$  under  a proper regularity condition involving $d^2$; see, for example, \cite{SheetalCRL}.

\subsection{Statistical error analysis of a general optimal solution}
\label{subsec:genopt} This part demonstrates that using the statistical notions and Bregman calculus  developed earlier can perform statistical analysis of a general optimization problem that may not be in the MLE setup:
\begin{align}
\min_{\bsb{\beta}} f(\bsb{\beta}) \mbox{ s.t. } \bsb{\beta} \in \mathcal S \label{genoptdef-global}
\end{align}
where $f$ is directionally differentiable and
$\mathcal S\subset \mathbb R^p$ can be formulated by  linear equality constraints $\bsb{A}\bsb{\beta}=\bsb{\alpha}$, sparsity constraints $\| \bsb{\beta}\|_0\le s$,  nonnegativity constraints $\bsb{\beta}\ge 0$, and so on.

Statistically, we would like to study how a target parameter can be recovered from solving \eqref{genoptdef-global} in the present of data noise. Following  \eqref{noise-def}, let   $\bsb{\beta}^*$ be a statistical truth and   define the associated effective noise by   $\bsb{\epsilon} =-\nabla f(\bsb{\beta}^*)$, assuming $f$ is differentiable at  $\bsb{\beta}^*$.

Although  $\bsb{\beta}^*$ in the above definition can be any point,     a meaningful recovery must be under some conditions satisfied  the associated $\bsb{\epsilon}$. Consider the following three scenarios:

(a)
Statistical estimation often assumes a zero mean noise:\begin{align}\EE \bsb{\epsilon} = \bsb{0},\label{assnoisemean0}\end{align} which essentially means that    the statistical truth  makes the gradient of the \textit{expectation} of $f$ (so as to remove data randomness) vanish---see Section \ref{subsec:stat}. Yet \eqref{assnoisemean0} alone does not always guarantee a unique   $\bsb{\beta}^*$.

(b) Stronger conclusions can be obtained for the $\bsb{\beta}^*$ that satisfies  the no-model-ambiguity  assumption: $ f$ is differentiable at    $   \bsb{\beta}^* \in D=\mbox{dom}(f)$ with the gradient  $   \nabla f(  \bsb{\beta}^*  )=-\bm\epsilon$,  $   \bsb{\beta}^* $ is a finite optimal solution  to the Fenchel conjugate as  $\bsb{\zeta} = - \bsb{\epsilon}$:
\begin{align}
f^*(\bsb{\zeta}) = \sup_{\bsb{\beta}} \langle \bsb{\zeta} ,\bsb{\beta} \rangle -f(\bsb{\beta}),\label{f-conj}
\end{align}
and the extended real-valued convex function $f^*$ is differentiable at $-\bsb{\epsilon}$. This assumption simply means that $(    \bsb{\beta}^* , -\bsb{\epsilon})$ makes a so-called ``conjugate pair''. Note that   $f$  need not be overall strictly convex, especially when $D$ is compact according to Danskin's min-max theorem \citep{Bertsekas1999Book}.

(c) Another popular assumption in  statistical learning is strong convexity in a restricted sense (especially when $\bsb{\beta} = \bsb{A} \bsb{\alpha} $ with $\|\bsb{\alpha}\|_0\le s$):
\begin{align}\label{statmodelass-scvx}
  ({\bm\Delta}_f -  \mu\Breg_2 )(   \bsb{\beta}_1 ,    \bsb{\beta}_2)
 \ge 0,\  \forall \bsb{\beta}_1, \bsb{\beta}_2\in \mathcal S   \end{align}
 for some $\mu>0$.
  The condition  may  hold even when the number of unknowns is much larger than the sample size \cite{Candes2005,Loh2015}.

The following theorem uses  the GBF calculus to argue how the statistical accuracy of the obtained solutions is determined by the (tail decay of) effective noise. Probabilistic arguments can follow to bound the stochastic terms more explicitly.

\begin{thm}\label{theorem:genalgananal}
Let $\hat {\bsb{\beta}}$ be an optimal solution to \eqref{genoptdef-global}.

 (i) Under the zero mean assumption \eqref{assnoisemean0}   and $\bsb{\beta}^*\in \mathcal S$, the  risk of $\hat {\bsb{\beta}}$ in terms of $\breg_f$  satisfies a  Fenchel-Young  form bound
 \begin{align}
\EE \breg_f(  \hat {\bsb{\beta}},     {\bsb{\beta}^*}) \le \EE  [f^{*} (\bsb{\epsilon})
 + f(\bsb{\beta}^*)].\label{bregfbound1}
\end{align}

(ii) Under the  no-model-ambiguity assumption in (b) and $\bsb{\beta}^*\in \mathcal S$, we have
  \begin{align}
  \breg_f(   \hat {\bsb{\beta}},     {\bsb{\beta}^*}) \le \breg_{f^*}(\bsb{\epsilon}, -\bsb{\epsilon}). \label{bregfbound2}
  \end{align}

(iii) An  \textit{oracle inequality}   holds for any $\delta>0$ and any reference $\bsb{\beta}\in \mathcal S$:
\begin{align}
( \breg_f -\delta \Breg_2)(\hat { \bsb{\beta}},    { \bsb{\beta}}^*)\le \breg_f(    { \bsb{\beta}},     { \bsb{\beta}}^*) + \frac{1}{2\delta} [\,\sup_{\bsb{\theta}\in \Gamma(\bsb{\beta})} \langle \bsb{\epsilon}, \bsb{\theta}  \rangle   ]^2,
  \end{align}
  where
 $ \Gamma(\bsb{\beta}) = \{   \bsb{\theta}: \|\bsb{\theta}\|_2 \le 1, \bsb{\theta} =     \bar {\bsb{\beta}} - \bsb{\beta} \mbox{ for some }    \bar {\bsb{\beta}} \in \mathcal S \}$.
In particular, under \eqref{statmodelass-scvx}, \begin{align}
   \Breg_2 (  \hat { \bsb{\beta}},    \bsb{\beta}^*  ) 
\le     \frac{1}{\mu } \breg_f(    { \bsb{\beta}},    \bsb{\beta}^*  )+\frac{1}{2\mu^2 }
 [\,\sup_{\bsb{\theta}\in \Gamma(\bsb{\beta})} \langle \bsb{\epsilon}, \bsb{\theta}  \rangle   ]^2. \end{align}
\end{thm}

The first two bounds reveal the important  role of  the Fenchel conjugate of the loss, and can be made more explicit under proper Orlicz norm conditions of $\bsb{\epsilon}$;  the third conclusion, on the basis of the supremum of  an empirical process \citep{van1996weak}, demonstrates  how modern probabilistic tools can be used to derive finite-sample error bounds of $\hat {\bsb{\beta}}$ in a general noisy setup.
\begin{proof}
First, by definition, $f(\bsb{}  \hat {\bsb{\beta}}) \le f(    {\bsb{\beta}^*})$, from which it follows that $\breg_f(  \hat {\bsb{\beta}},    {\bsb{\beta}^*}) \le \langle \bsb{\epsilon},    \hat {\bsb{\beta}}-      {\bsb{\beta}^*} \rangle$. Define
$$
h(\bsb{\delta}) = \breg_f(\bsb{\delta}+\bsb{\beta}^* , \bsb{\beta}^*).
$$
By assumption, $f$ is a proper function and applying Fenchel-Young's inequality gives
\begin{align*}
 \breg_f(   \hat {\bsb{\beta}},     {\bsb{\beta}^*}) \le \langle \bsb{\epsilon}, \bsb{\delta} \rangle|_{\bsb{\delta} =   \hat {\bsb{\beta}}-     {\bsb{\beta}^*}}\le \frac{1}{c}\breg_f(\bsb{\delta}+     {\bsb{\beta}^*} ,      {\bsb{\beta}^*})|_{\bsb{\delta} =    \hat {\bsb{\beta}}-    {\bsb{\beta}^*}}  + \frac{1}{c}h^*(c \bsb{\epsilon})
\end{align*}
or $(1-1/c)\breg_f(  \hat {\bsb{\beta}},    {\bsb{\beta}^*})\le h^*(c \bsb{\epsilon})/c $
for any $c>0$.

On the other hand,
\begin{align*}
h^*(\bsb{\zeta}) &= \sup_{\bsb{\delta}} \langle \bsb{\zeta}, \bsb{\delta} \rangle  - f(\bsb{\beta}^*+\bsb{\delta})
 + f(\bsb{\beta}^*) + \langle \nabla f(\bsb{\beta}^*), \bsb{\delta}\rangle \\
 &= \sup_{\bsb{\delta}} \langle \bsb{\zeta}+\nabla f(\bsb{\beta}^*), \bsb{\delta} \rangle  - f(\bsb{\beta}^*+\bsb{\delta})
 + f(\bsb{\beta}^*) \\
 &= \sup_{\bsb{\delta}} \langle \bsb{\zeta}+\nabla f(\bsb{\beta}^*),\bsb{\beta}^*+ \bsb{\delta} \rangle  - f(\bsb{\beta}^*+\bsb{\delta})
 + f(\bsb{\beta}^*)-  \langle \bsb{\zeta}+\nabla f(\bsb{\beta}^*),\bsb{\eta}^* \rangle \\
 &= f^{*} (\bsb{\zeta}+\nabla f(\bsb{\beta}^*))
 + f(\bsb{\beta}^*)- \langle \bsb{\zeta}+\nabla f(\bsb{\beta}^*),\bsb{\beta}^* \rangle,
 \end{align*}
 where $ f(\bsb{\beta}^*),  \nabla f(\bsb{\beta}^*)$ are known to be finite.
 Therefore we obtain
 \begin{align}
\breg_f(  \hat {\bsb{\beta}},    {\bsb{\beta}^*})\le \frac{1}{c-1} [f^{*} ((c-1) \bsb{\epsilon})
 + f(\bsb{\beta}^*)- (c-1)\langle \bsb{\epsilon},\bsb{\beta}^* \rangle ], \ \forall c>0.
\end{align} Taking $c=2$ and using $\EE \bsb{\epsilon} = \bsb{0}$ gives the   $\breg_f$ risk bound  \eqref{bregfbound1}.

Next,  we prove the second bound under the no-model-ambiguity assumption.
 Using the optimality of $\bsb{\beta}^*$, we have further
 \begin{align*}
h^*(\bsb{\zeta}) &= f^{*} (\bsb{\zeta}+\nabla f(\bsb{\beta}^*))
 - f^{*}(-\bsb{\epsilon})- \langle \bsb{\zeta},\bsb{\beta}^*\rangle.
\end{align*}
Moreover, from the assumption and definition \eqref{f-conj}, it is easy to show that $ \bsb{\beta}^*\in \partial f^{*}(-\bsb{\epsilon})$, and so   $\nabla f^{*}(-\bsb{\epsilon})=\bsb{\beta}^*$, from which it follows that
\begin{align}
h^*(\bsb{\zeta}) = \breg_{f^{*}} (\bsb{\zeta}-\bsb{\epsilon}, -\bsb{\epsilon}).
\end{align}
Taking $c=2$  gives \eqref{bregfbound2} (even though   $\EE \bsb{\epsilon} $ may not be $\bsb{0}$).

Finally,  for any    $\bsb{\beta}\in \mathcal S$,   $f(  \hat { \bsb{\beta}}) \le f(    { \bsb{\beta}}) $ and so
\begin{align*}
\breg_f(  \hat { \bsb{\beta}},    { \bsb{\beta}}^*) \le \breg_f(    { \bsb{\beta}},     { \bsb{\beta}}^*) + \langle \bsb{\epsilon},   (\hat { \bsb{\beta}} -   { \bsb{\beta}})/\|  \hat { \bsb{\beta}} -   { \bsb{\beta}}\|_2)\rangle \|   \hat { \bsb{\beta}} -   { \bsb{\beta}} \|_2.
\end{align*}
 We obtain a general result
 \begin{align}\label{generalnoisebasicineq}
( \breg_f -\delta \Breg_2)(\hat { \bsb{\beta}},    { \bsb{\beta}}^*)\le \breg_f(    { \bsb{\beta}},     { \bsb{\beta}}^*) + \frac{1}{2\delta} [\,\sup_{\bsb{\theta}\in \Gamma(\bsb{\beta})} \langle \bsb{\epsilon}, \bsb{\theta}  \rangle   ]^2,
  \end{align}
 for any $\delta>0$.

Based on  the regularity condition,
$$
\frac{\delta}{2}\Breg_2( \hat { \bsb{\beta}} ,    { \bsb{\beta}}^*) \le \breg_f(\hat { \bsb{\beta}},    { \bsb{\beta}}^*) -\frac{\delta}{2}\Breg_2( \hat { \bsb{\beta}} ,    { \bsb{\beta}}^*)\le \breg_f(    { \bsb{\beta}},     { \bsb{\beta}}^*) + \frac{1}{\delta} (\sup_{\bsb{\theta}\in \Gamma(\bsb{\beta})} \langle \bsb{\epsilon}, \bsb{\theta}  \rangle   )^2
$$
for any $\delta \le \mu$. Taking $\delta = \mu$ gives the desired result.
\end{proof}

\section{Algorithms for Accelerations}
\label{sec:algs}
For clarity, we give an outline of the algorithms for acceleration. 

\begin{algorithm}[h]
\caption{\small Accelerated Bregman of the second kind   \label{alg:2stab}}
\textbf{Input} { $\bsb{\beta}^{(0)} $:   initial value; $\rho_{\textrm{min}}>0$, $\alpha>0$, $M\in \mathbb N$, $\mu_0\ge 0$ (e.g., $\rho_{\textrm{min}} = 1$, $\alpha = 2$, $M=3$)
}
\begin{algorithmic}[1]
\State  $\theta_0 \in (0, 1]$, $t \leftarrow 0$, $\bm{\alpha}^{(0)}\leftarrow \bm{\beta}^{(0)}$; 
\While{not converged}
\State $\rho_{t}\leftarrow\rho_{\min}/\alpha$, $s\leftarrow 0$
\Repeat
\State $s \leftarrow s+1$
\State $\rho_t \leftarrow \alpha \rho_t$
\State  {if} $t\ge  1$, then $\theta_t\leftarrow (\sqrt{r^2 +4r} -r)/2$ with $r =(\rho_{t-1}\theta_{t-1} + \mu_0)\theta_{t-1}/{\rho_t}$
\State $\bm{\gamma}^{(t)}  \leftarrow  (1-\theta_t)\bm{\beta}^{(t)} + \theta_t\bm{\alpha}^{(t)}$
\State $\bm{\alpha}^{(t+1)} \leftarrow \mathop{\arg\min}_{\bm{\beta}}\{f(\bm{\beta}) - \bm\Delta_{\psi_0}(\bm{\beta},\bm{\gamma}^{(t)}) {+}\mu_0 \breg_{\phi}(\bm{\beta},\bm{\gamma}^{(t)})+ \theta_t\rho_t\bm\Delta_\phi(\bm{\beta},\bm{\alpha}^{(t)})\}$
\State $\bm{\beta}^{(t+1)} \leftarrow (1-\theta_t)\bm{\beta}^{(t)} + \theta_t\bm{\alpha}^{(t+1)}$

\State $R_t \leftarrow  \theta_t^2\rho_t\bm\Delta_\phi(\bm{\alpha}^{(t+1)},\bm{\alpha}^{(t)}) - \bm\Delta_{\bar\psi_0}(\bm{\beta}^{(t+1)},\bm{\gamma}^{(t)})$ \Statex \qquad\qquad\quad $   +(1-\theta_t)\bm\Delta_{\bar \psi_0}(\bm{\beta}^{(t)},\bm{\gamma}^{(t)})+ \mathbf C_{f(\cdot)-\bm\Delta_{\bar \psi_0}(\cdot,\bm\gamma^{(t)})}(\bm\alpha^{(t+1)},\bm\beta^{(t)},\theta_t)$

\Until{$R_t\ge 0$ or $s>M$}
\State if $s > M$, pick   $(\bsb{\alpha}^{(t+1)}, \bsb{\beta}^{(t+1)},\bsb{\gamma}^{(t)}, \rho_t, \theta_t)$   with the largest   $R_t/(\theta_t^2\rho_t)$
\State $t \leftarrow t+1$
\EndWhile
\State \Return $\boldsymbol{\beta}^{(t+1)}$.

 \end{algorithmic}
\end{algorithm}


\begin{algorithm}[H]
\caption{\small Accelerated Bregman of the first kind   \label{alg:1stab}}
\textbf{Input} { $\bsb{\beta}^{(0)} $:   initial value; $\rho_{\textrm{min}}>0$, $\alpha>0$, $M\in \mathbb N$, $\mu_0\ge 0$ ($\rho_{\textrm{min}} = 1$, $\alpha = 2$, $M=3$)
}
\begin{algorithmic}[1]
\State  $\theta_0 \in (0, 1]$, $t \leftarrow 0$; 
\While{not converged}
\State $\rho_{t}\leftarrow\rho_{\min}/\alpha$, $s\leftarrow 0$
\Repeat
\State $s \leftarrow s+1$
\State $\rho_t \leftarrow \alpha \rho_t$
\State  {if} $t\ge  1$, then $\theta_t\leftarrow (\sqrt{r^2 +4r} -r)/2$ with $r =(\rho_{t-1}\theta_{t-1} + \mu_0)\theta_{t-1}/{\rho_t}$
\State $\bm{\gamma}^{(t)} \leftarrow \bm{\beta}^{(t)} + \{\rho_{t-1}\theta_{t}(1 - \theta_{t-1})/(\rho_{t-1}\theta_{t-1} + \mu_0)\}(\bm{\beta}^{(t)}-\bm{\beta}^{(t-1)})$
\Statex $\qquad\qquad\qquad$ if $t\ge 1$ and $\bsb{\beta}^{(t)}$ if $t=0$
\State $\bm{\beta}^{(t+1)} \leftarrow  \mathop{\arg\min}_{\bm{\beta}}\{f(\bm{\beta}) - \bm\Delta_{\psi_0}(\bm{\beta},\bm{\gamma}^{(t)}) + \mu_0\mathbf D_2(\bm{\beta},\bm{\gamma}^{(t)})+ \rho_t\mathbf D_2(\bm{\beta},\bm{\gamma}^{(t)})\}$

\State $R_t \leftarrow (\rho_t\mathbf D_2 - \bm\Delta_{\bar\psi_0})(\bm{\beta}^{(t+1)},\bm{\gamma}^{(t)}) + (1-\theta_t)\bm\Delta_{\bar\psi_0}(\bm{\beta}^{(t)},\bm{\gamma}^{(t)}) $

\Until{$R_t\ge 0$ or $s>M$}
\State if $s > M$, pick   $(\bsb{\beta}^{(t+1)},\bsb{\gamma}^{(t)}, \rho_t, \theta_t)$   with the largest associated  $R_t/(\theta_t^2\rho_t)$
\State $t \leftarrow t+1$
\EndWhile
\State \Return $\boldsymbol{\beta}^{(t+1)}$.

 \end{algorithmic}
\end{algorithm}

\section{Experiments} \label{sec:simu}

This section performs some simulation studies to support the theoretical results.

\subsection{Computational error} \label{subsec:simu_comp}

In this part, we use mirror descent and DC programming to solve two nonconvex problems.

\subsubsection*{$\bullet$ Nonconvex mirror descent for IS divergence minimization}

In infrared astronomical satellite (IRAS) image reconstruction \citep{Cao99} and audio signal processing \citep{Fevotte2009,Lefevre2011}, the Itakura-Saito (IS) divergence (or the negative cross Burg entropy), $\mathrm{IS}(\bm a,\bm b) = \sum_{i}(a_i/b_i - \log(a_i/b_i) - 1)$, is popularly used to measure the discrepancy between the observed data and the reconstructed data.
Given $\bm X\in\mathbb R_+^{n\times p}$, and $\bm y\in\mathbb R_+^n$, the problem can be defined by
$\min f(\bm\beta) := \mathrm{IS}(\bm y, \bm X\bm\beta)~\text{s.t.\,}\bm\beta\in\mathbb R^p_+$, which is nonconvex in $\bm\beta$.
To maintain the nonnegativity constraint in updating $\bm\beta$ automatically, we develop a mirror descent algorithm. Concretely, define
$g(\bm\beta; \bm\beta^-) = f(\bm\beta) + (\rho \mathbf D_\varphi - \bm\Delta_f)(\bm\beta,\bm\beta^-)$,
where $\varphi(\bm\beta) = \sum_j\beta_j\log \beta_j - \beta_j$.
Then, minimizing $g(\bm\beta;\bm\beta^{(t)})$ with respect to $\bm\beta$ gives rise to a \textit{multiplicative} rule
\begin{equation}
\beta_j^{(t+1)} = \beta_j^{(t)}\exp\Big[ -\frac{1}{\rho}\sum_i \frac{(\bm X\bm\beta^{(t)})_i - y_i}{(\bm X\bm\beta^{(t)})_i^2} X_{ij} \Big].
\end{equation}
From Theorem \ref{th:comp_nonconvex}, the $\mathcal O(1/T)$ rate of convergence holds for the optimization error $\mathop{\avg}_{0\leq t \leq T}(2\rho\sym{\mathbf D}_\varphi  -  {\bm\Delta}_f)( \bm{\beta}^{(t+1)},\bm{\beta}^{(t)})$. To verify this, we generated a design matrix of size $1000\times 1000$ with all elements drawn from $U(0,1)$, and set $\bm y = \bm X \bm\beta^* + \bm e$ with $\beta_j^*$ chosen uniformly from the interval $(0,5)$ and $  e_i\sim \mathcal N(  0,\sigma^2)$ with $\sigma^2=10$. We fixed $1/\rho = 0.01$.
Figure \ref{fig:comp_NBEM} shows how the logarithm of optimization error converges to $-\infty$ (since $\log 0 = -\infty$) for 50 different $\bm\beta^{(0)}$ with $\beta^{(0)}_j$ randomly chosen from $U(0,1)$. Observe that all the error curves in the log-log plot are bounded by a line with slope $-1$.

\begin{figure}[h]
    \centering
    \includegraphics[width=0.6\textwidth, height=2in]{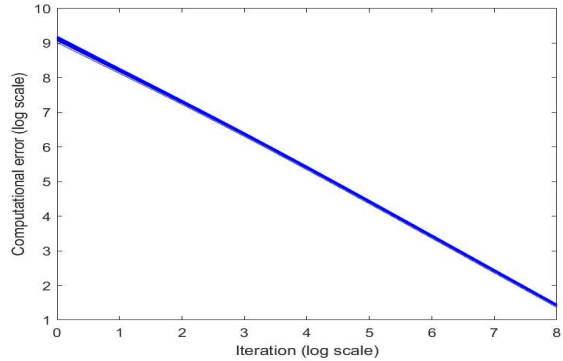}
    \caption{\small Log-log plot of optimization error v.s. number of iterations: mirror descent for IS divergence minimization with 50 random starting points. All error curves are bounded above by the dashed line which has slope $-1$.} \label{fig:comp_NBEM}
\end{figure}

\subsubsection*{$\bullet$ DC programming for capped-$\ell_1$ SVM}

High-dimensional classification with concurrent feature selection can be achieved by minimizing a composite objective function. Given $\bm X = [\bm x_1,\ldots,\bm x_n]^\top \in \mathbb R^{n \times p}$ and $\bm y\in\{-1,1\}^n$, let $l(\bm\beta) = \sum_{i=1}^n (1-y_i\bm x_i^\top\bm\beta)_+$ be the hinge loss \citep{Vapnik1995} that is nondifferentiable,
and $P(\bm\beta;\lambda) = \sum_{j=1}^p \min(\lambda|\beta_j|, \lambda^2/2 )$ be the capped-$\ell_1$ penalty \citep{Zhang2010} which is nonsmooth and nonconvex.
\cite{Ong2013} proposed an effective DC algorithm for solving $\min_{\bm\beta} f(\bm\beta) := l(\bm\beta) + P(\bm\beta;\lambda)$ based on the decomposition $P(\bm\beta;\lambda) = d_1(\bm\beta;\lambda) - d_2(\bm\beta;\lambda)$ with
\begin{equation} \label{simu:d_2}
d_1(\bm\beta;\lambda) = \lambda\|\bm\beta\|_1,~~ d_2(\bm\beta;\lambda) = \sum_{j=1}^p \max(\lambda|\beta_j| - \lambda^2/2, 0).
\end{equation}
As stated in Example \ref{ex:DC}, we can recharacterize DC as a Bregman-surrogate algorithm
{\small \begin{subequations}\label{DC-Bregman}
\begin{align}
\bm\beta^{(t+1)} &\in \mathop{\arg\min} f(\bm\beta) + \bm\Delta_{d_2}(\bm\beta,\bm\beta^{(t)}) \\
& \in \mathop{\arg\min}_{\bm\beta} \sum_{i=1}^n \max(0, 1-y_i\bm x_i^\top\bm\beta) + \lambda \sum_{j=1}^p \big(|\beta_j| - \beta_j 1_{|\beta_j^{(t)}|\ge \lambda/2}\big).
\label{DC-Bregman-b} \end{align}
\end{subequations}}\normalfont
\eqref{DC-Bregman-b} is equivalent to a linear program:
$\min_{\bm\xi,\bm\zeta,\bm\beta} \sum_{i=1}^n \xi_i + \lambda \sum_{j=1}^p \zeta_j - \lambda\sum_{j=1}^p \beta_j 1_{|\beta_j^{(t)}|\ge \lambda/2}$
s.t.\ $\xi_i \ge 1-y_i\bm x_i^\top\bm\beta$ for $i\in [n]$,
$-\zeta_j \le \beta_j \le \zeta_j$ for $j\in [p]$ and
$\xi_i, \zeta_j \ge 0$ for $i\in [n], j\in[p]$, which can be efficiently solved by standard linear programming (LP) solvers. For the convergence of the DC algorithm, a similar result can be shown for the optimization error $\avg_{0\le t\le T}(\bm\Delta_l +\back{\bm\Delta}_{d_2}+ \bm\Delta_{d_1})(\bm\beta^{(t)},\bm\beta^{(t+1)})$, following the lines of the proof of Proposition \ref{pro:LLA_comp}.

We generated $\bm X \in \mathbb R^{400\times 800}$ with each row following $\mathcal N(\bm 0,\bm\Sigma)$ and $\Sigma_{ij} = 0.5^{|i-j|}$, $\bm\beta^* = [15,10,0,\ldots,0]^\top$, and
$\bm y = \sgn(\bm X\bm\beta^* + \bm e)$ where $ e_i \sim \mathcal N(  0,\sigma^2)$ with $\sigma^2=10$. We fixed $\lambda = 1$ and ran the DC algorithm for 50 different starting points with each component randomly drawn from $U(0,1)$. The corresponding optimization error curves are plotted in Figure \ref{fig:comp_svm}, where the $\mathcal O(1/T)$ rate of convergence is impressive.

\begin{figure}[h]
    \centering
    \includegraphics[width=0.6\textwidth, height=2in]{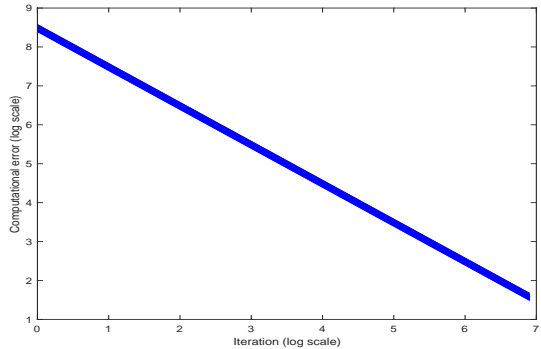}
    \caption{\small Log-log plot of optimization error v.s.\ number of iterations: DC programming for capped-$\ell_1$ SVM (50 different initial points).} \label{fig:comp_svm}
\end{figure}

\subsection{Statistical error} \label{subsec:simu_stat}

In this part, we consider two algorithms for sparse regression: LLA and iterative thresholding. The nonconvex ``hard'' penalty defined by \eqref{def:PH} is applied, which is constructed from the hard-thresholding rule via \eqref{pendef}.
Given $\bm X\in\mathbb R^{n\times p}$ and $\bm y\in\mathbb R^n$, we study the following regularized problem $\min_{\bm\beta} f(\bm\beta) :=  l(\bm\beta) + \sum_{j=1}^p P_H(\varrho\beta_j;\lambda)$,
where $l$ is the loss function and $\varrho = \|\bm X\|_2$.
The loss functions under consideration are the ordinary quadratic loss and a nonconvex loss which is resistant to gross outliers:
\begin{itemize}
  \item[(i)] $\ell_2$ loss: $l(\bm\beta) = \|\bm y-\bm X\bm\beta\|_2^2/2$;

  \item[(ii)] Tukey's biweight loss: $l(\bm\beta) = \sum_{i=1}^n \int_0^{|\bm x_i^\top\bm\beta-y_i|}\psi(t)\text{d}t$ and $\psi(t) = t[ 1-(t/c)^2]^2$ if $|t| \leq c$ and $0$ otherwise,
  where $c = 4.685 \sigma$ with $\sigma$ a robust estimate of the standard deviation of errors \citep{Hampel2005}.
\end{itemize}
In either case, we have a nonconvex optimization problem. In simulations,
the design matrix $\bm X \in \mathbb R^{n\times p}$ has i.i.d.\,rows drawn from $\mathcal N(\bm 0, \bm\Sigma)$ with $\Sigma_{ij} = 0.15^{|i-j|}$, the response is given by $\bm y = \bm X\bm\beta^*+\bm e$ with $e_i \sim N(0,\sigma^2)$ and $\sigma^2 = 10$, and the regularization parameter $\lambda$ is set to $A\sigma\sqrt{\log(ep)}$.
We set $n = 800, p = 1000, \bm\beta^* = [12,8,0,\ldots,0]^\top$ and $A = 2$.

First, we tested the statistical accuracy of LLA (cf. Theorem \ref{thm:errrate} and Proposition \ref{thm:stat_LLA}).   We generated 15 initial points $\bm\beta^{(0)}$ with each element following  $U(-a,a)$, with $a\in\{0.5,1,1.5\}$ and 5 for each. Figure \ref{fig:stat-LLA} shows how the statistical error varies as the cycles progress, with each curve representing an average over 20 implementations of the same setting. Here, the errors are plotted on a log scale for a better view of the convergence rate. Unlike Figure \ref{fig:comp_NBEM} and Figure \ref{fig:comp_svm}, the statistical errors can not reach 0 (or $-\infty$ in the log plot) due to the existence of noise. But they all achieved essentially the same order of statistical precision, which verifies Theorem \ref{thm:errrate}, and the statistical convergence of LLA was really fast.

\begin{figure}[H]
    \subfigure{\includegraphics[width=0.49\textwidth]{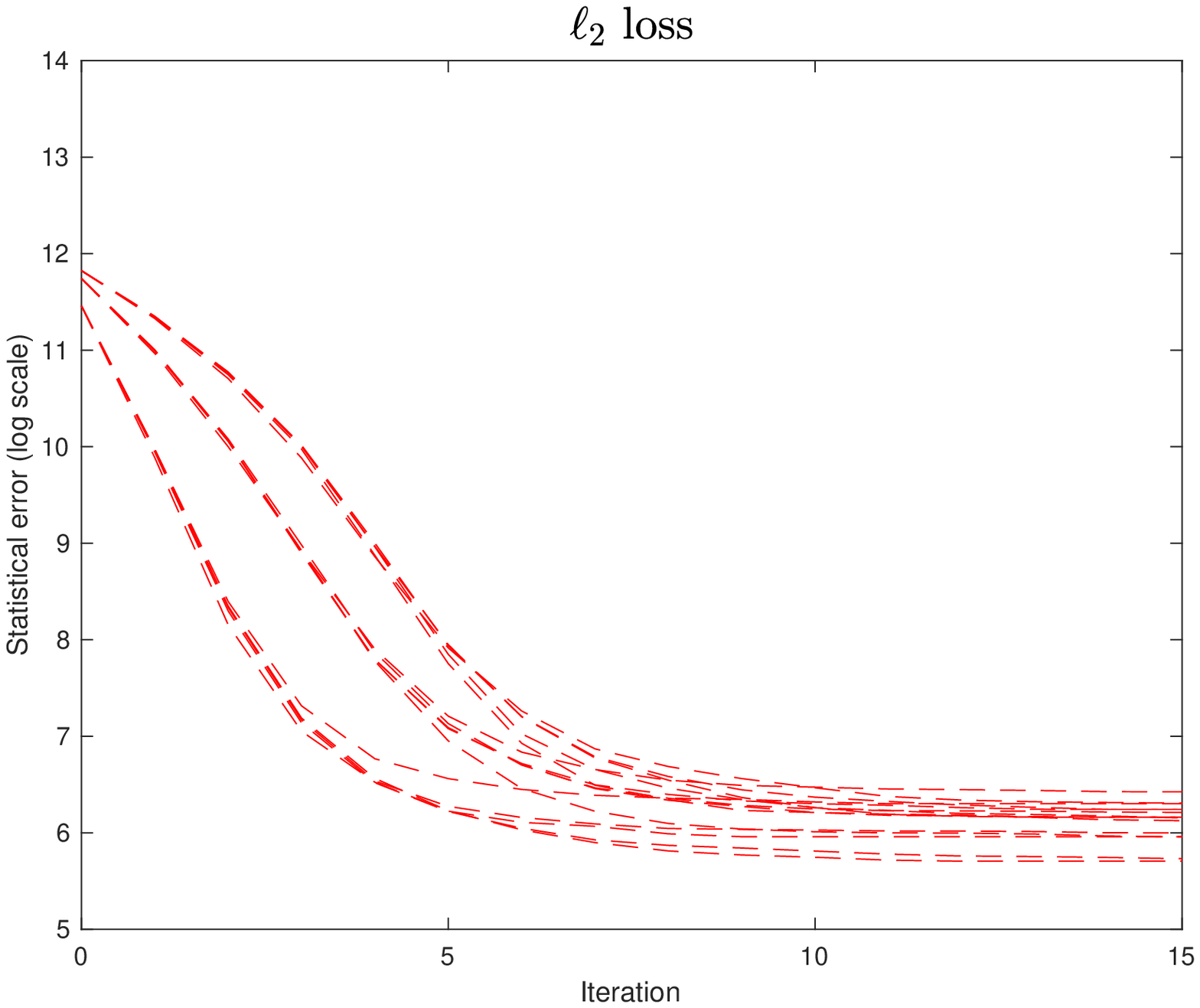}}
    \subfigure{\includegraphics[width=0.49\textwidth]{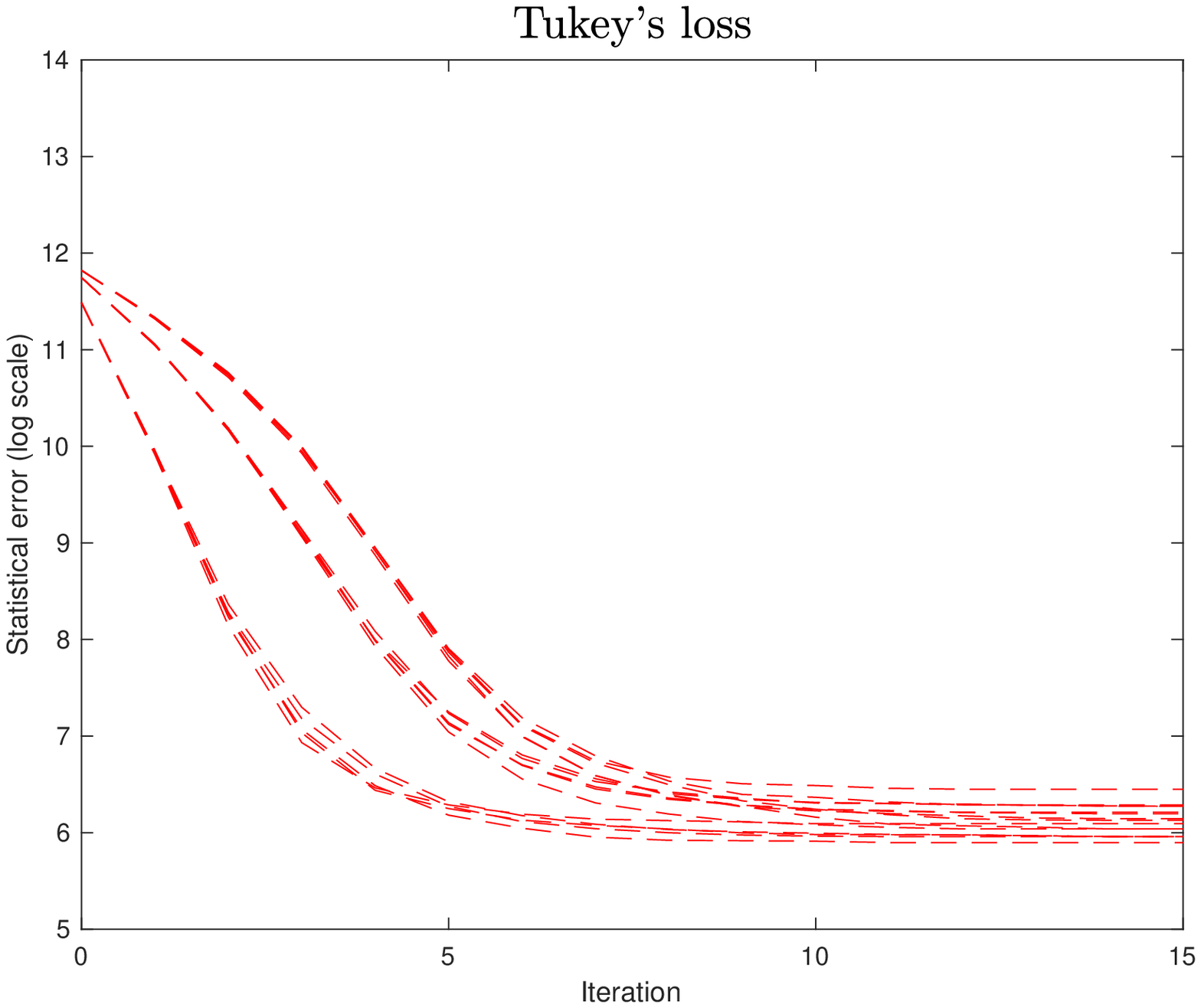}}
    \caption{\small Log plot of the statistical accuracy of LLA iterates in $P_H$-penalized sparse regression (left) and robust regression (right).}\label{fig:stat-LLA}
\end{figure}

On the other hand, the computational burden of LLA turned out to be pretty high, mainly due to the cost of solving a weighted lasso problem at each iteration. We thus turned to iterative thresholding because of its low per-iteration complexity. Figure \ref{fig:stat-Thres} shows some analogous results.
According to Figure \ref{fig:stat-Thres}, all final statistical errors were controlled within the same order of precision. The convergence process seems to conform to the bound in Theorem \ref{th:stat_bound}: when $t$ is small, $\log\bm\Delta_\psi(\bm\beta^*,\bm\beta^{(t)}) \lesssim -\log(1/\kappa)\,t + \log(\bm\Delta_\psi(\bm\beta^*,\bm\beta^{(0)}))$, and when $t$ is large, $\log\bm\Delta_\psi(\bm\beta^*,\bm\beta^{(t)})\lesssim \kappa^t\bm\Delta_\psi(\bm\beta^*,\bm\beta^{(0)}) + \log(\kappa K\lambda^2J^*/(1-\kappa))$, demonstrating an exponential decay.

\begin{figure}[h]
    \subfigure{\includegraphics[width=0.49\textwidth]{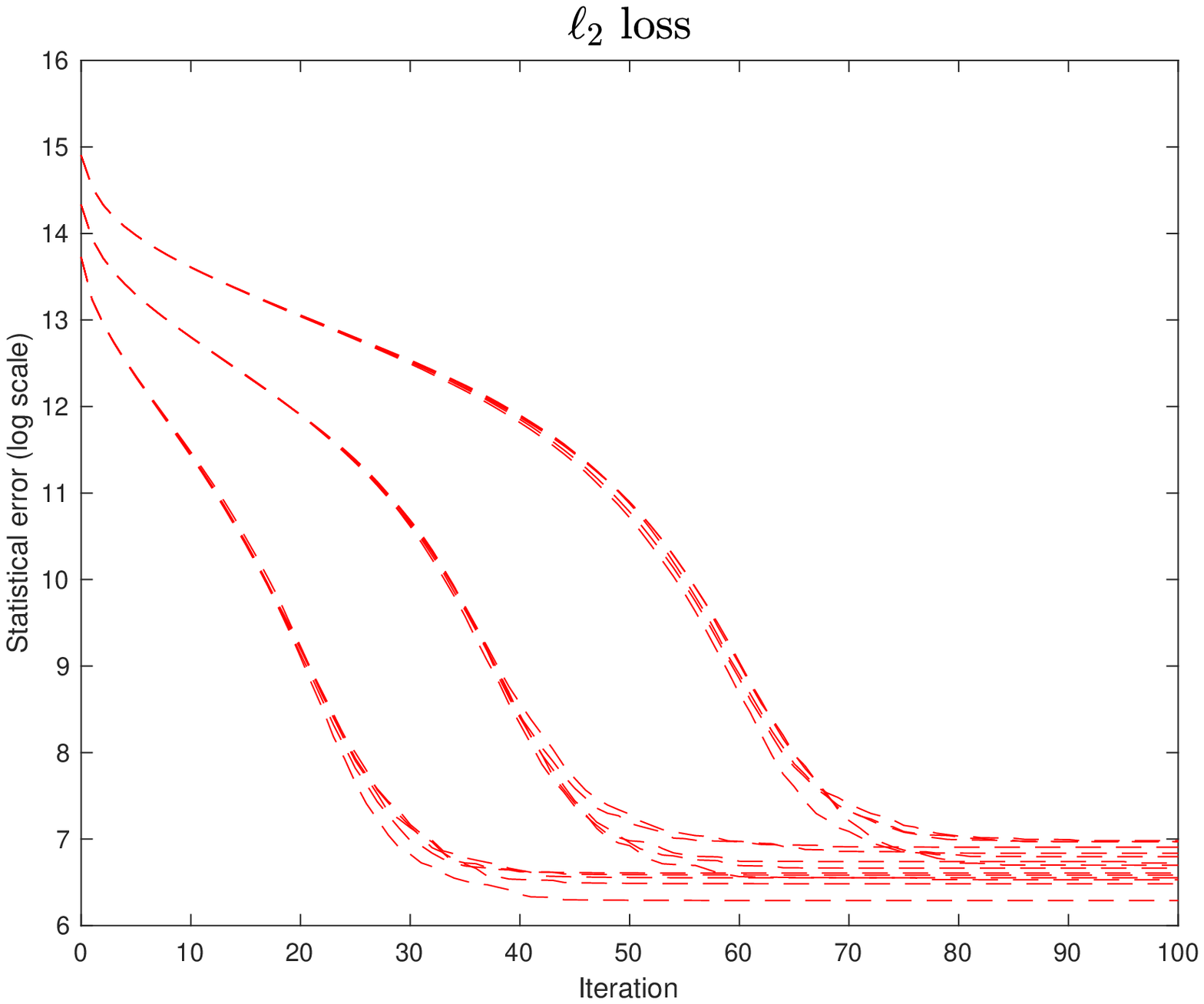}}
    \subfigure{\includegraphics[width=0.49\textwidth]{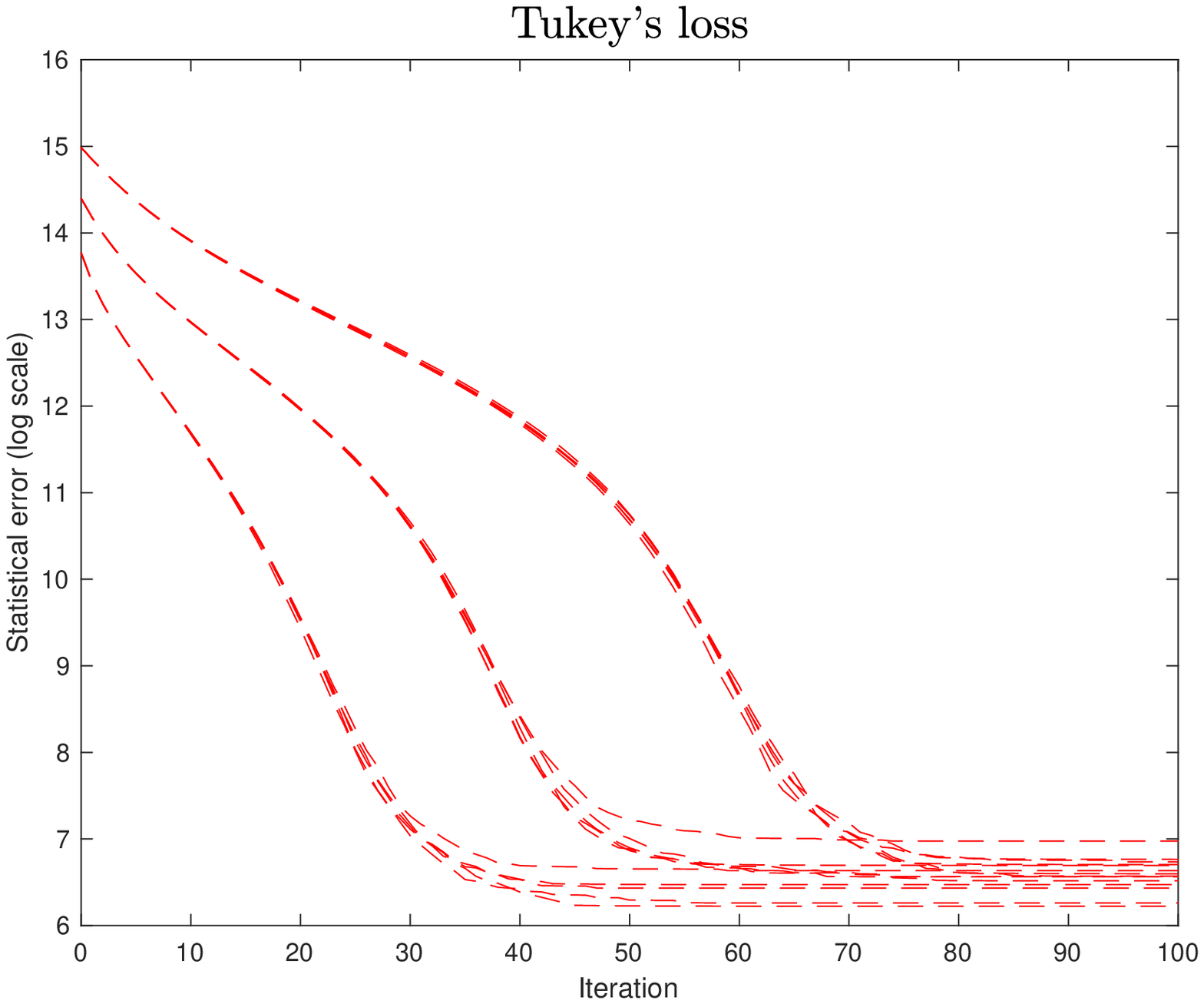}}
    \caption{\small Log plot of the statistical accuracy of iterative thresholding iterates in $P_H$-penalized sparse regression (left) and robust regression (right).}\label{fig:stat-Thres}
\end{figure}

\subsection{Accelerations} \label{subsec:simu_acc}

We test the acceleration schemes in IS divergence minimization and robust sparse regression in this subsection.

Figure \ref{fig:NBEM_acc} shows the power of applying the (second) acceleration in IS divergence minimization problem in Section \ref{subsec:simu_comp}, where we used 50 starting points with $\beta_j^{(0)} \sim U(0,1 ), 1\le j\le p$. With the acceleration, the number of iterations was brought down from 1000 to less than 50 to reach the same value of the objective function, and the overall computational time was saved by nearly 90\%.

\begin{figure}[h]
    \setcounter{subfigure}{0}
    \centering
    \includegraphics[width=0.6\textwidth, height=2in]{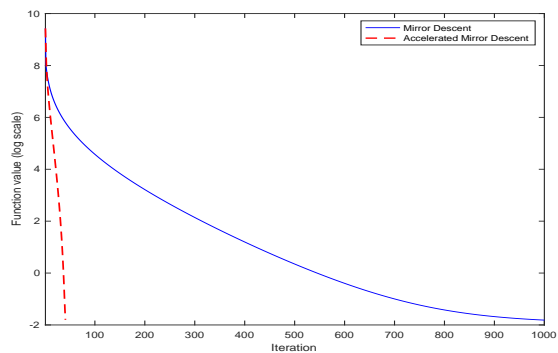}
    \caption{\small Objective function value (shown on $\log$ scale) v.s.\ number of iterations for the plain and accelerated exponentiated gradient descent algorithms in nonconvex Burg entropy optimization. }
\label{fig:NBEM_acc}
\end{figure}

Figure \ref{fig:Reg_acc} shows the convergence of statistical error when applying the (first) acceleration scheme in iterative thresholding for the $P_H$-penalized Tukey's loss minimization problem as mentioned in Section \ref{subsec:simu_stat}. The simulation setting remains the same as before and we sampled 20 initial points with $\beta_j^{(0)}\sim U(-1,1), 1\le j\le p$.
A substantial reduction in the number of iterations was achieved. Of course, the line search causes some overhead in computation. But the accelerated iterative thresholding still reduced the overall running time by more than 30\%, and obtained slightly better statistical accuracy.

\begin{figure}[h]
    \setcounter{subfigure}{0}
    \centering
    \includegraphics[width=0.6\textwidth, height=2in]{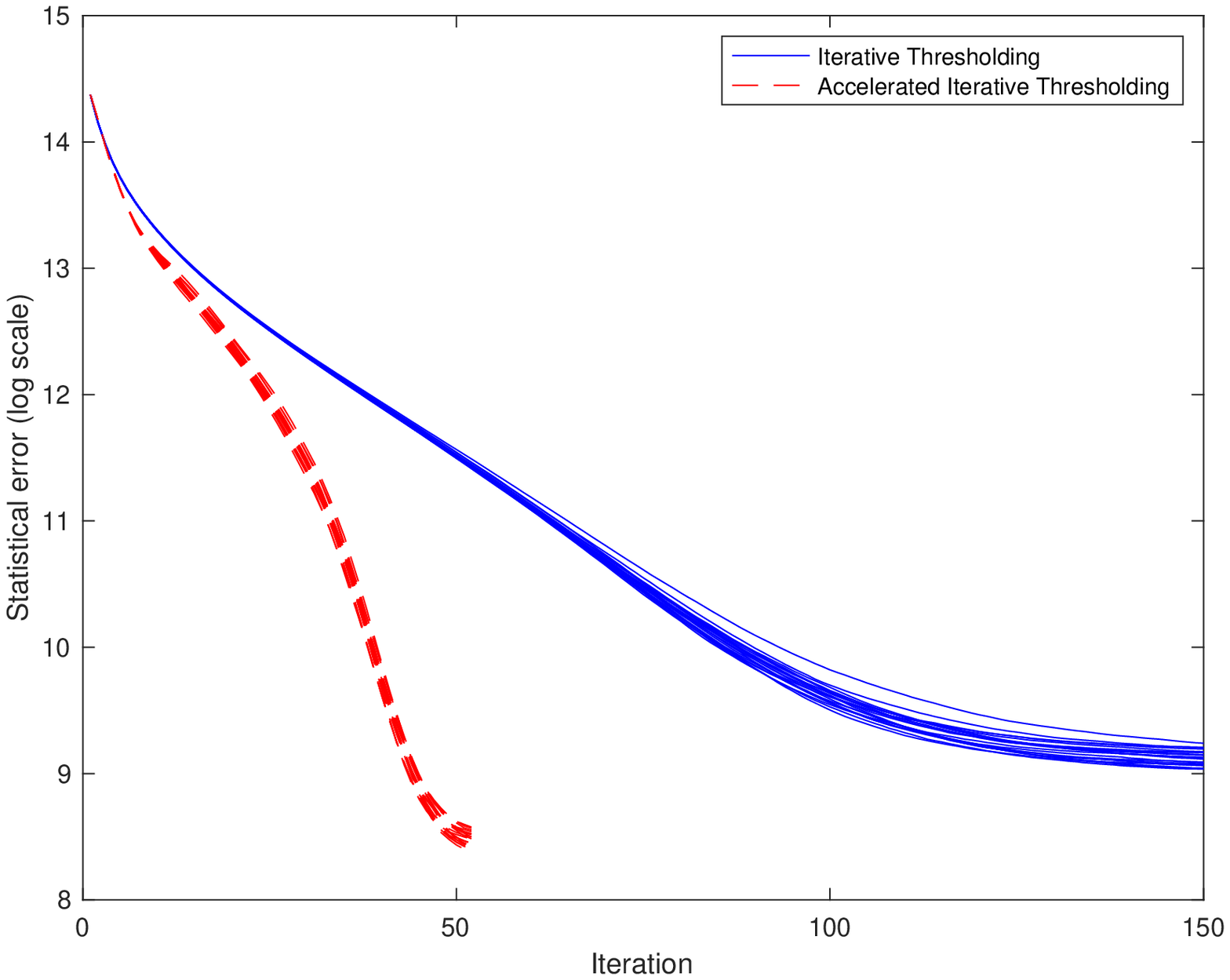}
    \caption{\small Statistical error (shown on $\log$ scale) v.s. iteration number for iterative thresholding and accelerated iterative thresholding in robust sparse regression. } 
\label{fig:Reg_acc}
\end{figure}

\bibliographystyle{imsart-number} 
\bibliography{Bregman}       

\end{document}